\documentclass[11pt,draftcls,onecolumn]{IEEEtran}
\usepackage{fancyhdr}
\usepackage[dvips, usenames]{color}
\usepackage{amsmath}
\usepackage{mathrsfs}
\usepackage{amsfonts}
\usepackage{tikz}
\usetikzlibrary{arrows,automata}
\usepackage[latin1]{inputenc}
\usepackage{verbatim}
\usepackage{amsmath}
\usepackage{mathptmx}
\usepackage{color}
\usepackage{graphics}
\usepackage{graphicx}
\usepackage{dcolumn}
\usepackage{bm}
\usepackage{amssymb}
\newtheorem{theorem}{\textbf{Theorem}}
\newtheorem{corollary}{\textbf{Corollary}}
\newtheorem{definition}{\textbf{Definition}}
\newtheorem{lemma}{\textbf{Lemma}}

\newtheorem{remark}{\textbf{Remark}}
\newtheorem{example}{\textbf{Example}}
\newtheorem{proof}{\textbf{Proof}}
\newtheorem{algorithm}{\textbf{Algorithm}}

\begin{document}
%
\title{Convergence of a $\theta$-scheme to solve the stochastic nonlinear Schr\"odinger equation with Stratonovich noise}
\author{Chuchu Chen$^{\dag}$\thanks{$^{\dag}$Institute of Computational Mathematics and
Scientific/Engineering Computing, Chinese Academy of Sciences, Beijing, PR China. ({\tt chenchuchu@lsec.cc.ac.cn})({\tt hjl@lsec.cc.ac.cn}).},
 Jialin Hong$^{\dag}$ and Andreas Prohl$^{\S}$\thanks{$^{\S}$Mathematisches Institut,
  Universit\"at T\"ubingen, Auf der Morgenstelle 10, D-72076 T\"ubingen, Germany. ({\tt prohl@na.uni-tuebingen.de}).}}
%
%

%

%

\maketitle
\begin{abstract}
We propose a $\theta$-scheme to discretize the $d$-dimensional stochastic cubic Schr\"odinger equation in Stratono\-vich sense.
{A uniform bound for the Hamiltonian of the discrete problem is obtained, which is a crucial property to verify the  convergence in probability  towards a mild solution.}
 Furthermore, based on the uniform bounds of iterates in ${\mathbb H}^2(\mathcal{O})$ for $\mathcal{O}\subset\mathbb{R}^{1}$,
 the convergence order $\frac12$ in strong local sense is obtained.


\end{abstract}
\section{Introduction}
Let ${\mathscr O} \subset {\mathbb R}^d$ be a bounded domain with $C^{2}$  boundary.
We study different discretizations for the following stochastic cubic Schr\"odinger equation with multiplicative noise of Stratonovich type ($\lambda \in \{-1,1\}$),
\begin{align}\label{sdd1}
i d\psi + \bigl(\Delta \psi  + \lambda \vert \psi \vert^2 \psi \bigr) dt &= \psi \circ dW(t) \qquad \mbox{in } {\mathscr O}_T :=\mathcal{O}\times(0,T)\,, \nonumber\\
\psi&=0 \qquad\qquad \quad\;\;  \mbox{on } \partial \mathcal{O}\times (0,T)\, ,\\
\psi(0) &= \psi_0 \qquad\qquad\quad\; \mbox{in } \mathcal{O} \,.\nonumber
\end{align}
Here, $W$ denotes a real-valued trace-class ${\bf Q}$-Wiener process.
This problem was e.g.~studied in \cite{RGBC1} to motivate the possible role of noise
to  prevent or delay collapse formation; see also \cite{DDM1} for the case $\lambda=1$.
It is due to the special type of the multiplicative noise that the mass of solutions of \eqref{sdd1} is preserved $P$-a.s.,
\begin{align}\label{L2}
\|\psi(t)\|_{\mathbb{L}^2}=\|\psi_0\|_{\mathbb{L}^2} \qquad \forall \; t\in[0,T],
  \end{align}
   which is similar to the deterministic case.
    For the deterministic cubic Schr\"odinger equation, the Hamiltonian $\mathcal{H}(\psi)=\frac12\int_{\mathcal{O}}|\nabla\psi|^2dx-\frac{\lambda}{4}
    \int_{\mathcal{O}}|\psi|^4dx$
  is another invariant quantity, which is also essential to construct a solution to this problem. In the stochastic case \eqref{sdd1}, it is no longer preserved and satisfies (see \cite{DBD0a})
  \begin{equation}\label{H}
  \mathcal{H}(\psi(t))=\mathcal{H}(\psi_0)-\Im \int_{0}^{t}\int_{\mathcal{O}}\bar{\psi}\nabla \psi  d(\nabla W(s))dx+\dfrac12  \int_{0}^{t}\int_{\mathcal{O}}
  |\psi|^{2}\sum_{\ell}|\nabla {\bf Q}^{\frac12}e_{\ell}|^{2}dxds \qquad P-a.s.
  \end{equation}
  Corresponding uniform bounds for its expectation in the case of Galerkin approximations of \eqref{sdd1} and ${\mathscr O} = {\mathbb R}^d$ then allow a compactness argument to construct a global $\mathbb{H}^1$-valued mild solution for $\lambda=-1$ in \cite{DBD0a}; and for the case $\lambda=1$ with the nonlinear term being replaced by $|\psi|^{2\sigma}\psi$, the condition for global existence is
$0<\sigma <\frac{2}{d}$.

A relevant work on the numerical analysis of (\ref{sdd1})  and ${\mathscr O} = {\mathbb R}^d$ is \cite{DBD1}, where iterates $\{ \phi_R^n;\, n \in {\mathbb N}\}$ of the temporal discretization with underlying mesh of size $\tau >0$ covering $[0,T]$ are studied,
\begin{eqnarray}\nonumber
&&i \bigl(\phi^{n+1}_R - \phi^n_R\bigr) + \tau \Delta \phi^{n+1/2}_R + \frac{\lambda \tau}{2}  \bigl(\vert \phi^{n+1}_R\vert^2 +
\vert \phi^n_R\vert^2 \bigr) \phi_R^{n+1/2} \\ \label{wio1}
&&\qquad = \theta_R(\phi_R^n) \theta_R(\phi_R^{n+1})\phi_R^{n+1/2} \Delta_n W \qquad (n \geq 0)\,, \qquad
\phi_R^0 = \psi_0\, ,
\end{eqnarray}
where $\phi_{R}^{n+\frac12}=\frac12\big(\phi_{R}^{n}+\phi_{R}^{n+1}\big)$ and $\Delta_n W=W(t_{n+1})-W(t_{n})$.
This scheme is constructed in a way that iterates preserve the ${\mathbb L}^2$-norm, i.e., ${ P}$-a.s.
$\Vert \phi^n_R\Vert_{{\mathbb L}^2} = \Vert \psi_0\Vert_{{\mathbb L}^2}$ for $n \in {\mathbb N}$. However, such a bound is not sufficient for the use of compactness methods to construct the ${\mathbb H}^1$-valued solution of (\ref{sdd1}), which requires a uniform bound for the Hamiltonian $ {\mathscr H}( \phi^n_R) := \frac{1}{2} \int_{\mathcal{O}} \vert \nabla \phi^n_R\vert^2\, {d}{ x} -
\frac{\lambda}{4} \int_{\mathscr O} \vert \phi^n_R\vert^4\, {d}{ x}$ for every finite time $T>0$, i.e.,
\begin{equation}\label{sdd2}
{\mathbb E}\bigl[\max_{0 \leq n \leq [\frac{T}{\tau}]} {\mathscr H}( \phi^n_R)\bigr] \leq C(T).
\end{equation}
Since the scheme (\ref{wio1}) with $\theta_R \equiv 1$ is not known to yield this property, a truncation concept is applied in \cite{DBD1} where e.g.~$\theta_R (\cdot) = \rho\bigl( \frac{\Vert \cdot \Vert_{{\mathbb L}^6}}{R}\bigr)$, for
some $\rho \in C_0^{\infty}\bigl((-1,1); [0,1] \bigr)$ such that $\rho \bigl\vert_{[-\frac{1}{2}, \frac{1}{2}]} \equiv 1$, and some fixed $R\in{\mathbb R}^{+}$;
in this case, the right-hand side in \eqref{sdd2} needs to be replaced by a constant $C_{R}(T)$.
By tending $\tau^{-1},R \rightarrow \infty$, it is shown in \cite[Theorem 2.2]{DBD1} that iterates  construct the mild solution of (\ref{sdd1}), where the convergence of the iterates is in probability sense.

This practical construction of the mild solution of (\ref{sdd1}) is valid for initial data $\psi_0$ having a finite Hamiltonian, and a given real-valued trace-class ${\bf Q}$-Wiener process. In \cite{DBD2},
the authors study rates of convergence of the following different time semi-discretization
\begin{equation}\label{sdd3}
i \bigl( \phi^{n+1}_R - \phi^n_R \bigr) - \tau \Delta \phi^{n+1/2}_R -  \frac{\lambda \tau}{2} \theta_R \bigl( \phi^{n}_R\bigr)
\theta_R \bigl( \phi^{n+1}_R\bigr)  \bigl(\vert \phi^{n+1}_R\vert^2 +
\vert \phi^n_R\vert^2 \bigr) \phi_R^{n+1/2} = \phi^{n}_R \Delta_n W
\end{equation}
to approximate the stochastic Schr\"odinger equation in It\^o sense
\begin{equation}
i d\psi - \big(\Delta \psi +\lambda \vert \psi \vert^{2} \psi\big)dt = \psi dW(t) \quad \text{ on }  (0,T) \times {\mathbb R}^d, \qquad \psi(0)=\psi_0,
\end{equation}
for more regular initial data  $\psi_0\in\mathbb{H}^{\frac32+s}$, $s>\max\{ \frac{d}{2},1 \}$, and a more regular ${\bf Q}$-Wiener process $W$.
The view-point to achieve this goal is different to the one above: a truncation $\theta_{R}(\cdot)$ with $R>0$ of the drift term
 is employed which hinders a (direct) bound for the Hamiltonian but allows to apply  semigroup methods for the convergence analysis of this semilinear SPDE with Lipschitz drift: for $\psi_0 \in {\mathbb H}^{\frac{3}{2} + s}$, $s > \max\{ \frac{d}{2},1\}$, the (locally) existing
 mild solution $\psi$ is approximated at a rate $\frac{1}{2}$ in the following sense,
 \begin{equation}\label{eqq1}
  \lim_{C\rightarrow\infty}
P\Big[\max_{n=0,\cdots, K_{\tau^{*}}}\|\phi^{n}-\psi(t_{n})\|_{{\mathbb H}^{s}}\geq C\tau^{\frac12}\Big]=0,
\end{equation}
 see \cite[Theorem 5.6]{DBD2}.

A further step towards constructing efficient discretizations of \eqref{sdd1} is the work \cite{L1} which uses a Lie-type time-splitting method.
This scheme amounts to solving a family of timely explicitly discretized SODEs for all ${ x} \in {\mathbb R}^d$, and a
linear PDE with random force.
Iterates $\{ \xi^n;\, n \in {\mathbb N}\}$ preserve mass, but again no uniform bounds for the Hamiltonian are
known to hold in the case  $\psi_0 \in {\mathbb H}^1$, thus leaving unclear convergence behavior towards
a solution of (\ref{sdd1}) under minimum regularity requirements. However,
some strong rates are obtained in the presence of regular data.
 The strategy to validate this result is again based on a proper truncation argument.

The main goal of this work is to propose and study a new discretization (\ref{sdd21}) of (\ref{sdd1}) which inherits a
uniform estimate for the related Hamiltonian,
\begin{equation}\label{sdd21}
i \bigl(\phi^{n+1} - \phi^n\bigr) + \tau\big(\theta \Delta \phi^{n+1}+(1-\theta)\Delta\phi^{n}\big) + \frac{\lambda \tau}{2}  \bigl(\vert \phi^{n+1}\vert^2 +
\vert \phi^n\vert^2 \bigr) \phi^{n+1/2} = \phi^{n+1/2} \Delta_n W \qquad (n \geq 0)\, .
\end{equation}
 For the case {$\theta\in[\frac12+c\sqrt{\tau}),1]$ with $c\geq c^{*}>0$}, and $\mathcal{O}\subset\mathbb{R}^{d}$ a bounded Lipschitz domain, $\lambda=-1$, and initial data $\psi_0\in L^2(\Omega;\mathbb{H}_0^1(\mathcal{O}))$, iterates $\{ \phi^n;\, n \in {\mathbb N}\}$ satisfy
 \begin{equation}\label{sdd20}
{ E}\bigl[\max_{0 \leq n \leq [\frac{T}{\tau}]} {\mathscr H}( \phi^n)\bigr] \leq  C(c^*,T)\, .\qquad
\end{equation}
In order to derive this result, we multiply \eqref{sdd21} with $\bar\phi^{n+1}-\bar\phi^{n}$, integrate in space and then take the real part of the resulting equality.
It is then obvious from the stability analysis which leads to Lemma \ref{alg2:H1} that the parameter $\theta$ has to be chosen from the range
{$[\frac12+c\sqrt{\tau},\;1]$ $(c\geq c^{*}>0)$} to generate enough numerical dissipativity to control discretization effects of the noise term.
{This uniform boundedness of the discrete Hamiltonian allows a brief and concise approach by a compactness argument which constructs a family of solution processes related to \eqref{sdd21} converging to the mild solution of \eqref{sdd1} for ${\mathcal O}\subset{\mathbb R}^{1}$; see Remark~\ref{r-two}.
No additional truncation concept is needed here --- which is a relevant tool in \cite{DBD1,DBD2} (see also (\ref{wio1}) and (\ref{sdd3})) to compensate for the lack of (\ref{sdd20}) in the case $\theta = \frac{1}{2}$; we remark that the involved truncation and discretization parameters require a proper balancing for the convergence proof in \cite{DBD1,DBD2}.}
Finally,
 Lemmas \ref{alg2:L2'} and \ref{alg2:L2} favor the choice {$\theta=\frac12+ \sqrt{\tau}$} in order to guarantee an approximate conservation of the expectation of the $\mathbb{L}^2$-norm of iterates.

%

In the second part of this work, we study pathwise approximation of the solution \eqref{sdd1}, which  requires initial data $\psi_0\in L^{8}(\Omega;\;\mathbb{H}_0^1\cap \mathbb{H}^2)$. In particular, we are interested in the concept of local rates of convergence for iterates of \eqref{sdd21}, see \cite{CP}, which is stronger than that of rates in probability given above, and requires to  deal with the discretization of the nonlinear drift term directly.
A relevant prerequisite for this purpose is to provide strong stability results for the non-truncated original problem \eqref{sdd1}, and also for the discretization \eqref{sdd21}. However, it is due to the interaction of the cubic nonlinearity with the stochastic term that we are only able to provide the corresponding uniform bounds in higher spatial norms  for $d=1$. These estimates are then essential for the error analysis, which allows to establish optimal strong convergence rates on large subsets of $\Omega$ (see Theorem \ref{thm1}).
An immediate consequence of this result is the following version of rates of convergence in probability (see Corollary \ref{probability}),
\begin{equation}\label{eqq2}
  \exists\; C>0:\qquad \lim_{\tau\rightarrow 0}P\Big[\max_{0\leq n\leq M}\|\psi(t_{n})-\phi^{n}\|_{{\mathbb L}^2}\geq C\tau^{\alpha}\Big]=0,
\end{equation}
for all $\alpha<\frac{1}{2}$. Note that $C$ is a constant which does not depend on $\alpha$ and $\tau$.

This paper is organized as follows. In section \ref{preliminaries}, some preliminaries are stated, including the notion of a mild solution of \eqref{sdd1} and some properties of the linear Schr\"odinger semigroup $\{S(t);\; t\geq 0\}$. In section \ref{sec2}, uniform bounds in higher `spatial' norms, together with the H\"older continuity in time for solutions $\{\psi(t);\;t\in[0,T]\}$ of equation \eqref{sdd1} are obtained. In section \ref{sec3}, the bound
(\ref{sdd20}) for iterates $\{\phi^{n};\;0\leq n\leq M\}$ of (\ref{sdd21}) is shown ($d\geq 1$), and uniform bounds in higher spatial norms are proven ($d=1$).
These results in sections \ref{sec2} and \ref{sec3} are used in section \ref{sec4} to establish strong rate of convergence $\frac12$ for iterates of \eqref{sdd21} in local sense, and in the probability sense
(\ref{eqq2}) for $\mathcal{O}\subset{\mathbb R}^{1}$ as a simple consequence. Some computational studies are presented in section VI which complement the theoretical results.

\section{Preliminaries}\label{preliminaries}

Throughout this work, let $W$ be a ${\bf Q}$-Wiener process defined on a given filtered probability space $(\Omega,\mathcal{F},\{\mathcal{F}_{t}\}_{0\leq t\leq T},P)$, with values
in the real-valued Hilbert space ${\mathbb U}={\mathbb L}^{2}(\mathcal{O},\mathbb{R})$. Here ${\bf Q}\in \mathcal{L}({\mathbb U})$ is a non-negative, symmetric operator with finite trace.

 Equation  (\ref{sdd1}) with $\lambda=-1$ has an equivalent It\^o form (see \cite{DBD0a})
\begin{align}\label{S_NLSE_Ito}
 & id\psi+\Delta \psi dt-(|\psi|^{2}\psi-\frac{i}{2}\psi F_{\bf Q})dt=\psi dW(t).
\end{align}
Here
 $ F_{\bf Q}(x)=\sum_{\ell\in \mathbf{N}}({\bf Q}^{\frac12}e_{\ell}(x))^{2} \text{ for } x\in \mathcal{O},$
 with $\{e_{\ell}\}_{\ell\in\mathbb{N}}$ being an orthonomal basis of ${\mathbb U}$.

To study \eqref{S_NLSE_Ito}, we introduce $\mathcal{L}_{2}({\mathbb U},\;{\mathbb H})$, the space of Hilbert-Schmidt operators from Hilbert space ${\mathbb U}$ to another Hilbert space ${\mathbb H}$, where the corresponding norm is defined by
$\|{\bf Q}^{\frac12}\|_{\mathcal{L}_{2}({\mathbb U},\;{\mathbb H})}=\Big(\sum_{\ell\in\mathbb{N}}\|{\bf Q}^{\frac12}e_{\ell}\|_{{\mathbb H}}^{2}\Big)^{\frac12}.$
In the following analysis, we always assume ${\bf Q}^{\frac12}\in\mathcal{L}_{2}({\mathbb U},\;\mathbb{H}^{3}(\mathcal{O}))$.

We recall the mild solution concept for the It\^o equation \eqref{S_NLSE_Ito} from \cite{DBD0a,DBD2}.
\begin{definition}\label{mild_sol}
  An ${\mathbb H}_{0}^{1}$-valued $\{\mathcal{F}_{t}\}_{0 \leq t \leq T}$-adapted process $\{\psi(t);\; t\in[0,T]\}$, is called a mild solution of problem  \eqref{S_NLSE_Ito} if
  for $\forall\; t\in[0,T]$ holds P-a.s.
   \begin{equation}\label{integral_form_solution}
   \psi(t)=S(t)\psi(0)-i\int_{0}^{t}S(t-r)|\psi(r)|^{2}\psi dr-\frac12\int_{0}^{t}S(t-r)\psi(r)F_{{\bf Q}}dr-i\int_{0}^{t}S(t-r)\psi(r)dW(r),
 \end{equation}
 where $S\equiv\{S(t);\; t\in\mathbb{R}\}$, with $S(t)=e^{it\Delta}$ denotes the semigroup of the solution operator of the deterministic linear differential equation
\begin{equation}\label{linear}
  i\;\frac{d\psi}{dt}+\Delta \psi=0 \qquad \text{in } \mathcal{O}_{T},\qquad \psi=0\qquad\text{on }\partial\mathcal{O}\times(0,T),\qquad \psi(0)=\psi_0\qquad \text{in }\mathcal{O}.
\end{equation}
\end{definition}

\begin{remark}
  Due to the regularity estimate given in Corollary \ref{spatial H1}, and to \cite[Proposition F.0.5, (ii)]{PR}, we also have the following representation for the mild solution of \eqref{S_NLSE_Ito}: for every $t\in[0,T]$, and all $z\in{\mathbb H}_0^1$, there holds  $P$-a.s.
  \begin{align}\label{variational}
    i\int_{{\mathcal O}}\psi(t)zdx-\int_{0}^{t}\int_{{\mathcal O}}\nabla\psi\nabla zdxds-\int_0^t \int_{{\mathcal O}}\big(|\psi|^2\psi-\frac{i}{2}\psi F_{\mathbf Q}\big)zdxds=i\int_{{\mathcal O}}\psi_0 zdx+\int_{0}^{t}\int_{\mathcal O}\psi zdW(s)dx.
  \end{align}
  We will use this form in the error analysis in section \ref{sec4}.
\end{remark}
We end this section with some useful properties of $\{S(t);\;t\geq 0\}$, which will be needed in Lemma \ref{temporal L2} and Lemma \ref{temporal H1 new} (see \cite{DBD2} for a corresponding study in the case $\mathcal{O}={\mathbb R}^{d}$).

In the following, the constant $K>0$ differs from line to line; it depends on the initial value $\psi_0$, $T$, ${\bf Q}^{\frac12}$, and $\mathcal{O}$, but not on $\tau$, $n$.

\begin{lemma}\label{lemma1}
  The semigroup $\{S(t);t\geq 0\}$ is an isometry in ${\mathbb L}^2(\mathcal{O})$, and it holds that
  \begin{align*}
   \|S(t)-Id\|_{\mathcal{L}(\mathbb{H}_0^1,\;{\mathbb L}^2)}\leq Kt^{\frac12},\label{S10}
\end{align*}
{where $K$ does not depend on $t$.}
\end{lemma}
\begin{proof}
  To show the isometry property of $S(t)$, we multiply \eqref{linear} by $\psi$, integrate in $\mathcal{O}_{T}$ and take the imaginary part. We get
  \[\|\psi(t)\|_{{\mathbb L}^2}=\|\psi_0\|_{{\mathbb L}^2},\]
  which implies that $\|S(t)\|_{\mathcal{L}({\mathbb L}^2,\;{\mathbb L}^2)}=1$.

  Next, let $\psi_0\in \mathbb{H}_0^1(\mathcal{O})$. By multiplying \eqref{linear} by $\Delta\bar\psi$, integrating in $\mathcal{O}_{T}$ and taking the imaginary part, we easily deduce $\|S(t)\|_{\mathcal{L}(\mathbb{H}_0^1,\;\mathbb{H}_0^1)}=1$.
  The assertion {\rm (i)} is equivalent to $\|\psi(t)-\psi_0\|_{{\mathbb L}^2}=\|\big(S(t)-Id\big)\psi_0\|_{{\mathbb L}^2}\leq K\|\psi_0\|_{\mathbb{H}^1}t^{\frac12}$.
  In fact, we may conclude from \eqref{linear} that \[i\int_{\mathcal{O}}\psi(t)\xi dx-i\int_{\mathcal{O}}\psi_0\xi dx
  =\int_{0}^{t}\int_{\mathcal{O}}\nabla\psi(\lambda)\nabla\xi dxd\lambda\qquad \forall\xi\in {\mathbb H}_{0}^{1}(\mathcal{O}).\]
 We choose $\xi=\bar\psi(t)$, and take the imaginary part to get
  \begin{align*}
    \frac12\Big(\|\psi(t)\|_{\mathbb{L}^2}^2-\|\psi_0\|_{\mathbb{L}^2}^2+\|\psi(t)-\psi_0\|_{{\mathbb L}^2}^2\Big)=&\Im\int_{0}^{t}\int_{\mathcal{O}}\nabla\psi(\lambda)\nabla\bar\psi(t) dx d\lambda\\
    \leq &\int_{0}^{t}\|\nabla\psi(\lambda)\|_{\mathbb{L}^2}\|\nabla\psi(t)\|_{\mathbb{L}^2}d\lambda
    \leq K\|\psi_0\|_{{\mathbb H}^1}^2t.\nonumber
  \end{align*}
  The proof of the assertion is finished.
\end{proof}

\section{Stability results in higher norms for more regular initial data}\label{sec2}
In this section, we study stability properties of solutions of \eqref{sdd1} with $\lambda=-1$.
  A formal application of It\^o's formula shows that the pathwise $\mathbb{L}^2$-norm of the solution of \eqref{sdd1} is preserved as in the deterministic case. The Hamiltonian $\mathcal{H}(\psi)$, however, is no longer preserved for \eqref{sdd1}, but one can obtain its boundedness in ${ L}^p(\Omega)$ for any finite time $T>0$; see Lemma \ref{lemma3}. For $\psi_0\in { L}^p(\Omega;\mathbb{H}_{0}^{1}\cap\mathbb{H}^2(\mathcal{O}))$ and $\mathcal{O}\subset \mathbb{R}^{1}$, we show that the solution is also $\mathbb{H}_{0}^{1}\cap\mathbb{H}^2(\mathcal{O})$-valued and that its ${ L}^p(\Omega;{L}^{\infty}(0,T;\mathbb{H}_{0}^{1}\cap\mathbb{H}^2(\mathcal{O})))$-norm is bounded; see Lemma \ref{spatial E max}. Those bounds in strong (spatial) norms for the mild solution of \eqref{sdd1} may be used to prove H\"older regularity with respect to time in strong norms;
 see Lemma \ref{temporal L2} and \ref{temporal H1 new}. They are useful in section \ref{sec4} to establish
  rates of convergence for the $\theta$-scheme \eqref{sdd21}.

In the following lemmas, the application of It\^o formula is formal; the argument can, however, be made rigorous by using a truncated version of \eqref{S_NLSE_Ito}, and passing to the limit after It\^o's formula has been applied; we refer to \cite{DBD0a} for a corresponding argumentation.

\begin{lemma}\label{lemma3}
Let $\mathcal{O}\subset\mathbb{R}^d$ be a bounded Lipschitz domain, $\mathcal{H}(\psi_0)\in L^p(\Omega)$ for some $p\geq 1$  with $\psi_0 = 0$ on $\partial {\mathscr O}$, and  {\color{black} and $\psi$ be a mild solution of \eqref{S_NLSE_Ito}}. Then there exists a constant
 {$K\equiv K\big(p,T\big)>0$} such that
 \begin{align*}
   & {\rm (i)}\quad \sup_{0\leq t\leq T}E\Big[\big(\mathcal{H}(\psi(t))\big)^{p}\Big]\leq K,\\
   & {\rm (ii)}\quad E\Big[\sup_{0\leq t\leq T}\big(\mathcal{H}(\psi(t))\big)^{p}\Big]\leq K.
 \end{align*}
\end{lemma}
\begin{proof}
{\em Step 1: Case p=1.}
  Applying expectation to \eqref{H}, we have
  \begin{align*}
    E\Big(\mathcal{H}(\psi(t))\Big)&=E\Big(\mathcal{H}(\psi_0)\Big)+\dfrac12E\Big(\int_{0}^{t}\int_{\mathcal{O}}|\psi|^{2}\sum_{\ell}|\nabla {\bf Q}^{\frac12}e_{\ell}|^{2}dxds\Big).
  \end{align*}
  Since
  \begin{align}\label{w1}
  \int_{\mathcal{O}}|\psi|^{2}\sum_{\ell}|\nabla {\bf Q}^{\frac12}e_{\ell}|^{2}dx\leq
  \|\psi\|_{\mathbb{L}^4}^2\sum_{\ell}\|\nabla {\bf Q}^{\frac12}e_{\ell}\|_{\mathbb{L}^4}^{2}
  &\leq \frac14\|\psi\|_{\mathbb{L}^4}^{4}+\|\nabla {\bf Q}^{\frac12}\|_{\mathcal{L}_{2}({\mathbb U},\;\mathbb{L}^4)}^4
 \end{align}
 we get the following estimate for $E(\mathcal{H}(\psi(t)))$,
 \begin{align*}
   E\Big(\mathcal{H}(\psi(t))\Big)
   &\leq E\Big(\mathcal{H}(\psi_0)\Big)+Kt\|\nabla {\bf Q}^{\frac12}\|_{\mathcal{L}_{2}({\mathbb U},\;\mathbb{L}^4)}^4+E\int_{0}^{t}\|\psi\|_{\mathbb{L}^4}^4ds.
 \end{align*}
  From the definition of the Hamiltonian $\mathcal{H}(\psi)$, we know that $\|\psi\|_{\mathbb{L}^4}^4\leq 4\mathcal{H}(\psi)$, which leads to
  \begin{align*}
     \sup_{0\leq t\leq T}E\Big(\mathcal{H}(\psi(t))\Big)\leq  K+KE\int_{0}^{T}\mathcal{H}(\psi(s))ds.
  \end{align*}
  Gronwall's Lemma then implies the assertion ${\rm(i)}$ of the lemma.

  To show assertion ${\rm(ii)}$ for $p=1$, we take the supremum over $t\in[0,T]$ in \eqref{H} before taking the expectation. If compared to assertion ${\rm(i)}$, the main difference is the appearance of the supremum of a stochastic integral, whose expectation can be estimated by the Burkholder-Davis-Gundy inequality:
  \begin{align}\label{stochatic}
   & E\Big[\sup_{0\leq t\leq T}\Big(-\Im\int_{0}^{t}\int_{\mathcal{O}}\bar{\psi}\nabla \psi  d(\nabla W(s))dx\Big)\Big]\nonumber\\
   & \leq KE\Big[\Big(\int_{0}^{T}\|\psi\|_{\mathbb{L}^4}^2\|\nabla\psi\|_{\mathbb{L}^2}^2\|\nabla {\bf Q}^{\frac12}\|_{\mathcal{L}_{2}({\mathbb U},\;\mathbb{L}^4)}^2ds\Big)^{\frac12}\Big]\nonumber\\
   & \leq KE\Big[\sup_{0\leq t\leq T}\|\nabla\psi(t)\|_{\mathbb{L}^2}\Big(\int_{0}^{T}\|\psi\|_{\mathbb{L}^4}^2\|\nabla {\bf Q}^{\frac12}\|_{\mathcal{L}_{2}({\mathbb U},\;\mathbb{L}^4)}^2ds\Big)^{\frac12}\Big]\nonumber\\
   & \leq \frac18 E\Big[\sup_{0\leq t\leq T}\|\nabla\psi(t)\|_{\mathbb{L}^2}^2\Big]+KE\int_{0}^{T}\Big(\|\psi\|_{\mathbb{L}^4}^4+\|\nabla {\bf Q}^{\frac12}\|_{\mathcal{L}_{2}({\mathbb U},\;\mathbb{L}^4)}^4\Big)ds\nonumber\\
   & \leq 4E\Big[\sup_{0\leq t\leq T}\mathcal{H}(\psi(t))\Big]+KE\int_{0}^{T}\Big(\mathcal{H}(\psi(s))+\|\nabla {\bf Q}^{\frac12}\|_{\mathcal{L}_{2}({\mathbb U},\;\mathbb{L}^4)}^4\Big)ds,
  \end{align}
  where in the last line we use  $\|\nabla\psi\|_{\mathbb{L}^2}^2\leq 2\mathcal{H}(\psi)$ and $\|\psi\|_{\mathbb{L}^4}^4\leq 4\mathcal{H}(\psi)$.
Then proceeding as in the proof of assertion ${\rm(i)}$, we can absorb the first term on the left-hand side, and use Gronwall's lemma.

{\em Step 2: $p\geq2$.}
  We apply It\^o's formula to  $\Big(\mathcal{H}(\psi)\Big)^{p}$, where $\mathcal{H}(\psi(t))$ satisfies \eqref{H}.
  \begin{align}\label{5}
  \Big(\mathcal{H}(\psi(t))\Big)^{p}&=\Big(\mathcal{H}(\psi_0)\Big)^{p}+\dfrac12 \int_{0}^{t}p\Big(\mathcal{H}(\psi)\Big)^{p-1}\int_{\mathcal{O}}|\psi|^{2}\sum_{\ell}|\nabla {\bf Q}^{\frac12}e_{\ell}|^{2}dxds\nonumber\\
  &\quad+\dfrac12 \int_{0}^{t}p(p-1)\Big(\mathcal{H}(\psi)\Big)^{p-2}\sum_{\ell}\Big(\Im\int_{\mathcal{O}}\bar{\psi}\nabla\psi\nabla {\bf Q}^{\frac12}e_{\ell}dx\Big)^{2}ds\nonumber\\
  &\quad +\dfrac12 \int_{0}^{t}p\Big(\mathcal{H}(\psi)\Big)^{p-1}\int_{\mathcal{O}}\bar{\psi}\nabla\psi d(\nabla W(s))dx.
  \end{align}
  Since the last term on the right-hand side vanishes after applying expectation, there remains to estimate the term
  \begin{align}\label{w2}
  \sum_{\ell}\Big(\int_{\mathcal{O}}\bar{\psi}\nabla\psi\nabla {\bf Q}^{\frac12}e_{\ell}dx\Big)^2\leq \|\psi\|_{\mathbb{L}^{4}}^2\|\nabla {\bf Q}^{\frac12}\|_{\mathcal{L}_{2}({\mathbb U},\;\mathbb{L}^4)}^2\|\nabla\psi\|_{\mathbb{L}^{2}}^2\leq K\Big(\mathcal{H}(\psi)\Big)^2+\|\nabla {\bf Q}^{\frac12}\|_{\mathcal{L}_{2}({\mathbb U},\;\mathbb{L}^4)}^8.\end{align}
  Because of \eqref{w1}, \eqref{w2}, and H\"older's inequality, we have
\begin{align*}
  \sup_{0\leq t\leq T}E\Big(\mathcal{H}(\psi(t))\Big)^p\leq K+KE\int_{0}^{T}\Big(\mathcal{H}(\psi(s))\Big)^pds.
\end{align*}
  We may now apply Gronwall's lemma to obtain the estimate ${\rm(i)}$.

  The assertion ${\rm(ii)}$ for $p\geq 2$ now uses arguments similar to \eqref{stochatic}, so we skip the details here.
\end{proof}
{
\begin{remark}
  In \cite[Theorem 4.6]{DBD0a}, a uniform bound for the Hamiltonian is used to
  construct a global unique solution with continuous ${\mathbb H}^1({\mathbb R}^d)$-valued paths for equation \eqref{sdd1} with $\lambda=-1$ or $d=1$.
  To accomplish this result, the unique local mild solution is constructed by a contraction argument, which is then shown to be global by a bound for the Hamiltonian.
  We can follow the same strategy in \cite{DBD0a} to construct the global unique  mild solution with continuous ${\mathbb H}^1_0({\mathcal O})$-valued paths in the case of a bounded Lipschitz domain ${\mathcal O} \subset {\mathbb R}^1$. It is an open problem to prove existence and uniqueness of a continuous solution in the case of a bounded domain in higher
  dimensions.
\end{remark}
}

\begin{corollary}\label{spatial H1}
Let $p\geq 1$, $\mathcal{O}\subset\mathbb{R}^d$ be a bounded Lipschitz domain, $E\big(\mathcal{H}(\psi_0)\big)^{p}< \infty$  such that $\psi_0 = 0$ on $\partial {\mathscr O}$, and $\psi$ be a mild solution. There exists a constant
{ $K\equiv K\big(p,T\big)>0$} such that
\begin{align*}
  &{\rm(i)}\quad\sup_{0\leq t\leq T}\Big(E\|\nabla\psi(t)\|_{\mathbb{L}^{2}}^{2p}+E\|\psi(t)\|_{\mathbb{L}^{4}}^{4p}\Big)\leq K,\\
  &{\rm(ii)}\quad E\Big[\sup_{0\leq t\leq T}\Big(\|\nabla\psi(t)\|_{\mathbb{L}^{2}}^{2p}+\|\psi(t)\|_{\mathbb{L}^{4}}^{4p}\Big)\Big]\leq K.
\end{align*}
\end{corollary}

In order to verify improved stability properties for the solution of \eqref{sdd1}, we have to restrict to bounded domains $\mathcal{O}\subset \mathbb{R}^1$; the technical reason for this restriction is discussed in Remark \ref{re1} below.
\begin{lemma}\label{spatial H2 new}
  Let $\mathcal{O}\subset \mathbb{R}^1$, and suppose that $\psi_0\in L^{2p}(\Omega;{\mathbb H}_{0}^{1}\cap{\mathbb{H}^{2}}(\mathcal{O}))$ for some $p\geq 1$. Then there exists a constant {$K=K (p,T)>0$} such that
  \begin{equation}\label{a1}
    \sup_{0\leq t\leq T}E\Big(\|\psi(t)\|_{\mathbb{H}^{2}}^{2p}\Big)\leq K.
  \end{equation}
\end{lemma}
\begin{proof}
To simplify notations, we present the proof of \eqref{a1} for the case $p=1$.
  We formally apply It\^o's formula to the function $f(\psi(\cdot))$, where  \[f(\psi)=\int_{\mathcal{O}}|(Id-\Delta)\psi|^{2}dx+\Re\int_{\mathcal{O}}\big((Id-\Delta)\bar{\psi}\big)|\psi|^{2}\psi dx,\]
  since for the leading term we have $\|\psi\|_{{\mathbb H}^2}^2\leq \|(Id-\Delta)\psi\|_{{\mathbb L}^2}^2\leq 2\|\psi\|_{{\mathbb H}^2}^2$, i.e., its square-root is equivalent to the norm ${\mathbb H}^1_0 \cap {\mathbb H}^2$.
  We use \eqref{S_NLSE_Ito} to get
  \begin{align}\label{fpsi}
    f(\psi(t))=&f(\psi_0)+\int_{0}^{t}Df(\psi)\Big(i\Delta \psi-i|\psi|^2\psi-\frac12\psi F_{{\bf Q}}\Big)ds+\frac12\int_{0}^{t}\text{Tr}\Big[D^2f(\psi)(-i\psi {\bf Q}^{\frac12})(-i\psi {\bf Q}^{\frac12})^{*}\Big]ds\nonumber\\
    &+\int_{0}^{t}Df(\psi)(-i\psi dW(s))\nonumber\\
    =:&f(\psi_0)+I+II+III,
  \end{align}
  with the first and second order derivatives
  \begin{align*}
    Df(\psi)(u)=&2\Re\int_{\mathcal{O}}\big((Id-\Delta)\bar{\psi}\big)\big((Id-\Delta)u\big)dx
    +\Re\int_{\mathcal{O}}\big((Id-\Delta)\bar{\psi}\big)
    \psi(\bar{\psi}u+\psi\bar{u})dx\\
    &+\Re\int_{\mathcal{O}}\big((Id-\Delta)\bar{\psi}\big)|\psi|^{2}udx
    +\Re\int_{\mathcal{O}}\big((Id-\Delta)(|\psi|^{2}\psi)\big)\bar{u}dx \quad \forall u\in {\mathbb C}_{0}^{\infty}(\mathcal{O}),
  \end{align*}
  and
  \begin{align*}
    D^{2}f(\psi)(u,v)=&2\Re\int_{\mathcal{O}}\big((Id-\Delta)\bar{u}\big)\big((Id-\Delta)v\big)dx
    +2\Re\int_{\mathcal{O}}\big((Id-\Delta)\bar{\psi}\big)\psi\Re (\bar{u}v)dx\\
    &+2\Re\int_{\mathcal{O}}\big((Id-\Delta)\bar{\psi}\big)u\Re(\bar{\psi}v)dx
    +2\Re\int_{\mathcal{O}}\big((Id-\Delta)\bar{u}\big)\psi\Re(\bar{\psi}v)dx\\
    &+2\Re\int_{\mathcal{O}}\big((Id-\Delta)\bar{\psi}\big)\Re(\bar{\psi}u)vdx
    +\Re\int_{\mathcal{O}}\big((Id-\Delta)\bar{u}\big)|\psi|^{2}vdx\\
    &+\Re\int_{\mathcal{O}}\big((Id-\Delta)\bar{v}\big)|\psi|^{2}udx
    +2\Re\int_{\mathcal{O}}\big((Id-\Delta)\bar{v}\big)\Re(\bar{\psi}u)\psi dx\quad \forall u,v \in {\mathbb C}_{0}^{\infty}(\mathcal{O}).
  \end{align*}
   For the term $f(\psi_0)$, we use the continuous embedding $\mathbb{H}^1\hookrightarrow \mathbb{L}^6$,
\[ E(f(\psi_0))\leq 2E\|\psi_0\|_{\mathbb{H}^2}^2+KE\big(\|\psi_0\|_{\mathbb{H}^2}\|\psi_0\|_{\mathbb{L}^6}^3\big)\leq KE \|\psi_0\|_{\mathbb{H}^2}^2 +KE\|\psi_0\|_{\mathbb{H}^1}^6\leq K.\]
 The term $I$ is the most difficult one: by the expression for $Df(\psi)$ above, we may represent it in the following form.
  \begin{align*}
    I=&2\int_{0}^{t}\Re\int_{\mathcal{O}}\big((Id-\Delta)\bar{\psi}\big) \Big((Id-\Delta)(i\Delta\psi-i|\psi|^{2}\psi-\frac12\psi F_{{\bf Q}})\Big)dxds\nonumber\\
    &+\int_{0}^{t}\Re\int_{\mathcal{O}}\big((Id-\Delta)\bar{\psi}\big)\psi\Big[\bar{\psi}(i\Delta\psi-i|\psi|^{2}\psi-\frac12 \psi F_{{\bf Q}})
    +\psi(-i\Delta\bar{\psi}+i|\psi|^{2}\bar{\psi}-\frac12 \bar{\psi}F_{{\bf Q}})\Big]dxds\nonumber\\
    &+\int_{0}^{t}\Re\int_{\mathcal{O}}\big((Id-\Delta)\bar{\psi}\big)|\psi|^{2}(i\Delta\psi-i|\psi|^{2}\psi-\frac12 \psi F_{{\bf Q}})dxds\nonumber\\
    &+\int_{0}^{t}\Re\int_{\mathcal{O}}\big((Id-\Delta)(|\psi|^{2}\psi)\big)(-i\Delta\bar{\psi}+i|\psi|^{2}\bar{\psi}-\frac12 \bar{\psi}F_{{\bf Q}})dxds\\
    =:&I^1+I^2+I^3+I^4.
  \end{align*}
We treat terms $I^{1}$, $I^{2}$ and $I^{4}$ together, for they have troublesome terms which cancel each other.
For this purpose, we first consider terms $I^{1}$, $I^{4}$ and $I^{2}$ independently. For the first term in $I$, we compute
\begin{align*}
 I^{1}=&-2\int_{0}^{t}\Re\int_{\mathcal{O}}i\big((Id-\Delta)\bar{\psi}\big)\big((Id-\Delta)(|\psi|^{2}\psi)\big)dxds\\
 & -\int_{0}^{t}\Re\int_{\mathcal{O}}\big((Id-\Delta)\bar{\psi}\big)\big((Id-\Delta)(\psi F_{{\bf Q}})\big)dxds\\
  =:&I_{a}^{1}+I_{b}^{1}.
\end{align*}
We conclude that
\[E(I_{b}^{1})\leq E\int_{0}^{t}\|\psi\|_{\mathbb{H}^{2}}^{2}\|F_{{\bf Q}}\|_{\mathbb{H}^{2}}ds\leq K E\int_{0}^{t}\|\psi\|_{\mathbb{H}^{2}}^{2}ds.\]
By $\Re\int_{\mathcal{O}}i\big((Id-\Delta)|\psi|^2\psi\big)|\psi|^2\psi dx=0$, we can rewrite the term $I^4$ in the following two parts,
\begin{align*}
  I^{4}=&\int_{0}^{t}\Re\int_{\mathcal{O}}(-i)\big((Id-\Delta)(|\psi|^{2}\psi)\big)\Delta\bar{\psi}dxds
  +\int_{0}^{t}\Re\int_{\mathcal{O}}\big((Id-\Delta)(|\psi|^{2}\psi)\big)(i|\psi|^{2}\bar{\psi}-\frac12 \bar{\psi}F_{{\bf Q}})dxds\\
  =&\int_{0}^{t}\Re\int_{\mathcal{O}}(-i)\big((Id-\Delta)(|\psi|^{2}\psi)\big)\Delta\bar{\psi}dxds
  -\frac12\int_{0}^{t}\Re\int_{\mathcal{O}}\big((Id-\Delta)(|\psi|^{2}\psi)\big)(\bar{\psi}F_{{\bf Q}})dxds\\
  =:&I_{a}^{4}+I_{b}^{4}.
\end{align*}
Summing the terms $\frac12 I_{a}^{1}$ and $I_{a}^{4}$ leads to
\begin{align*}
 & -\int_{0}^{t}\Re\int_{\mathcal{O}}i\big((Id-\Delta)\bar{\psi}\big)\big((Id-\Delta)(|\psi|^{2}\psi)\big)dxds
-\int_{0}^{t}\Re\int_{\mathcal{O}}i\big((Id-\Delta)(|\psi|^{2}\psi)\big)\Delta\bar{\psi}dxds\\
&=-\int_{0}^{t}\Re\int_{\mathcal{O}}i\bar{\psi}\big((Id-\Delta)(|\psi|^{2}\psi)\big)dxds\\
&=:I_{a}^{14}.
\end{align*}
This term and the term $I_{b}^{4}$ can be bounded by integration by parts, using the embedding $\mathbb{H}^1\hookrightarrow \mathbb{L}^6$, and Corollary \ref{spatial H1}, that is
\begin{align*}
E(I_{a}^{14}+I_{b}^{4})
&=-\frac12\int_{0}^{t}\Re\int_{\mathcal{O}}(|\psi|^{2}\psi)\big((Id-\Delta)(\bar{\psi}F_{{\bf Q}})\big)dxds
-\int_{0}^{t}\Re\int_{\mathcal{O}}i\big((Id-\Delta)\bar{\psi}\big)(|\psi|^{2}\psi)dxds\\
&\leq KE\int_{0}^{t}\|\psi F_{{\bf Q}}\|_{\mathbb{H}^2}\|\psi\|_{\mathbb{L}^6}^3ds+KE\int_{0}^{t}\|\psi\|_{\mathbb{H}^2}\|\psi\|_{\mathbb{L}^6}^3ds\\
&\leq KE\int_{0}^{t}(\|\psi\|_{\mathbb{H}^1}^{6}+\|\psi\|_{\mathbb{H}^2}^{2})ds\\
&\leq K+KE\int_{0}^{t}\|\psi\|_{\mathbb{H}^2}^{2}ds.
\end{align*}
Next, we consider the term $I^2$ and use the identity $a\bar b+\bar{a} b=2\Re(a\bar b)$ for $a,b\in\mathbb{C}$ to rewrite its part
\begin{align*}
 & \bar{\psi}(i\Delta\psi-i|\psi|^{2}\psi-\frac12 \psi F_{{\bf Q}})
    +\psi(-i\Delta\bar{\psi}+i|\psi|^{2}\bar{\psi}-\frac12 \bar{\psi}F_{{\bf Q}})\\
   & =i\bar{\psi}(\Delta \psi)-i\psi(\Delta\bar{\psi})+2\Re\big(\bar{\psi}(-i|\psi|^{2}\psi-\frac12 \psi F_{{\bf Q}}))\big)\\
   & =i\bar{\psi}(\Delta \psi)-i\psi(\Delta\bar{\psi})-|\psi|^2F_{{\bf Q}}
\end{align*}
Then the term $I^{2}$ equals to
\begin{align*}
  I^{2}=&
  -\int_{0}^{t}\Re\int_{\mathcal{O}}\big((Id-\Delta)\bar{\psi}\big)\psi|\psi|^2F_{{\bf Q}}dxds
  +\int_{0}^{t}\Re\int_{\mathcal{O}}i\big((Id-\Delta)\bar{\psi}\big)|\psi|^{2}\Delta \psi dxds\\
& +\int_{0}^{t}\Re\int_{\mathcal{O}}(-i)|\psi|^{2}\psi\Delta\bar{\psi}dxds
  -\int_{0}^{t}\Re\int_{\mathcal{O}}(-i)(\Delta\bar{\psi}\psi)^{2}dxds\\
  =:&I_{b}^{2}+I_{a}^{2},
\end{align*}
where $I_{a}^{2}=-\int_{0}^{t}\Re\int_{\mathcal{O}}(-i)(\Delta\bar{\psi}\psi)^{2}dxds$, while $I_{b}^{2}$ denotes the remainder terms in $I^2$.

We rewrite the term $\frac12 I_a^1$ in the form
\begin{align*}
 \frac12 I_a^1=&-2\int_{0}^{t}\Re\int_{\mathcal{O}}i\nabla\bar{\psi}\nabla(|\psi|^{2}\psi)dxds-\int_{0}^{t}\Re\int_{\mathcal{O}}i\Delta\bar{\psi}
  \Delta\big(|\psi|^2\psi\big)dx.
\end{align*}
We insert the identity $\Delta(|a|^2a)=2\Delta a|a|^2+4|\nabla a|^2a+2(\nabla a)^2\bar{a}+(a)^2\Delta\bar{a}$, for a complex-valued function $a(x)\in\mathbb{C}$ into the second integral in the above equation, add the terms $\frac12 I_{a}^{1}$ and $I_{a}^{2}$ to get
\begin{align*}
  \frac12 I_a^1+I_a^2
  =-2\int_{0}^{t}\Re\int_{\mathcal{O}}i\nabla\bar{\psi}\nabla(|\psi|^{2}\psi)dxds
  -2\int_{0}^{t}\Re\int_{\mathcal{O}}i\Big(\bar{\psi}\Delta\bar{\psi}(\nabla\psi)^{2}+2\psi\Delta\bar{\psi}|\nabla \psi|^{2}\Big)dxds.
\end{align*}
To estimate this term, we use integration by parts, H\"older inequality, the embedding $\mathbb{H}^1\hookrightarrow \mathbb{L}^{\infty}$ for $\mathcal{O}\subset\mathbb{R}^1$
and interpolation of $\mathbb{L}^4$ between $\mathbb{L}^2$ and $\mathbb{H}^1$,
\begin{align}
E(\frac12  I_a^1+I_a^2)&=-2 E\int_{0}^{t}\Re\int_{\mathcal{O}}i\nabla\bar{\psi}\nabla(|\psi|^{2}\psi)dxds
-2\int_{0}^{t}\Re\int_{\mathcal{O}}i\Big(\bar{\psi}(\nabla\psi)^{2}\Delta\bar{\psi}+2\psi|\nabla \psi|^{2}\Delta\bar{\psi}\Big)dxds\label{a4}\\
&\leq KE\int_{0}^{t}||\psi||_{\mathbb{H}^{1}}^{4}ds+KE\int_{0}^{t}\|\psi\|_{\mathbb{L}^{\infty}}^{2}\|\nabla\psi\|_{\mathbb{L}^4}^{4}ds+KE\int_{0}^{t}\|\Delta\psi\|_{\mathbb{L}^2}^{2}ds\nonumber\\
&\leq KE\int_{0}^{t}||\psi||_{\mathbb{H}^{1}}^{4}ds+KE\int_{0}^{t}\|\nabla\psi\|_{\mathbb{L}^2}^{10}ds+KE\int_{0}^{t}\|\Delta\psi\|_{\mathbb{L}^2}^{2}ds\nonumber\\
&\leq K+KE\int_{0}^{t}\|\Delta\psi\|_{\mathbb{L}^2}^{2}ds,\nonumber
\end{align}
where for the last inequality we use Corollary \ref{spatial H1} and equation \eqref{L2}.
Here, to estimate the second integral in \eqref{a4}, we have to restrict to $\mathcal{O}\subset\mathbb{R}^1$.

After using $\Re(i|\Delta\psi|^2|\psi|^2)=0$, the estimate of term $I_{b}^{2}$ is similar as before, and we have
\begin{equation}
  E(I^{2}_{b})\leq K+\int_{0}^{t}\|\psi\|_{\mathbb{H}^{2}}^{2}ds.
\end{equation}
Because of $\Re(i|\Delta \psi|^2|\psi|^2)=0$, the term $I^3$ can be estimated in a similar way by using H\"older's inequality and some embedding inequalities. It can be bounded by $K+KE\int_{0}^{t}\|\psi\|_{\mathbb{H}^{2}}^{2}ds$.

By the expression for $D^2 f(\psi)$ and since $\Re(\bar\psi(-i\psi {\bf Q}^{\frac12}))=0$, we have for term $II$,
\begin{align*}
  II=& \int_{0}^{t}\Re\int_{\mathcal{O}}\text{Tr}\Big[\big((Id-\Delta)(\overline{-i\psi {\bf Q}^{\frac12}})\big)\big((Id-\Delta)(-i\psi {\bf Q}^{\frac12})\big)\Big]dxds\\
    &+\int_{0}^{t}\Re\int_{\mathcal{O}}\text{Tr}\Big[\big((Id-\Delta)\bar{\psi}\big)\psi\Re\big((\overline{-i\psi {\bf Q}^{\frac12}})(-i\psi {\bf Q}^{\frac12})\big)\Big]dxds\\
    &+\int_{0}^{t}\Re\int_{\mathcal{O}}\text{Tr}\Big[\big((Id-\Delta)(\overline{-i\psi {\bf Q}^{\frac12}})\big)|\psi|^{2}(-i\psi {\bf Q}^{\frac12})\Big]dxds\\
    &+\int_{0}^{t}\Re\int_{\mathcal{O}}\text{Tr}\Big[\big((Id-\Delta)(\overline{-i\psi {\bf Q}^{\frac12}})\big)\psi\Re\big(\bar{\psi}(-i\psi {\bf Q}^{\frac12})\big)\Big]dxds.
\end{align*}
The estimate of term $II$ is similar to that of term $I^3$, using H\"older's inequality and embedding estimates.

 Because of the property of the It\^o stochastic integral, we know that the expectation of term $III$ equals to 0.

 Combining these together, we have
 \begin{align*}
   \sup_{0\leq t\leq T}E\|\psi(t)\|_{\mathbb{H}^{2}}^{2}&\leq |E(f(\psi(t)))|+\Big|E\Re\int_{\mathcal{O}}\big((Id-\Delta)\bar{\psi}(t)\big)|\psi(t)|^{2}\psi(t) dx\Big|\\
  & \leq \frac12\sup_{0\leq t\leq T}E\|\psi(t)\|_{\mathbb{H}^{2}}^{2}+K+ K\int_{0}^{T}E\|\psi(s)\|_{\mathbb{H}^2}^{2}ds,
 \end{align*}
 where in the last step, we use continuous embedding $\mathbb{H}^1\hookrightarrow \mathbb{L}^6$ and Corollary \ref{spatial H1}.
 Then the conclusion follows from Gronwall's lemma.
\end{proof}

\begin{remark}\label{re1}
There is only one term that requires a `1D-argument', which is the second term in \eqref{a4},
\[-2\int_{0}^{t}\Re\int_{\mathcal{O}}i\Big(\bar{\psi}(\nabla\psi)^{2}\Delta\bar{\psi}+2\psi|\nabla \psi|^{2}\Delta\bar{\psi}\Big)dxds
=-8\int_{0}^{t}\Re\int_{\mathcal{O}}i\psi|\nabla \psi|^{2}\Delta\bar{\psi}dx ds.\]
\end{remark}

\begin{lemma}\label{spatial E max}
  Let $\mathcal{O}\subset \mathbb{R}^1$, and suppose that $\psi_0\in L^{2p}(\Omega,{\mathbb H}_0^1\cap{\mathbb{H}^{2}(\mathcal{O})})$. Then there exists a constant {$K\equiv K(p,T)>0$} such that
  \begin{equation}\label{a2}
    E\Big(\sup_{0\leq t\leq T}\|\psi(t)\|_{\mathbb{H}^{2}}^{2p}\Big)\leq K.
  \end{equation}
\end{lemma}
\begin{proof}
 If compared to Lemma \ref{spatial H2 new}, the main difference of proof is the appearance of the supremum of stochastic integrals $III$ in \eqref{fpsi}, whose expectations do not vanish anymore. By the expression of $Df(\psi)$, we know
  \begin{align}\label{a13}
  III=&2\int_{0}^{t}\Re\int_{\mathcal{O}}\big((Id-\Delta)\bar\psi\big)(Id-\Delta)(-i\psi dW(s))dx\nonumber\\
  &+\int_{0}^{t}\Re\int_{\mathcal{O}}\big((Id-\Delta)\bar\psi\big)|\psi|^2(-i\psi dW(s))dx\nonumber\\
  &+\int_{0}^{t}\Re\int_{\mathcal{O}}\big((Id-\Delta)|\psi|^2\psi\big)(i\bar\psi dW(s))dx.
\end{align}

We deal with the first term in $III$ as an example, since the other two terms can be done similarly with Burkholder-Davis-Gundy inequality as well.
\begin{align*}
 & E\Big[\sup_{0\leq t\leq T}\|\int_{0}^{t}\Re\int_{\mathcal{O}}\big((Id-\Delta)\bar\psi\big)\big((Id-\Delta)(-i\psi dW(s))\big)dx\|_{\mathbb{L}^2}^p\Big]\\
  &\leq E\Big[\int_{0}^{T}\|\psi\|_{\mathbb{H}^2}^4 \|{\bf Q}^{\frac12}\|_{\mathcal{L}_{2}({\mathbb U},\;\mathbb{H}^2)}^{2}dt\Big]^{\frac{p}{2}}\\
  &\leq E\Big[\sup_{0\leq t\leq T}\|\psi(t)\|_{\mathbb{H}^2}^{p}\Big(\int_{0}^{T}\|\psi\|_{\mathbb{H}^2}^2 \|{\bf Q}^{\frac12}\|_{\mathcal{L}_{2}({\mathbb U},\;\mathbb{H}^2)}^{2}dt\Big)^{\frac{p}{2}}\Big]\\
  &\leq \frac18 E\Big(\sup_{0\leq t\leq T}\|\psi(t)\|_{\mathbb{H}^2}^{2p}\Big)+KE\int_{0}^{T}\|\psi(t)\|_{\mathbb{H}^2}^{2p}dt.
\end{align*}
Similar as the proof of Lemma \ref{spatial H2 new}, Gronwall's lemma leads to the assertion.
\end{proof}

\begin{lemma}\label{temporal L2}
Let $p\geq 1$, $\mathcal{O}\subset\mathbb{R}^1$ and $\psi_0\in L^{2p}(\Omega,\mathbb{H}_{0}^{1}(\mathcal{O}))$. There exists a constant
  $K\equiv K(p)$ such that
  \[E\Big(\|\psi(t_{1})-\psi(t_{2})\|_{\mathbb{L}^{2}}^{2p}\Big)\leq K|t_{1}-t_{2}|^p\qquad (0\leq t_2\leq t_1 \leq T).\]
\end{lemma}
\begin{proof}
  From equation \eqref{integral_form_solution}, we have the following expression for $\psi(t_{1})-\psi(t_{2})$,
  \begin{align}\label{A1}
    &\psi(t_{1})-\psi(t_{2})=(S(t_{1})-S(t_{2}))\psi_0\nonumber\\
    &+i\Big[\int_{0}^{t_{1}}S(t_{1}-r)\big(-|\psi|^{2}\psi+\frac{i}{2}\psi F_{{\bf Q}}\big)dr-
    \int_{0}^{t_{2}}S(t_{2}-r)\big(-|\psi|^{2}\psi+\frac{i}{2}\psi F_{{\bf Q}}\big)dr\Big]\nonumber\\
    &-i\Big[\int_{0}^{t_{1}}S(t_{1}-r)\psi dW(r)-\int_{0}^{t_{2}}S(t_{2}-r)\psi dW(r)\Big]\nonumber\\
    &=:I+II+III.
  \end{align}
  Because of Lemma \ref{lemma1} {\rm (i)},
  \begin{align*}
  \|S(t_{1})-S(t_{2})\|_{\mathcal{L}(\mathbb{H}_{0}^{1},\;\mathbb{L}^{2})}
  &=\|S(t_2)(S(t_1-t_2)-Id)\|_{\mathcal{L}(\mathbb{H}_{0}^{1},\;\mathbb{L}^{2})}\\
  &\leq
  \|S(t_2)\|_{\mathcal{L}(\mathbb{H}_{0}^1,\;\mathbb{H}_{0}^1)}\|S(t_1-t_2)-Id\|_{\mathcal{L}(\mathbb{H}_{0}^{1},\;\mathbb{L}^{2})}\\
&  \leq K|t_1-t_2|^{\frac12},
  \end{align*}
  such that
  \[E\|I\|_{\mathbb{L}^{2}}^{2p}\leq KE\|\psi_0\|_{\mathbb{H}_{0}^{1}}^{2p}|t_{1}-t_{2}|^p\leq K|t_{1}-t_{2}|^p.\]
  We divide $II$ into two parts,
  \begin{align}\label{12}
    II&=i\int_{0}^{t_{2}}(S(t_{1}-r)-S(t_{2}-r))(-|\psi|^{2}\psi+\frac{i}{2}\psi F_{{\bf Q}})dr+i\int_{t_{2}}^{t_{1}}S(t_{1}-r)(-|\psi|^{2}\psi+\frac{i}{2}\psi F_{{\bf Q}})dr\nonumber\\
    &=:II^{A}+II^{B}.
  \end{align}
  We use $\mathbb{H}^1 \hookrightarrow \mathbb{L}^{\infty}$ to estimate $II^{A}$ as follows,
  \begin{align*}
    \|II^{A}\|_{\mathbb{L}^{2}}
    &\leq K|t_{1}-t_{2}|^{\frac12}\int_{0}^{t_{2}}\|-|\psi|^{2}\psi+\frac{i}{2}\psi F_{{\bf Q}}\|_{\mathbb{H}_{1}}dr\\
    &\leq K|t_{1}-t_{2}|^{\frac12}\int_{0}^{t_{2}}(\|\psi\|_{\mathbb{H}^{1}}^{3}+\|\psi\|_{\mathbb{H}^{1}})dr,
  \end{align*}
  hence $E\|II^{A}\|_{\mathbb{L}^{2}}^{2p}\leq K|t_{1}-t_{2}|^p$ follows from \eqref{L2} and Corollary \ref{spatial H1}.
  By the embedding $\mathbb{H}^1\hookrightarrow \mathbb{L}^2$, the estimation of $II^{B}$ is
  \begin{align*}
    \|II^{B}\|_{\mathbb{L}^{2}}\leq K\int_{t_{1}}^{t_{2}}\|-|\psi|^{2}\psi+\frac{i}{2}\psi F_{{\bf Q}}\|_{\mathbb{L}^{2}}dr
    \leq K\int_{t_{1}}^{t_{2}}(\|\psi\|_{\mathbb{H}^{1}}^{3}+\|\psi\|_{\mathbb{H}^{1}})dr,
  \end{align*}
  thus $E\|II^{B}\|_{\mathbb{L}^{2}}^{2p}\leq K|t_{1}-t_{2}|^{2p}$.
  We split term $III$ as \eqref{12}.
   Based on the Burkholder-Davis-Gundy inequality, Lemma \ref{lemma1} {\rm (i)} and Lemma \ref{spatial H2 new}, the first stochastic term may be estimated as follows,
  \begin{align*}
    E\Big(\|\int_{0}^{t_{2}}(S(t_{1}-r)-S(t_{2}-r))\psi dW(r)\|_{\mathbb{L}^2}^{2p}\Big)
    &\leq KE\Big(\int_{0}^{t_{2}}\|\big(S(t_1-r)-S(t_2-r)\big)\psi\|_{\mathbb{L}^2}^2dr\Big)^{p}\\
    &\leq KE\Big(\int_{0}^{t_2}(t_1-t_2)\|\psi\|_{\mathbb{H}^1}^2dr\Big)^{p}\\
    &\leq K|t_{1}-t_{2}|^p,
  \end{align*}
and the estimate of the second stochastic term is
    \begin{align*}
    E\Big(\|\int_{t_{2}}^{t_{1}}S(t_{1}-r)\psi dW(r)\|_{\mathbb{L}^2}^{2}\Big)
    \leq E\Big(\int_{t_1}^{t_2}\|\psi\|_{\mathbb{L}^2}^2dr\Big)^p\leq K|t_{1}-t_{2}|^p.
  \end{align*}
Thus we have
  \[E\|III\|_{\mathbb{L}^{2}}^{2p}\leq K|t_{1}-t_{2}|^p.\]
  Inserting all these estimates into \eqref{A1} establishes the result.
\end{proof}

From  Lemma \ref{lemma1} {\rm (i)}, i.e.,  $\|S(t_1)-S(t_2)\|_{\mathcal{L}(\mathbb{H}_{0}^1,\;\mathbb{L}^2)}\leq K|t_1-t_2|^{\frac12}$, we may conclude that if we want to show the H\"older continuity property of the solution of \eqref{sdd1} in the $\mathbb{H}_{0}^1(\mathcal{O})$-norm, we need the boundedness of the $\mathbb{H}^2(\mathcal{O})$-norm of the solution, which is stated in Lemma \ref{spatial H2 new}. Therefore we present the following lemma without proof.
\begin{lemma}\label{temporal H1 new}
   Let $p\geq 1$, $\mathcal{O}\subset\mathbb{R}^1$ and $\psi_0\in L^{2p}(\Omega;\mathbb{H}_{0}^1 \cap \mathbb{H}^2(\mathcal{O}))$. There exists
   $K\equiv K(p)>0$ such that
  \[E\Big(\|\psi(t_{1})-\psi(t_{2})\|_{\mathbb{H}^{1}}^{2p}\Big)\leq K|t_{1}-t_{2}|^p\qquad (0\leq t_{2}\leq t_{1}\leq T).\]
\end{lemma}

\section{Stability of the $\theta$-scheme}\label{sec3}
In this section, we consider the following $\theta$-scheme on the uniform partition $I_{n}:=\{t_{n}\}_{n=0}^{M}$ covering $[0,T]$
with mesh-size $\tau=T/M>0$, where $t_{0}=0$ and $t_{M}=T$.
\begin{algorithm}\label{alg_theta}
  Let $\phi^{0}=\psi(t_{0})$ be a given $\mathbb{H}^{1}_{0}(\mathcal{O})$-valued random variable and let $\theta\in[0,1]$. Find for every $n\in\{0,\cdots,M\}$ a $\mathcal{F}_{t_{n+1}}$-adapted random variable $\phi^{n+1}$
  with values in $\mathbb{H}^{1}_{0}(\mathcal{O})$ such that $P$-a.s.
  \begin{align}\label{algorithm 2}
  i\int_{\mathcal{O}}\big(\phi^{n+1}-\phi^{n}\big)zdx&-\tau\int_{\mathcal{O}}\big(\theta\nabla \phi^{n+1}+(1-\theta)\nabla\phi^{n}\big)\nabla zdx\nonumber\\
 & -\frac{\tau}{2}\int_{\mathcal{O}}(|\phi^{n+1}|^{2}+|\phi^{n}|^{2})\phi^{n+\frac12}zdx=\int_{\mathcal{O}}\phi^{n+\frac12}\Delta_{n} W zdx\qquad \forall z\in {\mathbb H}_{0}^{1}(\mathcal{O}),
\end{align}
where $\Delta_{n} W=W(t_{n+1})-W(t_{n})$.
\end{algorithm}

A relevant property of the limiting system \eqref{sdd1} is a bound for the Hamiltonian of its solution; see \eqref{sdd20}. This property is not known for the Crank-Nicolson scheme ($\theta=\frac12$), which is why a truncation strategy is applied to the nonlinearity (see \cite{L1}) or the noise term (\cite{DBD1}), leading to a truncated Crank-Nicolson scheme. The next lemma establishes this property for the $\theta$-scheme and values {$\theta\in[\frac12+c\sqrt{\tau},1]$ with $c\geq c^{*}>0$}, avoiding any truncation. For simplicity, we assume $\phi^0 \in {\mathbb H}^1_0({\mathscr O})$.

\begin{lemma}\label{alg2:H1}
 Let $p\geq 1$ and $\mathcal{O}\subset\mathbb{R}^{d}$ be a bounded Lipschitz domain. Fix $T\equiv t_{M}>0$, and let {$\theta\in[\frac12+c\sqrt{\tau},1]$ with $c\geq c^{*}>0$}.
  Suppose $\tau\leq \tau^{*}$, where $\tau^{*}\equiv \tau^*( \Vert \phi^0\Vert_{{\mathbb H}^1_0} ,T)$. There exist a ${\mathbb H}_{0}^{1}(\mathcal{O})$-valued $\{\mathcal{F}_{t_{n}}\}_{0\leq n\leq M}$-adapted solution $\{\phi^{n};\;n=0,1,\cdots,M\}$ of the $\theta$-scheme \eqref{algorithm 2},
  and a constant {$K\equiv K(p,T,c^{*})>0$} such that
 \begin{align*}
   &{\rm (i)}\quad \max_{1\leq n\leq M}\Big[E\Big(\|\phi^{n}\|_{\mathbb{L}^{2}}^{2^{p}}+\big(\mathcal{H}(\phi^{n})\big)^{2^{(p-1)}}\Big)\Big]
        \leq K,   \\
      &  {\rm (ii)} \quad \max_{1\leq n\leq M}E\|\phi^{n+1}-\phi^{n}\|_{\mathbb{L}^{2}}^{2p}\leq K\tau^{p},\\
      &{\rm (iii)}\quad \max_{1\leq n\leq M } \Big[(2\theta-1)\sum_{k=0}^{n}E\|\nabla(\phi^{k+1}-\phi^{k})\|
        _{\mathbb{L}^{2}}^{2}\Big]\leq K.
 \end{align*}
\end{lemma}
  \begin{proof}
  {\em Step 1: Existence and $\mathcal{F}_{t_{n}}$-adaptedness.}
  Fix a set $\Omega^\prime\subset\Omega$, $P(\Omega^\prime)=1$ such that $W(t,x)\in {\mathbb U}$ for all $t\in[0,T]$ and $\omega\in \Omega^{\prime}$. In the following, let us assume that $\omega\in\Omega^\prime$.
  The existence of iterates $\{\phi^{n};\;n=0,1,\cdots,M\}$ follows from a standard Galerkin method and Brouwer's theorem, in combination with assertion {\rm (i)}.

 Define a map
 \[\Lambda:\; {\mathbb H}_0^1\times {\mathbb U}\ni\big(\phi^{n},\; \Delta_{n}W\big)\rightarrow \Lambda(\phi^{n},\Delta_{n}W)\in\mathcal{P}({\mathbb H}_0^1),\]
 where $\mathcal{P}({\mathbb H}_0^1)$ denotes the set of all subsets of ${\mathbb H}_0^1(\mathcal{O})$, and  $\Lambda(\phi^{n},\Delta_{n}W)$ is the set of solutions $\phi^{n+1}$ of \eqref{algorithm 2}. By the closedness of the graph of $\Lambda$ and a selector theorem (\cite{BT}, Theorem 3.1), there exists a universally and Borel measurable map $\lambda_{n}:{\mathbb H}_0^1\times {\mathbb U}\rightarrow {\mathbb H}_0^1$ such that $\lambda_{n}(s_1,\;s_2)\in\Lambda(s_1\;s_2)$ for all
 $(s_1,\;s_2)\in {\mathbb H}_0^1\times {\mathbb U}$. Therefore, $\mathcal{F}_{t_{n+1}}$-measurability of $\phi^{n+1}$ follows from the Doob-Dynkin lemma.

  {\em Step 2: Case $p=1$ for {\rm(i)}, {\rm(ii)} and {\rm(iii)}.}
    Consider equation \eqref{algorithm 2} for one $\omega\in\Omega$ and choose $z=\bar{\phi}^{n+\frac12}(\omega)$. Then take the imaginary part to get
    \begin{align}\label{alg L2}
      \frac12\|\phi^{n+1}\|_{\mathbb{L}^{2}}^{2}-\frac12\|\phi^{n}\|_{\mathbb{L}^{2}}^{2}&=\tau\Im\int_{\mathcal{O}}\big(\theta\nabla\phi^{n+1}+(1-\theta)\nabla\phi^{n}\big)\nabla\bar{\phi}^{n+\frac12}dx\nonumber\\
      &=\frac{(1-2\theta)\tau}{2}\Im\int_{\mathcal{O}}\nabla\phi^{n}\nabla\bar{\phi}^{n+1}dx\\
      &\leq \frac{2\theta-1}{4}\tau\Big(\|\nabla\phi^{n+1}\|_{\mathbb{L}^{2}}^{2}+\|\nabla\phi^{n}\|_{\mathbb{L}^{2}}^{2}\Big),\nonumber
    \end{align}
    where $\Re\big[(a-b)(\bar a+\bar b)\big]=|a|^2-|b|^2$ is used on the left-hand side.
    Next, we choose $z=-(\bar{\phi}^{n+1}-\bar{\phi}^{n})(\omega)$ in \eqref{algorithm 2},  and take the real part. We obtain
    \begin{align}\label{alg H1}
     &\Big(  \frac12\|\nabla\phi^{n+1}\|_{\mathbb{L}^{2}}^{2}+\frac14\|\phi^{n+1}\|_{\mathbb{L}^{4}}^{4}\Big)-\Big(\frac12\|\nabla\phi^{n}\|_{\mathbb{L}^{2}}^{2}+\frac14\|\phi^{n}\|_{\mathbb{L}^{4}}^{4}\Big)
      +\frac{(2\theta-1)}{2}\|\nabla(\phi^{n+1}-\phi^{n})\|_{\mathbb{L}^{2}}^{2}\nonumber\\
      &=-\frac{1}{\tau}\int_{\mathcal{O}}(|\phi^{n+1}|^{2}-|\phi^{n}|^{2})\Delta_{n} Wdx.
    \end{align}
    We will see that the last term on the left-hand side helps to bound the stochastic integral term,
    which is restated as follows by using the equation \eqref{algorithm 2}, properties of the real and imaginary parts of a complex number, and the fact that $W$ is real-valued,
    \begin{align*}
     & \int_{\mathcal{O}}(|\phi^{n+1}|^{2}-|\phi^{n}|^{2})\Delta_{n} Wdx\\
      &=2\Re \int_{\mathcal{O}}\bar{\phi}^{n+\frac12}(\phi^{n+1}-\phi^{n})\Delta_{n} Wdx\\
      &=2\Re \int_{\mathcal{O}}\bar{\phi}^{n+\frac12}
      \Big(i\tau\big(\theta\Delta\phi^{n+1}+(1-\theta)\Delta\phi^{n}\big)-i\frac{\tau}{2}(|\phi^{n+1}|^{2}+|\phi^{n}|^{2})\phi^{n+\frac12}-i\phi^{n+\frac12}\Delta_{n} W\Big)\Delta_{n} Wdx\\
      &=(1-2\theta)\tau\Im\int_{\mathcal{O}}\nabla\bar{\phi}^{n}\nabla\phi^{n+1}\Delta_{n} Wdx-2\tau\theta\Im\int_{\mathcal{O}}\bar{\phi}^{n+\frac12}\nabla\phi^{n+1}\nabla(\Delta_{n} W)dx
      -2\tau(1-\theta)\Im\int_{\mathcal{O}}\bar{\phi}^{n+\frac12}\nabla\phi^{n}\nabla(\Delta_{n} W)dx.
    \end{align*}
    We used integration by parts in the last step. By
   plugging it into equation \eqref{alg H1}, we find
    \begin{align}\label{ineq1}
    &\mathcal{H}(\phi^{n+1})-\mathcal{H}(\phi^{n})+\frac{2\theta-1}{2}\|\nabla(\phi^{n+1}-\phi^{n})\|_{\mathbb{L}^2}^2\nonumber\\
    &=(2\theta-1)\Im\int_{\mathcal{O}}\nabla\bar{\phi}^{n}\nabla\phi^{n+1}\Delta_{n} Wdx+2\theta\Im\int_{\mathcal{O}}\bar{\phi}^{n+\frac12}\nabla\phi^{n+1}\nabla(\Delta_{n} W)dx
      +2(1-\theta)\Im\int_{\mathcal{O}}\bar{\phi}^{n+\frac12}\nabla\phi^{n}\nabla(\Delta_{n} W)dx\nonumber\\
      &=:I_{1}+I_{2}+I_{3}.
    \end{align}
    Next we estimate the three terms separately. Because of $\Im\big(|\nabla\phi^{n}|^2\big)=0$, we have
    \begin{align*}
      I_{1}&=(2\theta-1)\Im\int_{\mathcal{O}}\nabla\bar{\phi}^{n}\big(\nabla\phi^{n+1}-\nabla\phi^{n}\big)\Delta_{n} Wdx\\
      &\leq \frac{2\theta-1}{8} \|\nabla\phi^{n+1}-\nabla\phi^{n}\|_{\mathbb{L}^2}^2+2(2\theta-1)\|\nabla\phi^{n}\|_{\mathbb{L}^2}^2\|\Delta_{n} W\|_{\mathbb{L}^{\infty}}^2.
    \end{align*}
    Rearranging terms and the identity $\phi^{n+\frac12}=\phi^{n}+\frac{\phi^{n+1}-\phi^{n}}{2}$ lead to
    \begin{align}\label{i2i3}
      I_{2}+I_{3}
      =&2\theta\Im \int_{\mathcal{O}}\bar{\phi}^{n+\frac12}\big(\nabla\phi^{n+1}-\nabla\phi^{n}\big)\nabla(\Delta_{n} W)dx
      +2\Im \int_{\mathcal{O}}\bar{\phi}^{n+\frac12}\nabla\phi^{n}\nabla(\Delta_{n} W)dx\nonumber\\
      =&2\theta \Im \int_{\mathcal{O}}\bar\phi^{n}\big(\nabla\phi^{n+1}-\nabla\phi^{n}\big)\nabla(\Delta_{n} W)dx+\theta\Im
      \int_{\mathcal{O}}(\bar{\phi}^{n+1}-\bar\phi^{n})\big(\nabla\phi^{n+1}-\nabla\phi^{n}\big)\nabla(\Delta_{n} W)dx\nonumber\\
      &+2\Im\int_{\mathcal{O}}\bar{\phi}^{n}\nabla\phi^{n}\nabla(\Delta_{n} W)dx
      +\Im\int_{\mathcal{O}}(\bar\phi^{n+1}-\bar{\phi}^{n})\nabla\phi^{n}\nabla(\Delta_{n} W)dx
    \end{align}
   Integration by parts for the first term, and using $\Im(a)=-\Im(\bar{a})$ $(a\in\mathbb{C})$ lead to
   \begin{align}\label{i2i3'}
     I_{2}+I_{3}
     =&(1+2\theta)\Im\int_{\mathcal{O}}\nabla\phi^{n}(\bar{\phi}^{n+1}-\bar{\phi}^{n})\nabla(\Delta_{n} W)dx
      +2\Im\int_{\mathcal{O}}\nabla\phi^{n}\bar{\phi}^{n}\nabla(\Delta_{n} W)dx\nonumber\\
      &-2\theta\Im\int_{\mathcal{O}}\bar{\phi}^{n}(\phi^{n+1}-\phi^{n})\Delta(\Delta_{n} W)dx
      + \theta\Im\int_{\mathcal{O}}(\bar{\phi}^{n+1}-\bar{\phi}^{n})(\nabla\phi^{n+1}-\nabla\phi^{n})\nabla(\Delta_{n} W)dx\nonumber\\
      =&:I_{23}^{a}+I_{23}^b+I_{23}^c+I_{23}^d.
   \end{align}
    The estimation of the first three terms is as follows,
    \begin{align}\label{eq4}
      I_{23}^{a}+I_{23}^b+I_{23}^c\leq&
      \frac14\|\phi^{n+1}-\phi^{n}\|_{\mathbb{L}^2}^2+K\|\nabla\phi^{n}\|_{\mathbb{L}^2}^2\|\nabla(\Delta_{n} W)\|_{\mathbb{L}^{\infty}}^2
      +K\|\phi^{n}\|_{\mathbb{L}^2}^2\|\Delta(\Delta_{n} W)\|_{\mathbb{L}^{\infty}}^2\nonumber\\
      &+2\Im\int_{\mathcal{O}}\nabla\phi^{n}\bar{\phi}^{n}\nabla(\Delta_{n} W)dx.
    \end{align}
    The troublesome term is $I_{23}^d$, we estimate it as follows,
    \begin{align}\label{eq5}
      I_{23}^d\leq& \|\nabla\phi^{n+1}-\nabla\phi^{n}\|_{\mathbb{L}^2}\|\phi^{n+1}-\phi^{n}\|_{\mathbb{L}^2}\|\nabla(\Delta_{n} W)\|_{\mathbb{L}^{\infty}}\nonumber\\
      \leq & \frac{2\theta-1}{8}\|\nabla\phi^{n+1}-\nabla\phi^{n}\|_{\mathbb{L}^2}^2+\frac{2}{2\theta-1}\|\phi^{n+1}-\phi^{n}\|_{\mathbb{L}^2}^2\|\nabla(\Delta_{n} W)\|_{\mathbb{L}^{\infty}}^2\nonumber\\
      \leq & \frac{2\theta-1}{8}\|\nabla\phi^{n+1}-\nabla\phi^{n}\|_{\mathbb{L}^2}^2+\frac18\|\phi^{n+1}-\phi^{n}\|_{\mathbb{L}^2}^2+\frac{2}{(2\theta-1)^2}
      \|\phi^{n+1}-\phi^{n}\|_{\mathbb{L}^2}^2\|\nabla(\Delta_{n} W)\|_{\mathbb{L}^{\infty}}^4\nonumber\\
      \leq & \frac{2\theta-1}{8}\|\nabla\phi^{n+1}-\nabla\phi^{n}\|_{\mathbb{L}^2}^2+\frac18\|\phi^{n+1}-\phi^{n}\|_{\mathbb{L}^2}^2
      +K\tau\Big(\|\phi^{n+1}\|_{\mathbb{L}^4}^4+\|\phi^{n}\|_{\mathbb{L}^4}^4\Big)+\frac{1}{\tau(2\theta-1)^4}\|\nabla(\Delta_{n} W)\|_{\mathbb{L}^{\infty}}^8,
    \end{align}
    where we use the embedding $\mathbb{L}^{4}(\mathcal{O})\hookrightarrow \mathbb{L}^{2}(\mathcal{O})$ in the last step. In order to complete the proof for {\rm(i)} and {\rm(ii)}, we need to bound $\|\phi^{n+1}-\phi^{n}\|_{\mathbb{L}^2}^{2}$, which appears in the last two estimates \eqref{eq4} and \eqref{eq5}. For this purpose, we
     test the equation \eqref{algorithm 2} with $(\bar{\phi}^{n+1}-\bar{\phi}^{n})(\omega)$, then take the imaginary part. Because of $\phi^{n+\frac12}=\phi^{n}+\frac{\phi^{n+1}-\phi^{n}}{2}$, we get
     \begin{align*}
       \|\phi^{n+1}-\phi^{n}\|_{\mathbb{L}^2}^{2}=&\tau\Im\int_{\mathcal{O}}\big(\theta\nabla\phi^{n+1}+(1-\theta)\nabla\phi^{n}\big)\nabla(\bar{\phi}^{n+1}-\bar{\phi}^{n})dx
       +\frac{\tau}{2}
       \Im\int_{\mathcal{O}}(|\phi^{n+1}|^{2}+|\phi^{n}|^{2})\phi^{n}\bar{\phi}^{n+1}dx\\
       &+\Im\int_{\mathcal{O}}\phi^{n}(\bar\phi^{n+1}-\bar\phi^{n})\Delta
       W_{n}dx.
     \end{align*}
     Estimating this equality leads to
     \begin{align}\label{alg diffe L2}
      \frac12 \|\phi^{n+1}-\phi^{n}\|_{\mathbb{L}^2}^{2}\leq & K\tau\Big(\frac12\|\nabla\phi^{n+1}\|_{\mathbb{L}^2}^{2}+\frac12\|\nabla\phi^{n}\|_{\mathbb{L}^2}^{2}+\frac14\|\phi^{n+1}\|_{\mathbb{L}^4}^{4}
      +\frac14||\phi^{n}||_{\mathbb{L}^4}^{4}\Big)
       +K\|\phi^{n}\|_{\mathbb{L}^2}^{2}\|\Delta_{n} W\|_{\mathbb{L}^{\infty}}^{2}\nonumber\\
       =&K\tau\Big(\mathcal{H}(\phi^{n+1})+\mathcal{H}(\phi^{n})\Big)+K\|\phi^{n}\|_{\mathbb{L}^2}^{2}\|\Delta_{n} W\|_{\mathbb{L}^{\infty}}^{2},
     \end{align}
     where Young's inequality is applied, and the term $\frac12 \|\phi^{n+1}-\phi^{n}\|_{\mathbb{L}^2}^{2}$ which appears from the stochastic term is absorbed in the left-hand side.

     We may now combine estimate \eqref{alg diffe L2} with \eqref{alg L2} and \eqref{ineq1}. By denoting $\mathcal{K}^{n}=\frac12\|\phi^{n}\|_{\mathbb{L}^2}^{2}+\mathcal{H}(\phi^{n})$, we obtain
     \begin{align}\label{eq6}
       &\mathcal{K}^{n+1}-\mathcal{K}^{n}+\frac{1}{8} \|\phi^{n+1}-\phi^{n}\|_{\mathbb{L}^2}^{2}+\frac{2\theta-1}{4} \|\nabla\phi^{n+1}-\nabla\phi^{n}\|_{\mathbb{L}^2}^{2}\nonumber\\
      & \leq K\tau\Big( \mathcal{K}^{n+1} +\mathcal{K}^{n}\Big)+K\|\nabla\phi^{n}\|_{\mathbb{L}^2}^2\|\Delta_{n} W\|_{\mathbb{L}^{\infty}}^2+K\|\phi^{n}\|_{\mathbb{L}^2}^2\|\Delta(\Delta_{n} W)\|_{\mathbb{L}^{\infty}}^2
      +K\|\phi^{n}\|_{\mathbb{L}^2}^2\|\Delta_{n} W\|_{\mathbb{L}^{\infty}}^2\nonumber\\
      &\quad
      +2\Im\int_{\mathcal{O}}\nabla\phi^{n}\bar\phi^{n}\nabla(\Delta_{n} W)dx+\frac{1}{\tau(2\theta-1)^4}\|\nabla(\Delta_{n} W)\|_{\mathbb{L}^{\infty}}^{8}\nonumber\\
      &=:K\tau\Big( \mathcal{K}^{n+1} +\mathcal{K}^{n}\Big)+A.
     \end{align}
     In order to efficiently bound the expectation of the last term, we recall that $E\|\nabla(\Delta_{n} W)\|_{\mathbb{L}^{\infty}}^{8}=O(\tau^4)$
     to admit {$2\theta-1\geq c\sqrt{\tau}$ with $c\geq c^{*}>0$}.

     After applying expectations on both sides of \eqref{eq6}, one arrives at
     \begin{align*}
       E\mathcal{K}^{n+1}-E\mathcal{K}^{n}+\frac{1}{8} E\|\phi^{n+1}-\phi^{n}\|_{\mathbb{L}^2}^{2}+\frac{2\theta-1}{4} E\|\nabla\phi^{n+1}-\nabla\phi^{n}\|_{\mathbb{L}^2}^{2}\leq K\tau+K\tau \Big(E\mathcal{K}^{n+1}+E\mathcal{K}^{n}\Big).
     \end{align*}

     The discrete Gronwall's lemma then leads to the assertions of this lemma in case $\tau\leq\tau^{*}$ is chosen.

     {\em Step 3: Case $p\geq 2$ for {\rm(i)}.} In order to show the assertion ${\rm(i)}$, we employ an inductive argument. To obtain the result for $p=2$, we multiply equality
   \eqref{eq6} by $\mathcal{K}^{n+1}$ and use the identity $(a-b)a=\frac12 \big(a^{2}-b^{2}+ (a-b)^{2}\big)$, where $a,b\in\mathbb{R}$, to get
    \begin{align}\label{eq8}
      \frac12 \Big[(\mathcal{K}^{n+1})^{2}-(\mathcal{K}^{n})^{2}\Big]+\frac12 (\mathcal{K}^{n+1}-\mathcal{K}^{n})^{2}
      \leq K\tau\Big((\mathcal{K}^{n+1})^{2}+(\mathcal{K}^{n})^{2}\Big)+A\mathcal{K}^{n+1},
    \end{align}
    where $A$ is from \eqref{eq6}.
   Applying expectation on both sides of \eqref{eq8}, we have
   \begin{align}\label{17}
        &\frac12 E\Big[(\mathcal{K}^{n+1})^{2}-(\mathcal{K}^{n})^{2}\Big]+\frac12 E(\mathcal{K}^{n+1}-\mathcal{K}^{n})^{2}\nonumber\\
       &\leq K\tau \Big(E(\mathcal{K}^{n+1})^{2}+ E(\mathcal{K}^{n})^{2}\Big)+\frac14E(\mathcal{K}^{n+1}-\mathcal{K}^{n})^{2}+K\tau.
   \end{align}
   In order to verify this inequality, we may restrict ourselves to the integral term in \eqref{eq6},
    since other terms can be easily estimated by Young's inequality.
  By the independency property of increments of the Wiener process, we obtain
   \begin{align*}
     E\Big[\mathcal{K}^{n+1}\Im\int_{\mathcal{O}}\nabla\phi^{n}\bar\phi^{n}\nabla(\Delta_{n} W)dx\Big]
     &=E\Big[\big(\mathcal{K}^{n+1}-\mathcal{K}^{n}\big)\Im\int_{\mathcal{O}}\nabla\phi^{n}\bar\phi^{n}\nabla(\Delta_{n} W)dx\Big]\\
     &\leq \frac14 E(\mathcal{K}^{n+1}-\mathcal{K}^{n})^{2}+ K\tau E(\mathcal{K}^{n})^2,
   \end{align*}
   and the leading term may be absorbed by the left-hand side of \eqref{17}.
   Therefore we have the conclusion of ${\rm(i)}$ in the case $p=2$ via the discrete Gronwall's lemma.
By repeating this procedure, one obtains the result for each $p\in \mathbb{N}$.

   {\em Step 4: Case $p\geq 2$ for {\rm(ii)}.}  We prove it for the case $p=2$, since for general $p$, the result follows from assertion ${\rm(i)}$. We deal with inequality \eqref{alg diffe L2} by squaring it,
   \begin{equation*}
     \|\phi^{n+1}-\phi^{n}\|_{\mathbb{L}^2}^{4}\leq K\tau^2 \Big((\mathcal{K}^{n+1})^2+ (\mathcal{K}^{n})^2\Big)+K\|\phi^{n}\|_{\mathbb{L}^2}^{4}\|\Delta_{n} W\|_{\mathbb{L}^{\infty}}^{4}.
   \end{equation*}
   Applying expectations leads to assertion ${\rm(ii)}$ in the case of $p=2$.
By repeating this procedure, one obtains the result for each $p\in \mathbb{N}$.
  \end{proof}

{\begin{remark}\label{r-two}
 A compactness argument is used in \cite{DBD1} to prove convergence of a family of (adapted, continuous) interpolating processes of the numerical solution towards  a mild solution of (\ref{S_NLSE_Ito}) for the case ${\mathcal O}={\mathbb R}^d$;
 a crucial prerequisite for it are the lemmas \cite[Lemmas 3.3 and 3.4]{DBD1}, which here are sharpened to Lemma~\ref{alg2:H1}.

 As is stated in Remark 2, a  mild solution of (\ref{S_NLSE_Ito}) may be constructed for the bounded domain case
 ${\mathcal O}\subset{\mathbb R}^1$ by a contraction argument following~\cite{DBD0a}; alternatively, we may follow the strategy of \cite{DBD1} and use the uniform bounds in Lemma \ref{alg2:H1} for a compactness argument which establishes
convergence of  (interpolated in time) iterates $\{\phi^{n};\;n=0,1,\cdots,M\}$ solving Algorithm~\ref{alg_theta}  towards the unique mild solution of
(\ref{S_NLSE_Ito}) for the case ${\mathcal O}\subset{\mathbb R}^1$.
No additional truncation parameter (and related stopping times) is involved in this construction based on
Algorithm~\ref{alg_theta}, which would otherwise require a proper balancing with the discretization parameter in this (practical) construction process of a solution for (\ref{S_NLSE_Ito}) as in \cite{DBD1}.
%
 \end{remark}}

  \begin{lemma}\label{H1 E max}
  Let $p\geq 1$.  Under the assumptions made in Lemma \ref{alg2:H1}, we have
    \[ E\Big[\max_{1\leq n\leq M} \Big(\|\phi^{n}\|_{\mathbb{L}^2}^{2}+\mathcal{H}(\phi^{n}) \Big)^{2^{p-1}}\Big]\leq {K(p,T)}.\]
  \end{lemma}
  \begin{proof}
   We only present the proof for $p=1$.
    We start from \eqref{eq6} for some $0\leq \ell\leq M$, sum over the index from $\ell=0$ to $n$, take the maximum between $0$ and $m\leq M$, and apply expectations.
    We may now employ the result of Lemma \ref{alg2:H1} to conclude that
    \begin{align}
      E\Big(\max_{0\leq n\leq m}\mathcal{K}^n\Big)\leq K+K\tau\sum_{\ell=0}^{m}E\Big(\max_{0\leq j\leq \ell}\mathcal{K}^{j}\Big)
      +E\Big[\max_{0\leq n\leq m}\sum_{\ell=0}^{n}\int_{\mathcal{O}}\nabla\phi^{\ell}\bar\phi^{\ell}\nabla(\Delta_{\ell} W)dx\Big].
    \end{align}
    The bound of the last term is similar to \eqref{stochatic}, using Burkholder-Davis-Gundy inequality.
  \end{proof}

The following lemma asserts approximate conservation of mass (in statistical average) for $\theta\downarrow \frac12$.
   \begin{lemma}\label{alg2:L2'}
  Let $\mathcal{O}\subset\mathbb{R}^{d}$ be a bounded Lipschitz domain, $T\equiv t_{M}>0$ be fixed, and {$\theta\in[\frac12+c\sqrt{\tau},1]$ with $c\geq c^{*}>0$}. There exist a constant {$K \equiv K(T,c^{*})>0$} and $\tau^{*}\equiv \tau^{*}(
   \Vert \phi^0 \Vert_{{\mathbb H}_0^1},T)$ such that
for all $\tau\leq\tau^{*}$, we have
    \begin{align}
      \max_{1\leq n\leq \lfloor T/\tau\rfloor}E\|\phi^{n}\|_{\mathbb{L}^2}^{2}-E\|\phi^{0}\|_{\mathbb{L}^2}^{2}\leq K\Big(\tau^{\frac34}+(1-2\theta)\tau^{\frac14}\Big).
    \end{align}
  \end{lemma}
  \begin{proof}
   Recall \eqref{alg L2} and use properties of the imaginary part of a complex number to conclude
    \begin{align}\label{E1}
      \|\phi^{n+1}\|_{\mathbb{L}^2}^{2}-\|\phi^{n}\|_{\mathbb{L}^2}^{2}
      &=(1-2\theta)\tau\Im\int_{\mathcal{O}}(\nabla\bar{\phi}^{n+1}-\nabla\bar{\phi}^{n})\nabla\phi^{n}dx\nonumber\\
      &\leq (2\theta-1)\Big(\tau^{\frac38}\|\nabla\phi^{n+1}-\nabla\phi^{n}\|_{\mathbb{L}^2}\Big)\Big(\tau^{\frac58}\|\nabla\phi^{n}\|_{\mathbb{L}^2}\Big)\\
      &\leq \frac{(2\theta-1) \tau^{\frac34}}{2}\|\nabla\bar{\phi}^{n+1}-\nabla\bar{\phi}^{n}\|_{\mathbb{L}^2}^{2}+\frac{(2\theta-1)\tau^{\frac54}}{2}\|\nabla\bar\phi^{n}\|_{\mathbb{L}^2}^{2}.\nonumber
    \end{align}
    Now consider the above inequality for some $0\leq \ell\leq M$, sum over the index from $\ell=0$ to $n$, take the expectation, and use Lemma \ref{alg2:H1} {\rm(i)} and {\rm (iii)} to establish the assertion.
  \end{proof}
A comparison of Lemma \ref{alg2:H1} and Lemma \ref{alg2:L2'} illustrates the role of numerical dissipation in the $\theta$-scheme and suggests a choice {$\theta=\frac12+ c\sqrt{\tau}$} to minimize this effect and approximately preserve the $\mathbb{L}^2$-norm of iterates.
\medskip

The following lemma validates improved stability properties for solutions of Algorithm 1 for $\mathcal{O}\subset\mathbb{R}^{1}$, which will be relevant in the error analysis below. In fact, a consequence of it will be an improved preservation of mass; see Lemma \ref{alg2:L2}.
  \begin{lemma}\label{alg2:H2}
  Let $p\geq 1$, $\mathcal{O}\subset\mathbb{R}^{1}$, $T\equiv t_{M}>0$ be fixed, $\phi^{0}\in L^{2p}(\Omega;\mathbb{H}_{0}^{1}\cap\mathbb{H}^2(\mathcal{O}))$, and $W$ be $\mathbb{H}_{0}^{2}\cap {\mathbb H}^3$-valued. Suppose {$\theta\in [\frac12+c\sqrt{\tau},1]$ with $c\geq c^{*}>0$}. There exist a constant {$K\equiv K(p,T,c^{*})>0$}, and $\tau^{*}\equiv \tau^{*}(
   \Vert \phi^0 \Vert_{\mathbb{H}_0^1\cap\mathbb{H}^2},T)$ such that for all $\tau\leq \tau^{*}$ holds
   \begin{align*}
   &{\rm (1)}\quad \max_{1\leq n\leq M}\Big[ E\Big(\|\phi^{n}\|_{\mathbb{H}^2}^{2}+\sum_{k=0}^{n}\|\phi^{k+1}-\phi^{k}\|_{\mathbb{H}^1}^{2}
            +(2\theta-1)\sum_{k=0}^{n}\|\phi^{k+1}-\phi^{k}\|_{\mathbb{H}^2}^{2}\Big)\Big]\leq K,\\
   &{\rm (ii)}\quad \max_{1\leq n\leq M}E\Big(\|\phi^{n}\|_{\mathbb{H}^2}^{2^p}\Big)\leq K,\\
   &{\rm (iii)}\quad \max_{1\leq n\leq M}E\|\phi^{n+1}-\phi^{n}\|_{\mathbb{H}^1}^{2p}\leq K\tau^{p},\\
   &{\rm (iv)}\quad E\Big(\max_{1\leq n\leq M}\|\phi^{n}\|_{\mathbb{H}^2}^{2^p}\Big)\leq K.
   \end{align*}
  \end{lemma}
  \begin{proof}
  We formally test equation \eqref{algorithm 2} with $z=\Delta\Big(\bar\phi^{n+1}-\bar\phi^{n}\Big)$ and take the real part. Because of $\theta\Delta\phi^{n+1}+(1-\theta)\Delta\phi^{n}=\Delta\phi^{n+1}+(\theta-1)\big(\Delta\phi^{n+1}-\Delta\phi^{n}\big)$ and
  $\Re\Big(a(\bar a-\bar b)\Big)=\frac12\Big(|a|^2-|b|^2+|a-b|^2\Big)$, we have
  \begin{align}\label{eq7}
    \|\Delta\phi^{n+1}\|_{\mathbb{L}^2}^2-\|\Delta\phi^{n}\|_{\mathbb{L}^2}^2+(2\theta-1)\|\Delta\phi^{n+1}-\Delta\phi^{n}\|_{\mathbb{L}^2}^2
    =&\Re\int_{\mathcal{O}}\big(|\phi^{n+1}|^2+|\phi^{n}|^2\big)\phi^{n+\frac12}\Delta(\bar\phi^{n+1}-\bar\phi^{n})dx\nonumber\\
    &+\frac{2}{\tau}\Re\int_{\mathcal{O}}\phi^{n+\frac12}\Delta_{n} W\Delta(\bar\phi^{n+1}-\bar\phi^{n})dx\nonumber\\
    =:&A+B.
  \end{align}
  {\em Step 1: Estimate of the stochastic integral term B.}
  We use integration by parts to benefit from equation \eqref{algorithm 2} and $W$ being real-valued,
  \begin{align}
       B&=\frac{2}{\tau}\Re\int_{\mathcal{O}}\phi^{n+\frac12}\Delta_{n} W(\Delta\bar{\phi}^{n+1}-\Delta\bar{\phi}^{n})dx\nonumber\\
      &=\frac{2}{\tau}\Re\int_{\mathcal{O}}\Delta(\bar{\phi}^{n+\frac12}\Delta_{n} W)(\phi^{n+1}-\phi^{n})dx\nonumber\\
      &=\frac{2}{\tau}\Re\int_{\mathcal{O}}\Delta(\bar{\phi}^{n+\frac12}\Delta_{n} W)\Big[i\tau\big(\theta\Delta\phi^{n+1}+(1-\theta)\Delta\phi^{n}\big)-i\frac{\tau}{2}(|\phi^{n+1}|^{2}+|\phi^{n}|^{2})\phi^{n+
      \frac12}-i\phi^{n+\frac12}\Delta_{n} W\Big]dx\nonumber\\
      &=2\Re\int_{\mathcal{O}}i\Delta(\bar{\phi}^{n+\frac12}\Delta_{n} W)\big(\theta\Delta\phi^{n+1}+(1-\theta)\Delta\phi^{n}\big)dx
      -\Re\int_{\mathcal{O}}i\Delta(\bar{\phi}^{n+\frac12}\Delta_{n} W)(|\phi^{n+1}|^{2}+|\phi^{n}|^{2})\phi^{n+\frac12}dx\nonumber\\
      &=:B^{1}+B^{2}.
    \end{align}

{\em Step 2: Estimate of term $B^{1}.$}
    We rewrite the term $B^1$ as follows,
    \begin{align}\label{b1}
      B^{1}&=2\Re\int_{\mathcal{O}}i\Delta(\bar{\phi}^{n+\frac12}\Delta_{n} W)\big(\theta\Delta\phi^{n+1}+(1-\theta)\Delta\phi^{n}\big)dx\nonumber\\
      &=2\Re\int_{\mathcal{O}}i\Delta\bar{\phi}^{n+\frac12}\big(\theta\Delta\phi^{n+1}+(1-\theta)\Delta\phi^{n}\big)\Delta_{n} Wdx
      +2\Re\int_{\mathcal{O}}i\bar{\phi}^{n+\frac12}\Delta(\Delta_{n} W)\big(\theta\Delta\phi^{n+1}+(1-\theta)\Delta\phi^{n}\big)dx\nonumber\\
      &\quad+4\Re\int_{\mathcal{O}}i\nabla\bar{\phi}^{n+\frac12}\nabla(\Delta_{n} W)\big(\theta\Delta\phi^{n+1}+(1-\theta)\Delta\phi^{n}\big)dx\nonumber\\
      &=:B^{1}_{a}+B_{b}^{1}+B_{c}^{1}.
    \end{align}
    Since $\theta\Re\big(i\|\Delta \phi^{n+1}\|_{\mathbb{L}^2}^2\big)+(1-\theta)\Re\big(i\|\Delta \phi^{n}\|_{\mathbb{L}^2}^2\big)=0$, we have
    \begin{align*}
      B_{a}^{1}&=\theta\Re\int_{\mathcal{O}}i\Delta\bar{\phi}^{n}\Delta\phi^{n+1}\Delta_{n} Wdx
      +(1-\theta)\Re\int_{\mathcal{O}}i\Delta\bar{\phi}^{n+1}\Delta\phi^{n}\Delta_{n} Wdx\\
      &=(2\theta-1)\Re\int_{\mathcal{O}}i\Delta\bar{\phi}^{n}\Delta\phi^{n+1}\Delta_{n} Wdx\\
      &=(2\theta-1)\Re\int_{\mathcal{O}}i\Delta\bar{\phi}^{n}(\Delta\phi^{n+1}-\Delta\phi^{n})\Delta_{n} Wdx\\
      &\leq (2\theta-1)\|\Delta\phi^{n}\|_{\mathbb{L}^{2}}\|\Delta\phi^{n+1}-\Delta\phi^{n}\|_{\mathbb{L}^{2}}\|\Delta_{n} W\|_{\mathbb{L}^{\infty}}\\
      &\leq \frac{2\theta-1}{8} \|\Delta\phi^{n+1}-\Delta\phi^{n}\|_{\mathbb{L}^{2}}^{2}+2(2\theta-1)\|\Delta\phi^{n}\|_{\mathbb{L}^2}^{2}\|\Delta_{n} W\|_{\mathbb{L}^{\infty}}^{2}.
    \end{align*}
    Therefore $E(B_{a}^{1})\leq \frac{2\theta-1}{8} E\|\Delta\phi^{n+1}-\Delta\phi^{n}\|_{\mathbb{L}^2}^{2}+K\tau E\|\Delta\phi^{n}\|_{\mathbb{L}^2}^{2}$.

    Since $\phi^{n+\frac12}=\phi^{n}+\frac{\phi^{n+1}-\phi^{n}}{2}$, we have the following estimate for the term $B_{b}^1$ in \eqref{b1},
    \begin{align*}
      B_{b}^{1}
      &=2\Re\int_{\mathcal{O}}i\bar\phi^{n}\Delta(\Delta_{n} W)\big(\theta\Delta\phi^{n+1}+(1-\theta)\Delta\phi^{n}\big)dx
      +\Re\int_{\mathcal{O}}i(\bar\phi^{n+1}-\bar\phi^{n})\Delta(\Delta_{n} W)\big(\theta\Delta\phi^{n+1}+(1-\theta)\Delta\phi^{n}\big)dx\\
      &=2\theta\Re\int_{\mathcal{O}}i\bar\phi^{n}\Delta(\Delta_{n} W)\Delta(\phi^{n+1}-\phi^{n})dx
      +2\Re\int_{\mathcal{O}}i\bar\phi^{n}\Delta(\Delta_{n} W)\Delta \phi^{n}dx\\
      &\quad  +\Re\int_{\mathcal{O}}i(\bar\phi^{n+1}-\bar\phi^{n})\Delta(\Delta_{n} W)\big(\theta\Delta\phi^{n+1}+(1-\theta)\Delta\phi^{n}\big)dx
    \end{align*}
    Integration by parts for the first term leads to
    \begin{align*}
      B_b^1&=-2\theta\Re\int_{\mathcal{O}}i\nabla\big(\bar{\phi}^{n}\Delta(\Delta_{n} W)\big)(\nabla\phi^{n+1}-\nabla\phi^{n})dx
      +2\Re\int_{\mathcal{O}}i\bar{\phi}^{n}\Delta(\Delta_{n} W)\Delta\phi^{n}dx\\
       &\quad  +\Re\int_{\mathcal{O}}i(\bar\phi^{n+1}-\bar\phi^{n})\Delta(\Delta_{n} W)\big(\theta\Delta\phi^{n+1}+(1-\theta)\Delta\phi^{n}\big)dx\\
       &\leq \frac18\|\nabla\phi^{n+1}-\nabla\phi^{n}\|_{\mathbb{L}^2}^{2}
      +8\|\nabla\big(\bar{\phi}^{n}\Delta(\Delta_{n} W)\big)\|_{\mathbb{L}^{2}}^{2}+K\frac{1}{\tau}\|\phi^{n+1}-\phi^{n}\|_{\mathbb{L}^2}^{4}\\
      &\quad
      +K\frac{1}{\tau}\|\Delta(\Delta_{n} W)\|_{\mathbb{L}^{\infty}}^{4}+\tau\Big(\|\Delta\phi^{n+1}\|_{\mathbb{L}^{2}}^{2}+\|\Delta\phi^{n}\|_{\mathbb{L}^{2}}^{2}\Big)
      +2\Re\int_{\mathcal{O}}i\bar{\phi}^{n}\Delta(\Delta_{n} W)\Delta\phi^{n}dx.
    \end{align*}
    By assertion ${\rm(i)}$ and ${\rm(ii)}$ of Lemma \ref{alg2:H1}, we get \[E(B_{b}^{1})\leq \frac18 E\|\nabla\phi^{n+1}-\nabla\phi^{n}\|_{\mathbb{L}^{2}}^{2}+K\tau+K\tau\Big(\|\Delta\phi^{n+1}\|_{\mathbb{L}^{2}}^{2}+\|\Delta\phi^{n}\|_{\mathbb{L}^{2}}^{2}\Big).\]

    For term $B_{c}^{1}$, we use again $\phi^{n+\frac12}=\phi^{n}+\frac{\phi^{n+1}-\phi^{n}}{2}$ to obtain
    \begin{align*}
      B_{c}^{1}=&4\theta\Re\int_{\mathcal{O}}i\nabla\bar{\phi}^{n+\frac12}\nabla(\Delta_{n} W)\Delta \phi^{n+1}dx
      +4(1-\theta)\Re\int_{\mathcal{O}}i\nabla\bar{\phi}^{n+\frac12}\nabla(\Delta_{n} W)\Delta \phi^{n}dx\\
      =&4\theta\Re\int_{\mathcal{O}}i\Big(\nabla\bar\phi^{n}+\frac{\nabla\bar\phi^{n+1}-\nabla\bar\phi^{n}}{2}\Big)\nabla(\Delta_{n} W)\Big((\Delta\phi^{n+1}-\Delta\phi^{n})+\Delta\phi^{n}\Big)dx\\
      &+4(1-\theta)\Re\int_{\mathcal{O}}i\Big(\nabla\bar\phi^{n}+\frac{\nabla\bar\phi^{n+1}-\nabla\bar\phi^{n}}{2}\Big)\nabla(\Delta_{n} W)\Delta\phi^{n}dx\, .
      \end{align*}
In the following step, we use that the Wiener process is ${\mathbb H}^2_0$-valued to allow for integration by parts,
      \begin{align*}
      =& 4\theta\Re\int_{\mathcal{O}}i\nabla\bar\phi^{n}\nabla(\Delta_{n} W)\Delta(\phi^{n+1}-\phi^{n})dx +2\Re\int_{\mathcal{O}}i\nabla(\bar\phi^{n+1}-\bar\phi^{n})\nabla(\Delta_{n} W)\Delta\phi^{n}dx\\
      &+4\Re\int_{\mathcal{O}}i\nabla\bar{\phi}^{n}\nabla(\Delta_{n} W)\Delta\phi^{n}dx
      +2\theta\Re\int_{\mathcal{O}}i(\nabla\bar{\phi}^{n+1}-\nabla\bar{\phi}^{n})\nabla(\Delta_{n} W)(\Delta \phi^{n+1}-\Delta\phi^{n})dx.
    \end{align*}
    Integration by parts for the first term then leads to
    \begin{align*}
      B_c^1=&-4\theta\Re\int_{\mathcal{O}}i\nabla\bar{\phi}^{n}\Delta(\Delta_{n} W)(\nabla \phi^{n+1}-\nabla\phi^{n})dx
      +2(1+2\theta)\Re\int_{\mathcal{O}}i(\nabla\bar{\phi}^{n+1}-\nabla\bar{\phi}^{n})\nabla(\Delta_{n} W)\Delta\phi^{n}dx
      \\
      &+4\Re\int_{\mathcal{O}}i\nabla\bar{\phi}^{n}\nabla(\Delta_{n} W)\Delta\phi^{n}dx
      +2\theta\Re\int_{\mathcal{O}}i(\nabla\bar{\phi}^{n+1}-\nabla\bar{\phi}^{n})\nabla(\Delta_{n} W)(\Delta \phi^{n+1}-\Delta\phi^{n})dx.
    \end{align*}

      We only present the estimate of the last term, the remainder terms can be easily bounded as before.
     \begin{align*}
       &2\theta\Re\int_{\mathcal{O}}i(\nabla\bar{\phi}^{n+1}-\nabla\bar{\phi}^{n})\nabla(\Delta_{n} W)(\Delta \phi^{n+1}-\Delta\phi^{n})dx\\
       &\leq2\|\Delta\phi^{n+1}-\Delta\phi^{n}\|_{\mathbb{L}^{2}}\|\nabla\phi^{n+1}-\nabla\phi^{n}\|_{\mathbb{L}^{2}}\|\nabla(\Delta_{n} W)\|_{\mathbb{L}^{\infty}}\\
       &\leq \frac{2\theta-1}{8}\|\Delta\phi^{n+1}-\Delta\phi^{n}\|_{\mathbb{L}^{2}}^2+\frac{8}{2\theta-1}\|\nabla\phi^{n+1}-\nabla\phi^{n}\|_{\mathbb{L}^{2}}^2\|\nabla(\Delta_{n} W)\|_{\mathbb{L}^{\infty}}^2\\
       &\leq \frac{2\theta-1}{8}\|\Delta\phi^{n+1}-\Delta\phi^{n}\|_{\mathbb{L}^{2}}^2+\frac18\|\nabla\phi^{n+1}-\nabla\phi^{n}\|_{\mathbb{L}^{2}}^2
       +\frac{32}{(2\theta-1)^2}\|\nabla\phi^{n+1}-\nabla\phi^{n}\|_{\mathbb{L}^{2}}^2\|\nabla(\Delta_{n} W)\|_{\mathbb{L}^{\infty}}^4\\
       &\leq \frac{2\theta-1}{8}\|\Delta\phi^{n+1}-\Delta\phi^{n}\|_{\mathbb{L}^{2}}^2+\frac18\|\nabla\phi^{n+1}-\nabla\phi^{n}\|_{\mathbb{L}^{2}}^2
       +K\tau\Big(\|\nabla\phi^{n+1}\|_{\mathbb{L}^2}^4+\|\nabla\phi^{n}\|_{\mathbb{L}^2}^4\Big)+\frac{1}{\tau(2\theta-1)^4}\|\nabla(\Delta_{n} W)\|_{\mathbb{L}^{\infty}}^8.
     \end{align*}
     Therefore, for {$2\theta-1\geq c\sqrt{\tau}$ with $c\geq c^{*}>0$} and since $E\|\nabla(\Delta_{n} W)\|_{\mathbb{L}^{\infty}}^8=O(\tau^4)$, by Lemma \ref{alg2:H1} ${\rm(i)}$ we obtain \[E(B_{c}^{1})\leq K\tau+\frac{2\theta-1}{8} E\|\Delta\phi^{n+1}-\Delta\phi^{n}\|_{\mathbb{L}^2}^{2}+\frac38\|\nabla\phi^{n+1}-\nabla\phi^{n}\|_{\mathbb{L}^{2}}^2+K\tau\|\Delta\phi^{n}\|_{\mathbb{L}^2}^{2}.\]
    {\em Step 3: Estimate of term $B^2$.}
    By integration by parts,
     \begin{align*}
     B^{2}
     =&\frac14\Re\int_{\mathcal{O}}i\nabla\Big((\bar{\phi}^{n+1}+\bar\phi^{n})\Delta_{n} W\Big)\nabla\Big((|\phi^{n+1}|^{2}+|\phi^{n}|^{2})(\phi^{n+1}+\phi^{n})\Big)dx\\
     =&\frac14\Re\int_{\mathcal{O}}i\Big(\nabla\bar\phi^{n+1}\Delta_{n} W+\nabla\bar\phi^{n}\Delta_{n} W+\bar\phi^{n+1}\nabla(\Delta_{n} W)+\bar\phi^{n}\nabla(\Delta_{n} W)\Big)
    \Big((\phi^{n+1})^2\nabla\bar\phi^{n+1}+2\nabla\phi^{n+1}|\bar\phi^{n+1}|^2\\
    &+\nabla\phi^{n+1}\bar\phi^{n+1}\phi^{n}
    +\phi^{n+1}\phi^{n}\nabla\bar{\phi}^{n+1}
     +|\phi^{n+1}|^2\nabla\phi^{n}+\nabla\phi^{n}\bar\phi^{n}\phi^{n+1}+2\nabla\phi^{n}|\phi^{n}|^2+\nabla\bar\phi^{n}\phi^{n}\phi^{n+1}\\
     &+\nabla\bar\phi^{n}(\phi^{n})^2
     +|\phi^n|^2\nabla\phi^{n+1}\Big)dx
     \end{align*}
     The estimates of these terms are done by inserting functions of $\phi^{n}$ and using the fact that  $E(\Delta_{n} W|\mathcal{F}_{t_{n}})=0$. So here we only present one troublesome term in $B^{2}$ as an example.
      \begin{align*}
        &\Re\int_{\mathcal{O}}i(\nabla\bar{\phi}^{n+1}\phi^{n+1})^{2}\Delta_{n} Wdx\\
        &=\Re\int_{\mathcal{O}}i\Big((\nabla\bar{\phi}^{n+1}\phi^{n+1})^{2}-(\nabla\bar{\phi}^{n}\phi^{n})^{2}\Big)\Delta_{n} Wdx
        + \Re\int_{\mathcal{O}}i(\nabla\bar{\phi}^{n}\phi^{n})^{2}\Delta_{n} Wdx.
      \end{align*}
      The expectation of the second term is zero. By the identity $a^2-b^2=(a+b)(a-b)$, we deal with the first term below.
      \begin{align}\label{g2}
        & \Re\int_{\mathcal{O}}i\Big((\nabla\bar{\phi}^{n+1}\phi^{n+1})^{2}-(\nabla\bar{\phi}^{n}\phi^{n})^{2}\Big)\Delta_{n} Wdx\nonumber\\
        &= \Re\int_{\mathcal{O}}i(\nabla\bar{\phi}^{n+1}\phi^{n+1}+\nabla\bar{\phi}^{n}\phi^{n})(\nabla\bar{\phi}^{n+1}\phi^{n+1}-\nabla\bar{\phi}^{n}\phi^{n})\Delta_{n} Wdx\nonumber\\
        &= \Re\int_{\mathcal{O}}i(\nabla\bar{\phi}^{n+1}\phi^{n+1}+\nabla\bar{\phi}^{n}\phi^{n})(\nabla\bar{\phi}^{n+1}-\nabla\bar{\phi}^{n})\phi^{n}\Delta_{n} Wdx\nonumber\\
        &\quad+\Re\int_{\mathcal{O}}i(\nabla\bar{\phi}^{n+1}\phi^{n+1}+\nabla\bar{\phi}^{n}\phi^{n})\nabla\bar{\phi}^{n+1}(\phi^{n+1}-\phi^{n})\Delta_{n} Wdx.
      \end{align}
      For the first term, we use $\mathbb{H}^1\hookrightarrow \mathbb{L}^{\infty}$ and Young's inequality to conclude
      \begin{align*}
       & \Re\int_{\mathcal{O}}i(\nabla\bar{\phi}^{n+1}\phi^{n+1}+\nabla\bar{\phi}^{n}\phi^{n})(\nabla\bar{\phi}^{n+1}-\nabla\bar{\phi}^{n})\phi^{n}\Delta_{n} Wdx\\
      &  \leq\Big(\|\nabla\phi^{n+1}\|_{\mathbb{L}^2}\|\phi^{n+1}\|_{\mathbb{L}^{\infty}}+\|\nabla\phi^{n}\|_{\mathbb{L}^2}\|\phi^{n}\|_{\mathbb{L}^{\infty}}\Big)\|\nabla(\phi^{n+1}-\phi^{n})\|_{\mathbb{L}^2}\|\phi^{n}\|_{\mathbb{L}^{\infty}}\|\Delta_{n} W\|_{\mathbb{L}^{\infty}}\\
        &\leq \frac18 \|\nabla(\phi^{n+1}-\phi^{n})\|_{\mathbb{L}^2}^{2}+\frac{1}{\tau}\|\Delta_{n} W\|_{\mathbb{L}^{\infty}}^{4}+K\tau\|\phi^{n+1}\|_{\mathbb{H}^{1}}^{8}\|\phi^{n}\|_{\mathbb{H}^{1}}^{4}+K\tau\|\phi^{n}\|_{\mathbb{H}^1}^{12},
      \end{align*}
      Similarly, by embedding $\mathbb{H}^1\hookrightarrow \mathbb{L}^{\infty}$ and H\"older inequality, we get the estimation of the second term in \eqref{g2},
      \begin{align*}
       & \Re\int_{\mathcal{O}}i(\nabla\bar{\phi}^{n+1}\phi^{n+1}+\nabla\bar{\phi}^{n}\phi^{n})\nabla\bar{\phi}^{n+1}(\phi^{n+1}-\phi^{n})\Delta_{n} Wdx\\
       &\leq \Big(\|\nabla\phi^{n+1}\|_{\mathbb{L}^2}\|\phi^{n+1}\|_{\mathbb{L}^{\infty}}+\|\nabla\phi^{n}\|_{\mathbb{L}^2}\|\phi^{n}\|_{\mathbb{L}^{\infty}}\Big)\|\nabla\phi^{n+1}\|_{\mathbb{L}^2}
       \|\phi^{n+1}-\phi^{n+1}\|_{\mathbb{L}^{\infty}}\|\Delta_{n} W\|_{\mathbb{L}^{\infty}}\\
        &\leq  \|\phi^{n+1}-\phi^{n}\|_{\mathbb{H}^1}^{2}+\frac{1}{\tau}\|\Delta_{n} W\|_{\mathbb{L}^{\infty}}^{4}+K\tau\|\phi^{n+1}\|_{\mathbb{H}^{1}}^{12}+K\tau\|\phi^{n}\|_{\mathbb{H}^1}^{12}.
      \end{align*}
      Therefore, from Lemma \ref{alg2:H1} ${\rm(i)}$ and ${\rm(ii)}$, we have \[E(B^2)\leq K\tau+K\tau \Big( E\|\Delta\phi^{n+1}\|_{\mathbb{L}^2}^2+E\|\Delta\phi^{n}\|_{\mathbb{L}^2}^2\Big)+\frac18E\|\nabla(\phi^{n+1}-\phi^{n})\|_{\mathbb{L}^2}^2.\]

      {\em Step 4: Estimate of term A.}
      Because of $(|a|^2+|b|^2)(a+b)=2|a|^2a+2|b|^2b-(|b|^2-|a|^2)(b-a)$ for $a,b\in\mathbb{C}$, we split term $A$ further into
      \begin{align}\label{eq0}
        A=&\Re\int_{\mathcal{O}}\big(|\phi^{n+1}|^2+|\phi^{n}|^2\big)\phi^{n+\frac12}(\Delta\bar\phi^{n+1}-\Delta\bar\phi^{n})dx\nonumber\\
        =&-\frac12\Re\int_{\mathcal{O}}(|\phi^{n+1}|^2-|\phi^{n}|^2)(\phi^{n+1}-\phi^{n})(\Delta\bar\phi^{n+1}-\Delta\bar\phi^{n})dx\nonumber\\
        &+\Re\int_{\mathcal{O}}|\phi^{n}|^2\phi^{n}(\Delta\bar\phi^{n+1}-\Delta\bar\phi^{n})dx
        +\Re\int_{\mathcal{O}}|\phi^{n+1}|^2\phi^{n+1}(\Delta\bar\phi^{n+1}-\Delta\bar\phi^{n})dx\nonumber\\
        =:&A^{1}+A^2+A^3.
      \end{align}
     We use the identity $|a|^2a-|b|^2b=|a|^2(a-b)+|b|^2(a-b)+ab(\bar{a}-\bar{b})$ for $a,b\in\mathbb{C}$ to
   rewrite term $A^2$ as
   \begin{align}\label{eq1}
     A^2=&\Re\int_{\mathcal{O}}\Delta\bar\phi^{n+1}|\phi^{n+1}|^2\phi^{n+1}dx-\Re\int_{\mathcal{O}}\Delta\bar\phi^{n}|\phi^{n}|^2\phi^{n}dx
     -\Re\int_{\mathcal{O}}\Delta\bar\phi^{n+1}\big(|\phi^{n+1}|^2\phi^{n+1}-|\phi^{n}|^2\phi^{n}\big)dx\nonumber\\
     =&\Re\int_{\mathcal{O}}\Delta\bar\phi^{n+1}|\phi^{n+1}|^2\phi^{n+1}dx-\Re\int_{\mathcal{O}}\Delta\bar\phi^{n}|\phi^{n}|^2\phi^{n}dx
     -\Re\int_{\mathcal{O}}\Delta\bar\phi^{n+1}|\phi^{n+1}|^2(\phi^{n+1}-\phi^{n})dx\nonumber\\
     &-\Re\int_{\mathcal{O}}\Delta\bar\phi^{n+1}|\phi^{n}|^2(\phi^{n+1}-\phi^{n})dx
     +\Re\int_{\mathcal{O}}\Delta\bar\phi^{n+1}\phi^{n+1}|\phi^{n+1}-\phi^{n}|^2dx
     -\Re\int_{\mathcal{O}}\Delta\bar\phi^{n+1}(\phi^{n+1})^2(\bar\phi^{n+1}-\bar\phi^{n})dx,
   \end{align}
   where for the last two terms in \eqref{eq1}, we use
   \[\Re\int_{\mathcal{O}}\Delta\bar\phi^{n+1}\phi^{n+1}\phi^{n}(\bar\phi^{n+1}-\bar\phi^{n})dx
   =-\Re\int_{\mathcal{O}}\Delta\bar\phi^{n+1}\phi^{n+1}|\phi^{n+1}-\phi^{n}|^2dx+\Re\int_{\mathcal{O}}\Delta\bar\phi^{n+1}(\phi^{n+1})^2(\bar\phi^{n+1}-\bar\phi^{n})dx.\]
We use integration by parts and product formula to rewrite term $A^3$.
\begin{align}\label{eq2}
  A^3=&\Re\int_{\mathcal{O}}(\bar\phi^{n+1}-\bar\phi^{n})\Delta(|\phi^{n+1}|^2\phi^{n+1})dx\nonumber\\
  =&2\Re\int_{\mathcal{O}}(\bar\phi^{n+1}-\bar\phi^{n})\Delta\phi^{n+1}|\phi^{n+1}|^2dx
  +\Re\int_{\mathcal{O}}(\bar\phi^{n+1}-\bar\phi^{n})(\phi^{n+1})^2\Delta\bar\phi^{n+1}dx\nonumber\\
  &+2\Re\int_{\mathcal{O}}(\bar\phi^{n+1}-\bar\phi^{n})(\nabla\phi^{n+1})^2\bar\phi^{n+1}dx
  +4\Re\int_{\mathcal{O}}(\bar\phi^{n+1}-\bar\phi^{n})|\nabla\phi^{n+1}|^2\phi^{n+1}dx.
\end{align}
Summing up \eqref{eq1} and \eqref{eq2} and $\Re(a)=\Re(\bar{a})$ for $a\in\mathbb{C}$ lead to
\begin{align}\label{eq3}
  A^2+A^3=&\Re\int_{\mathcal{O}}\Delta\bar\phi^{n+1}|\phi^{n+1}|^2\phi^{n+1}dx-\Re\int_{\mathcal{O}}\Delta\bar\phi^{n}|\phi^{n}|^2\phi^{n}dx\nonumber\\
  &+\Re\int_{\mathcal{O}}\Delta\bar\phi^{n+1}\big(|\phi^{n+1}|^2-|\phi^{n}|^2\big)(\phi^{n+1}-\phi^{n})dx
  +\Re\int_{\mathcal{O}}\Delta\bar\phi^{n+1}\phi^{n+1}|\phi^{n+1}-\phi^{n}|^2dx\nonumber\\
  &+2\Re\int_{\mathcal{O}}(\bar\phi^{n+1}-\bar\phi^{n})(\nabla\phi^{n+1})^2\bar\phi^{n+1}dx
  +4\Re\int_{\mathcal{O}}(\bar\phi^{n+1}-\bar\phi^{n})|\nabla\phi^{n+1}|^2\phi^{n+1}dx.
\end{align}
      Plugging equation \eqref{eq3} into \eqref{eq0}, one has
      \begin{align}
        A=&A^1+\Re\int_{\mathcal{O}}\Delta\bar\phi^{n+1}|\phi^{n+1}|^2\phi^{n+1}dx-\Re\int_{\mathcal{O}}\Delta\bar\phi^{n}|\phi^{n}|^2\phi^{n}dx\nonumber\\
  &+\Re\int_{\mathcal{O}}\Delta\bar\phi^{n+\frac12}\big(|\phi^{n+1}|^2-|\phi^{n}|^2\big)(\phi^{n+1}-\phi^{n})dx
  +\Re\int_{\mathcal{O}}\Delta\bar\phi^{n+1}\phi^{n+1}|\phi^{n+1}-\phi^{n}|^2dx\nonumber\\
  &+2\Re\int_{\mathcal{O}}(\bar\phi^{n+1}-\bar\phi^{n})(\nabla\phi^{n+1})^2\bar\phi^{n+1}dx
  +4\Re\int_{\mathcal{O}}(\bar\phi^{n+1}-\bar\phi^{n})|\nabla\phi^{n+1}|^2\phi^{n+1}dx\nonumber\\
  =:&A^1+A_{a,n+1}+A_{a,n}+A_{b}+A_{c}+A_{d}+A_{e}.
      \end{align}
      We estimate the terms separately.
       The estimation of the terms $A_{a,n+1}$ and $A_{a,n}$ follows from their special structure (when taking the sum with respect to $n$, all middle term are canceled) and Lemma \ref{alg2:H1}.
      For term $A_{b}$, we use binomial formula, and interpolation of $\mathbb{L}^4$ between $\mathbb{L}^2$ and $\mathbb{H}^1$ for $d=1$.
       \begin{align*}
      A_{b}&=\Re\int_{\mathcal{O}}\Delta\bar\phi^{n+\frac12}(|\phi^{n+1}|^2-|\phi^{n}|^2)(\phi^{n+1}-\phi^{n})dx\\
      &\leq 2\|\phi^{n+1}-\phi^{n}\|_{\mathbb{L}^4}^2\|\Delta\phi^{n+\frac12}\|_{\mathbb{L}^2}\|\phi^{n+\frac12}\|_{\mathbb{L}^{\infty}}\\
      &\leq \tau\Big(\|\Delta\phi^{n+1}\|_{\mathbb{L}^2}^2+\|\Delta\phi^{n}\|_{\mathbb{L}^2}^2\Big)+\frac{1}{\tau}\|\phi^{n+1}-\phi^{n}\|_{\mathbb{L}^4}^4\|\phi^{n+\frac12}\|_{\mathbb{L}^{\infty}}^2\\
      &\leq  \tau\Big(\|\Delta\phi^{n+1}\|_{\mathbb{L}^2}^2+\|\Delta\phi^{n}\|_{\mathbb{L}^2}^2\Big)+\frac{K}{\tau}\|\nabla(\phi^{n+1}-\phi^{n})\|_{\mathbb{L}^2}\|\phi^{n+1}-\phi^{n}\|_{\mathbb{L}^2}^3\|\phi^{n+\frac12}\|_{\mathbb{L}^{\infty}}^2\\
      &\leq \tau\Big(\|\Delta\phi^{n+1}\|_{\mathbb{L}^2}^2+\|\Delta\phi^{n}\|_{\mathbb{L}^2}^2\Big)+\frac{1}{8}\|\nabla(\phi^{n+1}-\phi^{n})\|_{\mathbb{L}^2}^2+K\tau\|\phi^{n+\frac12}\|_{\mathbb{L}^{\infty}}^8
      +\frac{1}{\tau^5}\|\phi^{n+1}-\phi^{n}\|_{\mathbb{L}^2}^{12}.
    \end{align*}
For term $A^1$, we use $|\phi^{n+1}|^2-|\phi^{n}|^2=2\Re\Big(\phi^{n+\frac12}(\phi^{n+1}-\phi^{n})\Big)$ and $\|\Delta(\phi^{n+1}-\Delta^{n})\|_{\mathbb{L}^2}\leq \|\Delta\phi^{n+1}\|_{\mathbb{L}^2}+\|\Delta\phi^{n}\|_{\mathbb{L}^2}$ to have
\begin{align*}
  A^1\leq& K\big(\|\Delta\phi^{n+1}\|_{\mathbb{L}^2}+\|\Delta \phi^{n}\|_{\mathbb{L}^2}\big)\|\phi^{n+1}-\phi^{n+1}\|_{\mathbb{L}^4}^2\|\phi^{n+\frac12}\|_{\mathbb{L}^{\infty}}\\
  \leq& K\tau\big(\|\Delta\phi^{n+1}\|_{\mathbb{L}^2}^2+\|\Delta \phi^{n}\|_{\mathbb{L}^2}^2\big)+K\frac{1}{\tau}\|\phi^{n+1}-\phi^{n+1}\|_{\mathbb{L}^4}^4\|\phi^{n+\frac12}\|_{\mathbb{L}^{\infty}}^2.
\end{align*}
Now follow the steps for $A_b$ to estimate the right-hand side.
      In order to bound the term $A_{c}$, we use once more the interpolation result for $\mathbb{L}^4$ which holds for $d=1$.
      \begin{align*}
        A_{c}&=\Re\int_{\mathcal{O}}\Delta\bar{\phi}^{n+1}\phi^{n+1}|\phi^{n+1}-\phi^{n}|^{2}dx\\
            &\leq \|\Delta\phi^{n+1}\|_{\mathbb{L}^2}\|\phi^{n+1}\|_{\mathbb{L}^{\infty}}\|\phi^{n+1}-\phi^{n}\|_{\mathbb{L}^4}^2\\
            &\leq K\tau\|\Delta\phi^{n+1}\|_{\mathbb{L}^2}^2+K\frac{1}{\tau}\|\phi^{n+1}\|_{\mathbb{L}^{\infty}}^2\|\nabla\phi^{n+1}-\nabla\phi^{n}\|_{\mathbb{L}^2}
            \|\phi^{n+1}-\phi^{n}\|_{\mathbb{L}^2}^3\\
            &\leq K\tau\|\Delta\phi^{n+1}\|_{\mathbb{L}^2}^2+\frac18\|\nabla(\phi^{n+1}-\phi^{n})\|_{\mathbb{L}^2}^2+K\tau\|\phi^{n+1}\|_{\mathbb{L}^{\infty}}^8+
            K\frac{1}{\tau^5}\|\phi^{n+1}-\phi^{n}\|_{\mathbb{L}^2}^{12}.
      \end{align*}
      For the last two terms $A_{d}+A_{e}$, we replace the expression  $\bar{\phi}^{n+1}-\bar{\phi}^{n}=-i\tau\big(\theta\Delta\bar{\phi}^{n+1}+(1-\theta)\Delta\bar{\phi}^{n}\big)+\frac{i}{2}\tau(|\phi^{n+1}|^{2}+|\phi^{n}|^{2})\bar{\phi}^{n+\frac12}
      +i\bar{\phi}^{n+\frac12}\Delta_{n} W$, then for the second term and third terms of the resulting equality, we can estimate them as before.

      Here by the interpolation of $\mathbb{L}^4$ between $\mathbb{H}^1$ and $\mathbb{L}^2$, and the continuous embedding $\mathbb{H}^1\hookrightarrow \mathbb{L}^{\infty}$, we estimate the first term of resulting equality after replacing $\bar{\phi}^{n+1}-\bar{\phi}^{n}$ into $A_d$,
      \begin{align*}
      &2\tau\Re\int_{\mathcal{O}}(-i)(\nabla\phi^{n+1})^{2}\bar{\phi}^{n+1}
     \big(\theta\Delta\bar{\phi}^{n+1}+(1-\theta)\Delta\bar{\phi}^{n}\big)dx\\
    &\quad+ 4\tau \Re\int_{\mathcal{O}}(-i)\phi^{n+1}|\nabla\phi^{n+1}|^2\big(\theta\Delta\bar{\phi}^{n+1}
 +(1-\theta)\Delta\bar{\phi}^{n}\big)dx\\
     &\leq 6\tau\Big(\|\Delta\phi^{n+1}\|_{\mathbb{L}^2}+\|\Delta\phi^{n}\|_{\mathbb{L}^2}\Big)\|\nabla\phi^{n+1}\|_{\mathbb{L}^4}^2\|\phi^{n+1}\|_{\mathbb{L}^{\infty}}\\
     &\leq K\tau\Big(\|\Delta\phi^{n+1}\|_{\mathbb{L}^2}^2+\|\Delta\phi^{n}\|_{\mathbb{L}^2}^2\Big)+K\tau\|\nabla\phi^{n+1}\|_{\mathbb{L}^2}^5\|\Delta\phi^{n+1}\|_{\mathbb{L}^2}\\
     &\leq K\tau\Big(\|\Delta\phi^{n+1}\|_{\mathbb{L}^2}^2+\|\Delta\phi^{n}\|_{\mathbb{L}^2}^2\Big)+K\tau\|\nabla\phi^{n+1}\|_{\mathbb{L}^2}^{10}.
     \end{align*}
     As a consequence, all terms on the right-hand side of \eqref{eq7} may be controlled with the help of Lemma \ref{alg2:H1} and a Gronwall's argument, apart from the term $\|\nabla(\phi^{n+1}-\phi^{n})\|_{\mathbb{L}^2}$.

     {\em Step 5: Estimate of the term $\|\nabla(\phi^{n+1}-\phi^{n})\|_{\mathbb{L}^2}$.}
We formally test equation \eqref{algorithm 2} with $-\Delta(\bar\phi^{n+1}-\bar\phi^{n})$ and take the imaginary part. We repeatedly use properties of the imaginary part of a complex number to obtain
     \begin{align*}
      \|\nabla(\phi^{n+1}-\phi^{n})\|_{\mathbb{L}^2}^2=&\tau\Im\int_{\mathcal{O}}\Delta\phi^{n+1}\Delta\bar\phi^{n}dx-\frac{\tau}{2}\Im
      \int_{\mathcal{O}}\big(|\phi^{n+1}|^2+|\phi^{n}|^2\big)\phi^{n+\frac12}(\Delta\bar\phi^{n+1}-\Delta\bar\phi^{n})dx\\
      &+\Im\int_{\mathcal{O}}\nabla(\phi^{n}\Delta_{n} W)(\nabla\bar\phi^{n+1}-\nabla\bar\phi^{n})dx
      -\frac12\Im\int_{\mathcal{O}}(\phi^{n+1}-\phi^{n})\Delta_{n} W(\Delta\bar\phi^{n+1}-\Delta\bar\phi^{n})dx\\
      \leq & K\tau\Big(\|\Delta\phi^{n+1}\|_{\mathbb{L}^2}^2+\|\Delta\phi^{n}\|_{\mathbb{L}^2}^2\Big)+\frac{\tau}{16}\|\Delta(\phi^{n+1}-\phi^{n})\|_{\mathbb{L}^2}^2+K\tau\Big(\|\phi^{n+1}\|_{\mathbb{L}^6}^6+\|\phi^{n}\|_{\mathbb{L}^6}^6\Big)\\
&      +\frac12\|\nabla(\phi^{n+1}-\phi^{n})\|_{\mathbb{L}^2}^2+K\|\nabla(\phi^{n}\Delta_{n} W)\|_{\mathbb{L}^2}^2
      +\frac{2\theta-1}{16}\|\Delta(\phi^{n+1}-\phi^{n})\|_{\mathbb{L}^2}^2\\
      &+\frac{1}{2\theta-1}\|\phi^{n+1}-\phi^{n}\|_{\mathbb{L}^2}^4+\frac{1}{2\theta-1}\|\Delta_{n} W\|_{\mathbb{L}^{\infty}}^4.
     \end{align*}
      By the continuous embedding $\mathbb{H}^1\hookrightarrow \mathbb{L}^6$, and Lemma \ref{alg2:H1}  (${\rm i}$), the term $\|\phi^{n+1}\|_{\mathbb{L}^6}^6+\|\phi^{n}\|_{\mathbb{L}^6}^6$ can be bounded. Other terms can be bounded by assertions (${\rm i}$) and (${\rm ii}$) of Lemma  \ref{alg2:H1}.
     Therefore
     \begin{align*}
      E \|\nabla(\phi^{n+1}-\phi^{n})\|_{\mathbb{L}^2}^2\leq& K\Big(\tau+\frac{\tau^2}{2\theta-1}\Big)+K\tau \Big(E\|\Delta\phi^{n+1}\|_{\mathbb{L}^2}^2+ E\|\Delta\phi^{n}\|_{\mathbb{L}^2}^2\Big)\\
      &+\Big(\frac{2\theta-1}{16}+\frac{\tau}{16}\Big)E\|\Delta(\phi^{n+1}-\phi^{n})\|_{\mathbb{L}^2}^2.
     \end{align*}
     {\em Step 6: Gronwall argument.}
We may combine these estimates for the terms on the right-hand side of \eqref{eq7}. For $\tau\leq \tau^{*}$ sufficiently small, we prove the assertion ${\rm (i)}$ to benefit from Gronwall's inequality and Lemma \ref{alg2:H1}.

      The proof of assertion ${\rm(ii)}$ is similar to Lemma \ref{alg2:H1} ${\rm(i)}$. Property {\rm(ii)} then allow to validate assertion ${\rm(iii)}$.The proof of assertion ${\rm(iv)}$ is similar to Lemma \ref{H1 E max}.
      \end{proof}

\begin{remark}\label{re2}
To derive uniform bounds in higher norms for iterates of Algorithm 1 is a bit more complicated than for the continuous problem (Lemma \ref{spatial H2 new}). Terms $A^1$, $A^{b}$ to $A^{e}$ can only be estimated in $1D$.
%
%
\end{remark}

Since we get a better estimate for $\|\nabla(\bar{\phi}^{n+1}-\bar{\phi}^{n})\|_{\mathbb{L}^2}$ in Lemma \ref{alg2:H2}, we can get a better conservation of the $\mathbb{L}^2$-norm for domains $\mathcal{O}\subset\mathbb{R}^{1}$; in fact, the next lemma asserts that the conservation of the $\mathbb{L}^2$-norm is of
order $\frac{1}{2}$ for {$2\theta-1=c\sqrt{\tau}$ with $c\geq c^{*}>0$}.
 \begin{lemma}\label{alg2:L2}
 Let $\mathcal{O}\subset\mathbb{R}^1$,  $T \equiv t_{M}>0$ be fixed, and {$\theta\in[\frac12+c\sqrt{\tau},1]$ with $c\geq c^{*}>0$}. There exist a constant {$K\equiv K(T,c^{*})>0$} and $\tau^* \equiv \tau^{*}(
   \Vert \phi^0 \Vert_{ \mathbb{H}_0^1\cap\mathbb{H}^2},T)$ such that for all $\tau\leq \tau^{*}$ holds
    \begin{equation}
      \max_{1\leq n\leq M}E\|\phi^{n}\|_{\mathbb{L}^2}^{2}-E\|\phi^{0}\|_{\mathbb{L}^2}^{2}\leq K(2\theta-1)\tau^{\frac12}.
    \end{equation}
  \end{lemma}
  \begin{proof}
   Recall \eqref{E1}, but now scale factors differently.
    \begin{align*}
      \|\phi^{n+1}\|_{\mathbb{L}^2}^{2}-\|\phi^{n}\|_{\mathbb{L}^2}^{2}
      &=(1-2\theta)\tau\Im\int_{\mathcal{O}}(\nabla\bar{\phi}^{n+1}-\nabla\bar{\phi}^{n})\nabla\phi^{n}dx\\
      &\leq (2\theta-1)\tau^{\frac12}\Big(\|\nabla\phi^{n+1}-\nabla\phi^{n}\|_{\mathbb{L}^2})(\tau^{\frac12}\|\nabla\phi^{n}\|_{\mathbb{L}^2}\Big)\\
      &\leq \frac{(2\theta-1) \tau^{\frac12}}{2}\|\nabla\bar{\phi}^{n+1}-\nabla\bar{\phi}^{n}\|_{\mathbb{L}^2}^{2}+\frac{(2\theta-1)\tau^{\frac32}}{2}\|\nabla\bar\phi^{n}\|_{\mathbb{L}^2}^{2}.
    \end{align*}
    Now consider the above inequality for some $0\leq \ell\leq M$, sum over the index from $\ell=0$ to $n$, take the expectation, and use Lemma \ref{alg2:H1} ${\rm(i)}$ and Lemma \ref{alg2:H2} to establish the assertion.
  \end{proof}

\section{Rates of convergence for the $\theta$-scheme}\label{sec4}
Let $e^{n}:=\psi(t_n)-\phi^{n}$, where $\psi$ solves \eqref{variational} and $\{\phi^{n}\}$ solves Algorithm \ref{alg_theta}.
  The error equation then reads for all $n\geq 0$,
  \begin{align}\label{error equation 1}
    &i\int_{\mathcal{O}}(e^{n+1}-e^{n})zdx-\int_{t_{n}}^{t_{n+1}}\int_{\mathcal{O}}\big(\nabla\psi(s)-\theta\nabla\phi^{n+1}-(1-\theta)\nabla\phi^{n}\big)\nabla zdxds\\
   & -\int_{t_{n}}^{t_{n+1}}\int_{\mathcal{O}}(|\psi(s)|^{2}\psi(s)-\frac12(|\phi^{n+1}|^{2}+|\phi^{n}|^{2})\phi^{n+\frac12})zdxds
   =\int_{t_{n}}^{t_{n+1}}\int_{\mathcal{O}}(\psi(s)-\phi^{n})zdxdW(s)\nonumber\\
   &-\frac{i}{2}\int_{t_{n}}^{t_{n+1}}\int_{\mathcal{O}}\psi(s)F_{\bf Q}
   zdxds-\frac12\int_{\mathcal{O}}(\phi^{n+1}-\phi^{n})z\Delta W_{n}dx\qquad \forall z\in {\mathbb H}_0^1.\nonumber
  \end{align}

  The following theorem states strong rates of convergence  for  the $\theta$-scheme for initial data $\psi_0\in L^{8}(\Omega;{\mathbb H}_0^1\cap{\mathbb H}^2)$,
 $\mathcal{O}\subset \mathbb{R}^1$ and {$\theta\in[\frac12+c\sqrt{\tau},1]$ with $c\geq c^{*}>0$}. Since its proof requires
properties which are stated in Lemma~\ref{alg2:H2}, we again consider ${\mathbb H}_0^2\cap{\mathbb H}^3$-valued driving Wiener processes.
\begin{theorem}\label{thm1}
Consider $\mathcal{O}\subset\mathbb{R}^{1}$, $T\equiv t_{M}>0$ and {$\theta\in[\frac12+c\sqrt{\tau},1]$ with $c\geq c^{*}>0$}. Let $\{\psi(t);\;0\leq t\leq T\}$ be the  solution of equation \eqref{sdd1} with $\lambda=-1$,
$\psi_0\in L^{8}(\Omega;{\mathbb H}_0^1\cap{\mathbb H}^2)$, and driving $ {\mathbb H}_0^2\cap{\mathbb H}^3$-valued Wiener process $W$. Let $\{\phi^{n};\, 0\leq n\leq M\}$ solve \eqref{algorithm 2}.
 Then there exist a constant {$K\equiv K(T,c^{*})>0$} and $\tau^* \equiv \tau^*(\Vert \psi_0\Vert_{L^8(\Omega;{\mathbb H}_0^1\cap{\mathbb H}^2)},T) >0$ such that for every $0<\tau\leq \tau^{*}$, we have
\begin{equation*}
   E\Big(\bm{1}_{\tilde{\Omega}_{\kappa}}\max_{0\leq n\leq M}\|e^n\|_{\mathbb{L}^2}^2\Big)
   \leq Ke^{K\kappa}\tau
 \end{equation*}
  for any fixed $\kappa >0$, and
  \[\tilde{\Omega}_{\kappa}:=\tilde{\Omega}_{\kappa,M}=\Big\{\omega\in\Omega\Big|\Big(\sup_{0\leq t
  \leq t_{M}}\|\psi(t)\|_{\mathbb{H}^1}^2+\max_{0\leq l\leq M}\|\phi^{l}\|_{\mathbb{H}^{1}}^{2}\Big)\leq \kappa\Big\}.\]
\end{theorem}
\medskip

{Let  $\kappa=K^{-1}\log(\tau^{-\varepsilon})$ for some $\varepsilon >0$.} We may employ stability properties of both $\psi$ and $\{\phi^{n}\}$ to conclude
 \begin{equation}\label{Y}
 \lim_{\tau\rightarrow 0}P(\tilde{\Omega}_{\kappa})=1.
 \end{equation}
Then Theorem \ref{thm1} amounts to
 \begin{equation*}
   E\Big(\bm{1}_{\tilde{\Omega}_{\kappa}}\max_{0\leq n\leq M}\|e^n\|_{\mathbb{L}^2}^2\Big)
   \leq K\tau^{1-\varepsilon}.
 \end{equation*}
 For the subset $\tilde{\Omega}_{\kappa}$, by Corollary \ref{spatial H1} and Lemma \ref{H1 E max}, there holds ($\tau<1$)
 \begin{align*}
   P(\tilde{\Omega}_{\kappa})\geq 1-\frac{E\Big(\sup_{t\in[0,T]}\|\psi(t)\|_{{\mathbb H}^1}^2\Big)+E\Big(\max_{0\leq n\leq M}\|\phi^{n}\|_{{\mathbb H}^1}^2\Big)}{K^{-1}\log(\tau^{-\varepsilon})}\geq 1+\frac{1}{\tilde{\varepsilon}\log(\tau)}\,,
 \end{align*}
 for $\tilde{\varepsilon}=\varepsilon\Big[ K\Big(E(\sup_{t\in[0,T]}\|\psi(t)\|_{{\mathbb H}^1}^2)+E(\max_{0\leq n\leq M}\|\phi^{n}\|_{{\mathbb H}^1}^2\big)\Big)\Big]^{-1}$. Therefore, \eqref{Y} is valid.

 A consequence of Theorem \ref{thm1} is  convergence with rates in probability sense for iterates of the scheme.
 For every $\alpha <\frac12$ and $C>0$, we estimate
 \begin{align*}
   P\Big[\max_{0\leq n\leq M}\|e^{n}\|_{{\mathbb L}^2}\geq C\tau^{\alpha}\Big]
   &\leq P\Big[\Big\{\max_{0\leq n\leq M}\|e^{n}\|_{{\mathbb L}^2}\geq C\tau^{\alpha}\Big\}\cap \tilde{\Omega}_{\kappa}\Big]+P[\Omega\setminus \tilde{\Omega}_{\kappa}]\\
   &\leq \frac{K\tau}{C^{2}\tau^{2\alpha}}-\frac{1}{\varepsilon \log{\tau}}.
 \end{align*}
 Therefore, we obtain the following corollary.
\begin{corollary}\label{probability}
There exists a constant $C>0$ such that for all $\alpha <\frac12$,
\begin{align*}
  \lim_{\tau\rightarrow 0}P\Big[\max_{0\leq n\leq M}\|\psi(t_{n})-\phi^{n}\|_{\mathbb{L}^2}\geq C\tau^{\alpha}\Big]=0.
\end{align*}
\end{corollary}
The constant $C>0$ used in this corollary may be determined from the constant $K>0$ in Theorem \ref{thm1}. \\

\begin{proof}(\textit{of Theorem \ref{thm1}})
We test equation \eqref{error equation 1} with $z=\bar{e}^{n+1}$, and take the imaginary part.
    In below, we address the three terms on the left-hand sides resp. the three terms on the right-hand side independently.

     \textit{LHS (first term I).}
     Because of the identity $\Re\Big(a(\bar a-\bar b)\Big)=\frac12\Big(|a|^2-|b|^2+|a-b|^2\Big)$ for $a,b\in{\mathbb C}$,
    we have
    \begin{align}
      I=\Im\Big(i\int_{\mathcal{O}}(e^{n+1}-e^{n})\bar{e}^{n+1}dx\Big)=\frac12 (\|e^{n+1}\|_{{\mathbb L}^2}^{2}-\|e^{n}\|_{{\mathbb L}^2}^{2}+\|e^{n+1}-e^{n}\|_{{\mathbb L}^2}^{2}).
    \end{align}

      \textit{LHS (second term II).}
    We decompose the negative of term $II$ as follows,
    \begin{align}
      -II&=\Im\int_{t_{n}}^{t_{n+1}}\int_{\mathcal{O}}\big(\nabla\psi(s)-\theta\nabla\phi^{n+1}-(1-\theta)\nabla\phi^{n}\big)\nabla\bar{e}^{n+1}dxds\nonumber\\
      &=\theta\Im\int_{t_{n}}^{t_{n+1}}\int_{\mathcal{O}}\big(\nabla\psi(s)-\nabla\psi(t_{n+1})\big)\nabla\bar{e}^{n+1}dxds
      +(1-\theta)\Im\int_{t_{n}}^{t_{n+1}}\int_{\mathcal{O}}\big(\nabla\psi(s)-\nabla\psi(t_{n})\big)\nabla\bar{e}^{n+1}dxds\nonumber\\
      &\quad+\tau\Im\int_{\mathcal{O}}\big(\theta\nabla e^{n+1}+(1-\theta)\nabla e^{n}\big)\nabla\bar{e}^{n+1}dx
      \nonumber\\
      &=II^1+II^2+II^3.
    \end{align}
    The estimates of terms $II^1$ and $II^2$ are similar, we use integration by parts and equation \eqref{S_NLSE_Ito}. Taking $II^1$ as an example,
    we know that
    \begin{align*}
      II^{1}
      &=-\theta\Im\int_{t_{n}}^{t_{n+1}}\int_{\mathcal{O}}\Delta\bar{e}^{n+1}\int_{t_{n+1}}^{s}\Big(i\Delta\psi(\nu)-i|\psi(\nu)|^{2}\psi(\nu)
      -\frac12 \psi(\nu)F_{\mathbf Q}\Big)d\nu dxds\nonumber\\
      &\quad-\theta\Im\int_{t_{n}}^{t_{n+1}}\int_{\mathcal{O}}\Delta\bar{e}^{n+1}\int_{t_{n+1}}^{s}i\psi(\nu)dW(\nu)dxds\\
      &=II^1_a+II^1_b.
    \end{align*}
    We use the embedding  ${\mathbb H}^1\hookrightarrow {\mathbb L}^6$ and the stability of solution $\{\psi(t);\,t\in[0,T]\}$ and iterates $\{\phi^n;\,n=0,1,\cdots,M\}$; i.e., Corollary \ref{spatial H1}, Lemma \ref{spatial H2 new} and Lemma \ref{alg2:H2} to obtain
    \begin{align*}
      II^1_a &\leq \int_{t_{n}}^{t_{n+1}}\int_{t_{n+1}}^{s}\|\Delta\bar{e}^{n+1}\|_{{\mathbb L}^2}\Big(\|\Delta\psi(\nu)\|_{{\mathbb L}^2}+\|\psi(\nu)\|_{{\mathbb H}^1}^3+\|\psi(\nu)\|_{{\mathbb L}^2}\|F_{\mathbf Q}\|_{{\mathbb L}^{\infty}}\Big)d\nu ds\\
      &\leq K\tau^2\Big(\|\Delta\psi(t_{n+1})\|_{{\mathbb L}^2}^{2}+\|\Delta\phi^{n+1}\|_{{\mathbb L}^2}^{2}\Big)+K\int_{t_{n}}^{t_{n+1}}\int_{t_{n+1}}^{s}\Big(\|\psi(\nu)\|_{{\mathbb H}^2}^{2}+\|\psi(\nu)\|_{{\mathbb H}^1}^{6}+\|\psi(\nu)\|_{{\mathbb L}^2}^{2}\Big)d\nu ds,
    \end{align*}
    where we use $\|\Delta e^{n+1}\|_{{\mathbb L}^2}\leq \|\Delta\psi(t_{n+1})\|_{{\mathbb L}^2}+\|\Delta\phi^{n+1}\|_{{\mathbb L}^2}$.
For the estimate of the term $II^1_b$, we use integration by parts twice and Young's inequality to get
    \begin{align*}
      II^1_b
      &\leq K\int_{t_n}^{t_{n+1}}\left[\| e^{n+1}\|_{{\mathbb L}^2}^{2}+\left\|\int_{t_{n+1}}^s \psi(\nu)dW(\nu)\right\|_{{\mathbb H}^2}^{2}\right]ds\\
      &= K\tau \| e^{n+1}\|_{{\mathbb L}^2}^{2}+K\int_{t_n}^{t_{n+1}}\left\|\int_{t_{n+1}}^s \psi(\nu)dW(\nu)\right\|_{{\mathbb H}^2}^{2}ds.
    \end{align*}

    Using a property of complex numbers, integration by parts and the triangle inequality we get
\begin{align*}
  II^3&=\tau(1-\theta)\Im \int_{\mathcal{O}}\Delta e^{n}(\bar e^{n+1}-\bar e^{n})dx\leq K\tau\|\Delta e^n\|_{{\mathbb L}^2}\|e^{n+1}-e^{n}\|_{{\mathbb L}^2}\\
&  \leq \frac{1}{16}\|e^{n+1}-e^{n}\|_{{\mathbb L}^2}^2+K\tau^2\Big(\|\Delta\psi(t_n)\|_{{\mathbb L}^2}^2+\|\Delta \phi^n\|_{{\mathbb L}^2}^2\Big).
\end{align*}

     \textit{LHS (third term III).}
The negative of the term $III$ is
    \begin{align*}
      III
      &=\Im\int_{t_{n}}^{t_{n+1}}\int_{\mathcal{O}}\Big(|\psi(s)|^{2}\psi(s)-|\psi(t_{n})|^{2}\psi(t_{n})\Big)\bar{e}^{n+1}dxds
     -\frac12\Im\int_{t_{n}}^{t_{n+1}}\int_{\mathcal{O}}(|\phi^{n+1}|^{2}-|\phi^{n}|^{2})\phi^{n+\frac12}\bar{e}^{n+1}dxds\\
      &\quad-\frac12\Im\int_{t_{n}}^{t_{n+1}}\int_{\mathcal{O}}|\phi^{n}|^{2}(\phi^{n+1}-\phi^{n})\bar{e}^{n+1}dxds
      +\Im\int_{t_{n}}^{t_{n+1}}\int_{\mathcal{O}}\Big(|\psi(t_{n})|^{2}\psi(t_{n})-|\phi^{n}|^{2}\phi^{n}\Big)\bar{e}^{n+1}dxds\\
      &=III^1+III^2+III^3+III^4.
    \end{align*}
    The estimations of terms $III^1$, $III^2$ and $III^3$ in the above equality are similar, using Lemmas \ref{temporal L2} and \ref{alg2:H2}, and Sobolev embeddings. Below we only present the estimate of the first term in the above equality.
We benefit from the identity $|a|^2a-|b|^2b=|a|^2(a-b)+|b|^2(a-b)+ab(\bar a-\bar b)$ for $a,b\in{\mathbb C}$ to obtain
    \begin{align*}
      III^{1}
      &=\Im\int_{t_{n}}^{t_{n+1}}\int_{\mathcal{O}}|\psi(s)|^{2}\big(\psi(s)-\psi(t_{n})\big)\bar{e}^{n+1}dxds
      +\Im\int_{t_{n}}^{t_{n+1}}\int_{\mathcal{O}}|\psi(t_{n})|^{2}\big(\psi(s)-\psi(t_{n})\big)\bar{e}^{n+1}dxds\\
      &\quad+\Im\int_{t_{n}}^{t_{n+1}}\int_{\mathcal{O}}\psi(s)\psi(t_{n})\big(\bar{\psi}(s)-\bar{\psi}(t_{n})\big)\bar{e}^{n+1}dxds.
      \end{align*}
      By the continuous embedding ${\mathbb H}^{1}\hookrightarrow {\mathbb L}^{\infty}$ for $d=1$, we may conclude that
      \begin{align*}
      III^{1}&\leq K\tau\|e^{n+1}\|_{{\mathbb L}^{2}}^{2}+K\int_{t_{n}}^{t_{n+1}}\|\psi(s)\|_{{\mathbb H}^{1}}^{4}\|\psi(s)-\psi(t_{n})\|_{{\mathbb L}^{2}}^{2}ds
      +K\int_{t_{n}}^{t_{n+1}}\|\psi(t_{n})\|_{{\mathbb H}^{1}}^{4}\|\psi(s)-\psi(t_{n})\|_{{\mathbb L}^{2}}^{2}ds\\
      &\leq K\tau\|e^{n+1}\|_{{\mathbb L}^2}^{2}+K\tau\int_{t_{n}}^{t_{n+1}}\|\psi(s)\|_{{\mathbb H}^{1}}^{8}ds+K\tau^{2}\|\psi(t_{n})\|_{{\mathbb H}^{1}}^{8}
      +K\frac{1}{\tau}\int_{t_{n}}^{t_{n+1}}\|\psi(s)-\psi(t_{n})\|_{{\mathbb L}^2}^{4}ds.
    \end{align*}
    The estimation of $III^{2}$ and $III^{3}$ are similar as that of $III^{1}$.
    So we have
    \begin{align*}
      III^{2}+III^{3}\leq K\tau\|e^{n+1}\|_{{\mathbb L}^2}^{2}+K\tau^{2}\|\phi^{n+1}\|_{{\mathbb H}^{1}}^{8}+K\tau^{2}\|\phi^{n}\|_{{\mathbb H}^{1}}^{8}
      +K\|\phi^{n+1}-\phi^{n}\|_{{\mathbb L}^2}^{4}.
    \end{align*}
    For term $III^{4}$ we use again the identity $|a|^2a-|b|^2b=|a|^2(a-b)+|b|^2(a-b)+ab(\bar a-\bar b)$, for $a,b\in{\mathbb C}$ to have
    \begin{align*}
      III^{4}
      &=\tau\Im\int_{\mathcal{O}}\Big(|\psi(t_{n})|^{2}e^{n}\bar{e}^{n+1}+|\phi^{n}|^{2}e^{n}\bar{e}^{n+1}
      +\psi(t_{n})\phi^{n}\bar{e}^{n}\bar{e}^{n+1}\Big)dx\\
      &\leq K\tau\|\psi(t_{n})\|_{{\mathbb H}^{1}}^{2}\|e^{n}\|_{{\mathbb L}^2}\|e^{n+1}\|_{{\mathbb L}^2}+K\tau\|\phi^{n}\|_{{\mathbb H}^{1}}^{2}\|e^{n}\|_{{\mathbb L}^2}\|e^{n+1}\|_{{\mathbb L}^2}\\
      &\leq K\tau\big(\|\psi(t_{n})\|_{{\mathbb H}^{1}}^{2}+\|\phi^{n}\|_{{\mathbb H}^{1}}^{2}\big)\|e^{n}\|_{{\mathbb L}^2}^{2}+K\tau\big(\|\psi(t_{n})\|_{{\mathbb H}^{1}}^{2}
      +\|\phi^{n}\|_{{\mathbb H}^{1}}^{2}\big)\|e^{n+1}\|_{{\mathbb L}^2}^{2}.
    \end{align*}
  \textit{RHS (first term IV).}
  By writing $\bar{e}^{n+1}=\big(\bar{e}^{n+1}-\bar{e}^{n}\big)+\bar{e}^{n}$, we have
    \begin{align*}
      IV&=\Im\int_{t_{n}}^{t_{n+1}}\int_{\mathcal{O}}(\psi(s)-\phi^{n})(\bar{e}^{n+1}-\bar{e}^{n})dxdW(s)
      +\Im\int_{t_{n}}^{t_{n+1}}\int_{\mathcal{O}}(\psi(s)-\phi^{n})\bar{e}^{n}dxdW(s)\\
      &=\Im\int_{t_{n}}^{t_{n+1}}\int_{\mathcal{O}}(\psi(s)-\psi(t_{n}))(\bar{e}^{n+1}-\bar{e}^{n})dxdW(s)+\Im\int_{t_{n}}^{t_{n+1}}\int_{\mathcal{O}}e^{n}(\bar{e}^{n+1}-\bar{e}^{n})dxdW(s)\\
      &\quad +\Im\int_{t_{n}}^{t_{n+1}}\int_{\mathcal{O}}(\psi(s)-\psi(t_{n}))\bar{e}^{n}dxdW(s)\\
      &=:IV^{1}+IV^{2}+IV^{3}.
    \end{align*}
    For term $IV^{1}$, via Fubini theorem we have
    \begin{align*}
      IV^{1}&=\Im\int_{\mathcal{O}}(\bar{e}^{n+1}-\bar{e}^{n})\int_{t_{n}}^{t_{n+1}}(\psi(s)-\psi(t_{n}))dW(s)dx\\
      &\leq \frac{1}{16}\|e^{n+1}-e^{n}\|_{{\mathbb L}^2}^{2}+K\left\|\int_{t_{n}}^{t_{n+1}}\big(\psi(s)-\psi(t_{n})\big)dW(s)\right\|_{{\mathbb L}^2}^{2}.
    \end{align*}
    For term $IV^{2}$, we have
    \begin{align*}
      IV^{2}=\Im\int_{\mathcal{O}}e^{n}(\bar{e}^{n+1}-\bar{e}^{n})\Delta_{n} Wdx
      \leq\frac{1}{16}\|e^{n+1}-e^{n}\|_{{\mathbb L}^{2}}^{2}+K\|e^{n}\|_{{\mathbb L}^2}^{2}\|\Delta_{n} W\|_{{\mathbb L}^{\infty}}^{2}.
    \end{align*}
     \textit{RHS (second and third terms V).} We insert the equation for $\phi^{n+1}-\phi^{n}$ to get
    \begin{align*}
      V
      &=-\frac12\Re\int_{t_{n}}^{t_{n+1}}\int_{\mathcal{O}}\psi(s)F_{\mathbf Q}\bar{e}^{n+1}dxds\\
     &\quad -\frac12\Im\int_{\mathcal{O}}
      \Big[i\tau\big(\theta\Delta\phi^{n+1}+(1-\theta)\Delta\phi^{n}\big)-i\frac{\tau}{2}(|\phi^{n+1}|^{2}+|\phi^{n}|^{2})\phi^{n+\frac12}-i\phi^{n+\frac12}\Delta_{n} W\Big]\bar{e}^{n+1}\Delta_{n} Wdx\\
      &=-\frac{\tau}{2}\Re\int_{\mathcal{O}}\big(\theta\Delta\phi^{n+1}+(1-\theta)\Delta\phi^{n}\big)\bar{e}^{n+1}\Delta_{n} Wdx+\frac{\tau}{4}\Re\int_{\mathcal{O}}(|\phi^{n+1}|^{2}
      +|\phi^{n}|^{2})\phi^{n+\frac12}\bar{e}^{n+1}\Delta_{n} Wdx\\
      &\quad-\frac12\Re\int_{t_{n}}^{t_{n+1}}\int_{\mathcal{O}}\Big(\psi(s)-\frac12(\psi(t_{n+1})+\psi(t_{n}))\Big)F_{\mathbf Q}\bar{e}^{n+1}dxds\\
      &\quad-\frac{\tau}{2}\Re\int_{\mathcal{O}}e^{n+\frac12}F_{\mathbf Q}\bar{e}^{n+1}dx
      +\frac12\Re\int_{\mathcal{O}}\phi^{n+\frac12}\bar{e}^{n+1}\big((\Delta_{n} W)^{2}-F_{\mathbf Q}\tau\big)dx\\
      &=:V^{1}+V^{2}+V^{3}+V^{4}+V^{5}.
    \end{align*}
    For term $V^{1}$, by the identity $\theta \Delta\phi^{n+1}+(1-\theta)\Delta\phi^{n}=\theta\Delta(\phi^{n+1}-\phi^{n})+\Delta\phi^{n}$ and Young's inequality we have
    \begin{align*}
      V^{1}&\leq K\tau\Big(\|\Delta\phi^{n+1}\|_{{\mathbb L}^2}+\|\Delta\phi^{n}\|_{{\mathbb L}^2}\Big)\|e^{n+1}\|_{{\mathbb L}^2}\|\Delta_n W\|_{{\mathbb L}^{\infty}}\\
      &\leq K\tau\|e^{n+1}\|_{{\mathbb L}^2}^2+K\tau\Big(\|\Delta\phi^{n+1}\|_{{\mathbb L}^2}^2\|\Delta_n W\|_{{\mathbb L}^{\infty}}^2+\|\Delta\phi^{n}\|_{{\mathbb L}^2}^2\|\Delta_n W\|_{{\mathbb L}^{\infty}}^2\Big)\\
      &\leq  K\tau\|e^{n+1}\|_{{\mathbb L}^2}^2+K\tau^2\|\Delta\phi^{n+1}\|_{{\mathbb L}^2}^4+K\|\Delta_n W\|_{{\mathbb L}^{\infty}}^4+K\tau \|\Delta\phi^{n}\|_{{\mathbb L}^2}^2\|\Delta_n W\|_{{\mathbb L}^{\infty}}^2.
    \end{align*}
    The estimation of $V^{2}$ is similar as that of $V^{1}$ and we have
    \begin{align*}
      V^{2}\leq K\tau\|e^{n+1}\|_{{\mathbb L}^2}^{2}+K\tau^{2}\Big(\| \phi^{n+1}\|_{{\mathbb L}^6}^{12}+\| \phi^{n}\|_{{\mathbb L}^6}^{12}\Big)+K\|\Delta_{n} W\|_{{\mathbb L}^{\infty}}^{4}.
    \end{align*}
    For term $V^{3}$, we have
    \begin{align*}
      V^{3}\leq K\int_{t_{n}}^{t_{n+1}}\Big(\|\psi(s)-\psi(t_{n})\|_{{\mathbb L}^2}^{2}+\|\psi(s)-\psi(t_{n+1})\|_{{\mathbb L}^2}^{2}\Big)ds+K\tau\|e^{n+1}\|_{{\mathbb L}^2}^{2}.
    \end{align*}
    For term $V^{4}$, we have
    \begin{align*}
      V^{4}\leq K\tau\Big(\|e^{n}\|_{{\mathbb L}^2}^{2}+\|e^{n+1}\|_{{\mathbb L}^2}^{2}\Big).
    \end{align*}
    For term $V^{5}$, we have
    \begin{align*}
      V^{5}&=\frac12\Re\int_{\mathcal{O}}\phi^{n+\frac12}\bar{e}^{n+1}\big((\Delta_{n} W)^{2}-F_{\mathbf Q}\tau\big)dx\\
      &=\frac12\Re\int_{\mathcal{O}}(\phi^{n+\frac12}\bar{e}^{n+1}-\phi^{n}\bar{e}^{n})\big((\Delta_{n} W)^{2}-F_{\mathbf Q}\tau\big)dx
      +\frac12\Re\int_{\mathcal{O}}\phi^{n}\bar{e}^{n}\big((\Delta_{n} W)^{2}-F_{\mathbf Q}\tau\big)dx\\
      &=:V_{a}^{5}+V_{b}^{5},
    \end{align*}
    where
    \begin{align*}
      V_{a}^{5}&=\frac12\Re\int_{\mathcal{O}}\phi^{n+\frac12}(\bar{e}^{n+1}-\bar{e}^{n})\big((\Delta_{n} W)^{2}-F_{\mathbf Q}\tau\big)dx
      +\frac14\Re\int_{\mathcal{O}}(\phi^{n+1}-\phi^{n})\bar{e}^{n}\big((\Delta_{n} W)^{2}-F_{\mathbf Q}\tau\big)dx\\
      &\leq \frac{1}{16}\|e^{n+1}-e^{n}\|_{{\mathbb L}^2}^{2}+K\tau^{2}\|\phi^{n+\frac12}\|_{{\mathbb L}^2}^{4}+K\frac{1}{\tau^2}\|(\Delta_{n} W)^{2}-F_{\mathbf Q}\tau\|_{{\mathbb L}^{\infty}}^{4}\\
      &\quad+K\frac{1}{\tau}\|e^{n}\|_{{\mathbb L}^2}^{2}\|(\Delta_{n} W)^{2}-F_{\mathbf Q}\tau\|_{{\mathbb L}^{\infty}}^{2}+K\tau\|\phi^{n+1}-\phi^{n}\|_{{\mathbb L}^2}^{2}.
    \end{align*}
    Combining all estimations above, we have
    \begin{align}\label{error inequality}
      \|e^{n+1}\|_{{\mathbb L}^2}^{2}-\|e^{n}\|_{{\mathbb L}^2}^{2}+\|e^{n+1}-e^{n}\|_{{\mathbb L}^2}^{2}\leq \mathcal{G}^{n}+\mathcal{M}^{n},
    \end{align}
    where
    \begin{align*}
      \mathcal{G}^{n}&:=K\int_{t_{n}}^{t_{n+1}}\int_{s}^{t_{n+1}}\|\Delta\psi(\rho)\|_{{\mathbb L}^2}^{2}d\rho ds+K\tau^2\|\Delta\psi(t_{n+1})\|_{{\mathbb L}^2}^{2}+K\tau^{2}\|\Delta\phi^{n}\|_{{\mathbb L}^2}^{2}\\
      &\quad+K\int_{t_{n}}^{t_{n+1}}\int_{s}^{t_{n+1}}
      (\|\psi(\rho)\|_{{\mathbb L}^6}^{6}+\|\psi(\rho)\|_{{\mathbb L}^2}^{2})d\rho ds+K\tau\|\nabla e^{n+1}-\nabla e^{n}\|_{{\mathbb L}^2}^{2}\\
      &\quad+K\int_{t_n}^{t_{n+1}}\left\|\int_{t_{n+1}}^s \psi(\nu)dW(\nu)\right\|_{{\mathbb {\mathbb H}}^2}^{2}ds
      +K\tau\|e^{n+1}\|_{{\mathbb L}^2}^{2}+K\tau\int_{t_{n}}^{t_{n+1}}\|\psi(s)\|_{{\mathbb H}^{1}}^{8}ds+K\tau^{2}\|\psi(t_{n})\|_{{\mathbb H}^{1}}^{8}\\
      &\quad +K\frac{1}{\tau}\int_{t_{n}}^{t_{n+1}}\|\psi(s)-\psi(t_{n})\|_{{\mathbb L}^2}^{4}ds+K\tau^{2}\|\phi^{n+1}\|_{{\mathbb H}^{1}}^{8}
      +K\tau^{2}\|\phi^{n}\|_{{\mathbb H}^{1}}^{8}+K\|\phi^{n+1}-\phi^{n}\|_{{\mathbb L}^{2}}^{4}\\
      &\quad+K\tau^{2}\|\phi^{n+1}\|_{{\mathbb L}^6}^{6}+K\|\Delta_{n} W\|_{{\mathbb L}^{\infty}}^{4}+K\tau\|\phi^{n+1}-\phi^{n}\|_{{\mathbb L}^2}^{2}
      +K\int_{t_{n}}^{t_{n+1}}\|\psi(s)-\psi(t_{n})\|_{{\mathbb L}^2}^{2}ds\\
      &\quad+K\int_{t_{n}}^{t_{n+1}}\|\psi(s)-\psi(t_{n+1})\|_{{\mathbb L}^2}^{2}ds+K\tau^2\|\Delta\phi^{n+1}\|_{{\mathbb L}^2}^{2}+K\tau\|\Delta\phi^{n}\|_{{\mathbb L}^2}^{2}\|\Delta_{n} W\|_{{\mathbb L}^{\infty}}^{2}\\
      &\quad+K\left\|\int_{t_{n}}^{t_{n+1}}(\psi(s)-\psi(t_{n}))dW(s)\right\|_{{\mathbb L}^2}^{2}+K\|e^{n}\|_{{\mathbb L}^2}^{2}\|\Delta_{n} W\|_{{\mathbb L}^{\infty}}^{2}
      +K\frac{1}{\tau^{2}}\|(\Delta_{n} W)^{2}-F_{\mathbf Q}\tau\|_{{\mathbb L}^{\infty}}^{4}\\
      &\quad+K\tau\big(\|\psi(t_{n})\|_{{\mathbb H}^1}^{2}+\|\phi^{n}\|_{{\mathbb H}^1}^{2}\big)\|e^{n}\|_{{\mathbb L}^2}^{2}
      +K\tau\big(\|\psi(t_{n})\|_{{\mathbb H}^1}^{2}+\|\phi^{n}\|_{{\mathbb H}^1}^{2}\big)\|e^{n+1}\|_{{\mathbb L}^2}^{2}
    \end{align*}
    and
    \begin{align*}
      \mathcal{M}^{n}:=\Im\int_{t_{n}}^{t_{n+1}}\int_{\mathcal{O}}(\psi(s)-\psi(t_{n}))\bar{e}^{n}dxdW(s)
      +\frac12\Re\int_{\mathcal{O}}\phi^{n}\bar{e}^{n}\big((\Delta_{n} W)^{2}-F_{\mathbf Q}\tau\big)dx.
    \end{align*}
    Now consider the error inequality \eqref{error inequality} for some $0\leq \ell\leq M$, multiply it by $\bm{1}_{\tilde{\Omega}_{\kappa,\ell}}$,
    sum over the index from $\ell=0$ to $n$, take the maximum between $0$ and $m\leq M$, and then take the expectation. The choice of this indicator function is necessary such that the term corresponding to the stochastic integral ${\mathcal M}^{\ell}$ is a martingale, which allows the use of the Burkholder-Davis-Gundy inequality. So we obtain correspondingly for the first term on the left-hand side of \eqref{error inequality}
    \begin{align*}
      &E\Big[\max_{0\leq n\leq m}\sum_{\ell=0}^{n}\bm{1}_{\tilde{\Omega}_{\kappa,\ell}}\big(\|e^{\ell+1}\|_{{\mathbb L}^2}^{2}-\|e^{\ell}\|_{{\mathbb L}^2}^{2}\big)\Big]\\
      &=E\Big[\max_{0\leq n\leq m}\Big(\bm{1}_{\tilde{\Omega}_{\kappa,n}}\|e^{n+1}\|_{{\mathbb L}^2}^{2}-\bm{1}_{\tilde{\Omega}_{\kappa,0}}\|e^{0}\|_{{\mathbb L}^2}^{2}
      +\sum_{\ell=1}^{n}(\bm{1}_{\tilde{\Omega}_{\kappa,\ell-1}}-\bm{1}_{\tilde{\Omega}_{\kappa,\ell}})\|e^{\ell}\|_{{\mathbb L}^2}^{2}\Big)\Big]\\
      &\geq E\Big[\max_{0\leq n\leq m}\bm{1}_{\tilde{\Omega}_{\kappa,n}}\|e^{n+1}\|_{{\mathbb L}^2}^{2}\Big],
    \end{align*}
    where we use the fact that the sum in the second line is positive because {$\tilde{\Omega}_{\kappa,\ell-1}\supset\tilde{\Omega}_{\kappa,\ell}$},
    and that $e^{0}=0$ $P$-a.s. The next terms to be considered are those corresponding to $\mathcal{G}^{\ell}$. Under the conclusions of Lemma \ref{spatial H2 new}, Lemma \ref{temporal L2}, Lemma \ref{temporal H1 new}, Lemma \ref{alg2:H1} and Lemma \ref{alg2:H2},
    one knows that
    \begin{align*}
      E\Big[\max_{0\leq n\leq m}\sum_{\ell=0}^{n}\bm{1}_{\tilde{\Omega}_{\kappa,\ell}}\mathcal{G}^{\ell}\Big]
      =\sum_{\ell=0}^{m}E\Big[\bm{1}_{\tilde{\Omega}_{\kappa,\ell}}\mathcal{G}^{\ell}\Big]
      \leq K\tau+K\tau(1+\kappa)\sum_{\ell=0}^{m}E\Big(\bm{1}_{\tilde{\Omega}_{\kappa,\ell}}\|e^{\ell+1}\|_{{\mathbb L}^2}^{2}\Big).
    \end{align*}
    In particular here, by Burkholder-Davis-Gundy inequality,
    \begin{align*}
      &E\Big[\max_{0\leq n\leq m}\sum_{\ell=0}^{n}\bm{1}_{\tilde{\Omega}_{\kappa,\ell}}\Im\int_{t_{\ell}}^{t_{\ell+1}}\int_{\mathcal{O}}
      (\psi(s)-\psi(t_{\ell}))\bar{e}^{\ell}dxdW(s)\Big]\\
      &\leq KE\Big[\Big(\sum_{\ell=0}^{m}\bm{1}_{\tilde{\Omega}_{\kappa,\ell}}
      \int_{t_{\ell}}^{t_{\ell+1}}\|\psi(s)-\psi(t_{\ell})\|_{{\mathbb L}^2}^{2}\|F_{\mathbf Q}\|_{{\mathbb L}^{\infty}}^{2}\|e^{\ell}\|_{{\mathbb L}^2}^{2}ds\Big)^{\frac12}\Big]\\
      &\leq KE\Big[\Big(\max_{0\leq n\leq m}\bm{1}_{\tilde{\Omega}_{\kappa,n}}\|e^{n}\|_{{\mathbb L}^2}\Big)\Big(
      \sum_{\ell=0}^{m}\bm{1}_{\tilde{\Omega}_{\kappa,\ell}}\int_{t_{\ell}}^{t_{\ell+1}}\|\psi(s)-\psi(t_{n})\|_{{\mathbb L}^2}^{2}ds\Big)^{\frac12}\Big]\\
      &\leq \frac14 E\Big[\max_{0\leq n\leq m}\bm{1}_{\tilde{\Omega}_{\kappa,n}}\|e^{n+1}\|_{{\mathbb L}^2}^{2}\Big]+K\tau.
    \end{align*}
    For the second term, one needs to prove the martingale property first, which is equivalent to proving
    \begin{align*}
      E\Big[\bm{1}_{\tilde{\Omega}_{\kappa,\ell}}\Re\int_{\mathcal{O}}\phi^{\ell}\bar{e}^{\ell}\Delta_{\ell}\tilde{W}dx\Big|\mathcal{F}_{t_{j}}\Big]=0,
    \end{align*}
    where $\Delta_{\ell}\tilde{W}=(\Delta_{\ell} W)^{2}-F_{\mathbf Q}\tau$ for $j\leq \ell\leq n$.
    In fact, we have
    \begin{align*}
      E\Big[\bm{1}_{\tilde{\Omega}_{\kappa,\ell}}\Re\int_{\mathcal{O}}\phi^{\ell}\bar{e}^{\ell}\Delta_{\ell}\tilde{W}dx\Big|\mathcal{F}_{t_{j}}\Big]
     & =E\Big[E\Big(\bm{1}_{\tilde{\Omega}_{\kappa,\ell}}\Re\int_{\mathcal{O}}\phi^{\ell}\bar{e}^{\ell}\Delta_{\ell}\tilde{W}dx\Big|\mathcal{F}_{t_{\ell}}\Big)\Big|\mathcal{F}_{t_{j}}\Big]\\
      &=E\Big[\Re\int_{\mathcal{O}}\bm{1}_{\tilde{\Omega}_{\kappa,\ell}}\phi^{\ell}\bar{e}^{\ell}E\big(\Delta_{\ell}\tilde{W}\big|\mathcal{F}_{t_{\ell}}\big)dx\Big|\mathcal{F}_{t_{j}}\Big]
      =0,
    \end{align*}
    the last line holds since
    \begin{align*}
      E\big[\Delta_{\ell}\tilde{W}\big|\mathcal{F}_{t_{\ell}}\big]=E\big[(\Delta_{\ell} W)^{2}\big|\mathcal{F}_{t_{\ell}}\big]-F_{\mathbf Q}\tau=0.
    \end{align*}
    Similar to before, we may estimate by  Burkholder-Davis-Gundy inequality
    \begin{align*}
      E\Big[\max_{0\leq n\leq m}\sum_{\ell=0}^{n}\bm{1}_{\tilde{\Omega}_{\kappa,\ell}}\Re\int_{\mathcal{O}}
      \phi^{\ell}\bar{e}^{\ell}\Delta_{\ell}\tilde{W}dx\Big]
      \leq \frac14 E\Big(\max_{0\leq n\leq m}\bm{1}_{\tilde{\Omega}_{\kappa,n}}\|e^{n+1}\|_{{\mathbb L}^2}^{2}\Big)+K\tau.
    \end{align*}
    Combining these estimates together, we have
    \begin{align*}
      \frac12 E\Big[\max_{0\leq n\leq m}\bm{1}_{\tilde{\Omega}_{\kappa,n}}\|e^{n+1}\|_{{\mathbb L}^2}^{2}\Big]\leq
      K\tau+K\tau(1+\kappa)\sum_{\ell=0}^{m}E\Big(\bm{1}_{\tilde{\Omega}_{\kappa,\ell}}\|e^{\ell+1}\|_{{\mathbb L}^2}^{2}\Big).
    \end{align*}
    The discrete Gronwall's lemma then leads to
    \[ E\Big[\max_{0\leq n\leq m}\bm{1}_{\tilde{\Omega}_{\kappa,n}}\|e^{n+1}\|_{{\mathbb L}^2}^{2}\Big]\leq Ke^{Kt_{m}\kappa}\tau.\]
    Using the nestedness of property $\tilde{\Omega}_{\kappa,m}\subset\tilde{\Omega}_{\kappa,n}$
    for all $0 \leq n \leq m$ one obtains
    \[E\Big(\bm{1}_{\tilde{\Omega}_{\kappa,m}}\max_{0\leq n\leq m}\|e^{n+1}\|_{{\mathbb L}^2}^{2}\Big)\leq E\Big[\max_{0\leq n\leq m}\bm{1}_{\tilde{\Omega}_{\kappa,n}}\|e^{n+1}\|_{{\mathbb L}^2}^{2}\Big]\leq Ke^{Kt_{m}\kappa}\tau.\]
    The proof is completed by letting $m=M$.
    \end{proof}

   {
   The following remark discusses uniqueness of solutions of \eqref{algorithm 2} on `large' subsets of $\Omega$, on subsets of which the error estimate in Theorem~\ref{thm1} is applied.
   \begin{remark}\label{uni1} Let $T \equiv t_M >0$, and $\{\phi^{n}_j;\, 0\leq n\leq M\}$, $j=1,2$ be two solutions of \eqref{algorithm 2},
   and denote $\xi^n := \phi^n_1 - \phi^n_2$, as well as $\xi^{n+\theta} := \theta \xi^{n+1} + (1-\theta) \xi^n$ to obtain
   $$ i \int_{{\mathscr O}} (\xi^{n+\theta} - \xi^n)z\, dx - \tau \theta \int_{{\mathscr O}} \nabla \xi^{n+\theta} \nabla z\, dx - \frac{\tau{\theta}}{2} \int_{{\mathscr O}} {\mathcal F}(\phi^n_j, \phi^{n+1}_j) z\, dx = \theta \int_{{\mathscr O}} \phi^{n+1/2} \Delta_n W z\, dx \qquad \forall z \in {\mathbb H}^1_0,$$
   where
   $$ {\mathcal F}(\phi^n_j, \phi^{n+1}_j) := \bigl( \vert \phi_1^{n+1}\vert^2 + \vert\phi_1^n\vert^2\bigr) \phi_1^{n+1/2} - \bigl( \vert \phi_2^{n+1}\vert^2 + \vert\phi_2^n\vert^2\bigr) \phi_1^{n+1/2}$$
   Then put $z = \bar{\xi}^{n+\theta}$, and take imaginary parts; by arguments
   which are similar to those in the proof of Theorem~\ref{thm1}, and using the algebraic identities
   {
   \begin{equation}\label{waa0} \xi^{n+1/2} =
   \frac{1}{2 \theta} \bigl( \xi^{n+\theta} + [2\theta-1] \xi^{n}\bigr) \qquad \mbox{\rm resp.} \qquad
   \xi^{n+1}=\frac{1}{\theta}\big(\xi^{n+\theta}-\xi^{n}\big)+\xi^{n} \,,
   \end{equation}
   }
   we arrive at
  \begin{equation}\label{waa1}
  \frac12 \Bigl(\|\xi^{n+\theta}\|_{{\mathbb L}^2}^{2}-\|\xi^{n}\|_{{\mathbb L}^2}^{2}+\|\xi^{n+\theta}-\xi^{n}\|_{{\mathbb L}^2}^{2} \Bigr)
  = I+ \theta \Im \int_{{\mathscr O}}
   \xi^{n+1/2} \bar{\xi}^{n+\theta}  \Delta_n W  \,dx,
  \end{equation}
   where
   {
   \begin{equation*}
     \begin{split}
       I&=\frac{\tau\theta}{2} \Im \int_{{\mathscr O}} \bigl( \vert  \phi^{n+1}_1\vert^2 + \vert \phi_1^n\vert^2 \bigr) \xi^{n+1/2} \bar{\xi}^{n+\theta}\, dx
       +\frac{\tau\theta}{2}\Im\int_{{\mathscr O}} \bigl(|\phi^{n+1}_{1}|^2-|\phi_{2}^{n+1}|^2+|\phi^{n}_{1}|^2-|\phi_{2}^{n}|^2\bigr)\phi_{2}^{n+\frac12}\bar{\xi}^{n+\theta}\,dx\\
       &=:I_a+I_b.
     \end{split}
   \end{equation*}
   }
   We use (\ref{waa0}) to compute for the last term in (\ref{waa1}) that
   \begin{equation}\label{waa2}
   \begin{split}
   \theta \Im \int_{{\mathscr O}}
   \xi^{n+1/2} \bar{\xi}^{n+\theta}  \Delta_n W  \,dx &= \frac{2\theta-1}{2} \Im \int_{{\mathscr O}} \xi^n \bar{\xi}^{n+\theta} \Delta_n W\, dx
    = \frac{2\theta-1}{2} \Im \int_{{\mathscr O}} \xi^n \bigl(\bar{\xi}^{n+\theta} - \bar{\xi}^n\bigr) \Delta_n W\, dx\\
   &\leq \frac{1}{4} \Vert \bar{\xi}^{n+\theta} - \bar{\xi}^n\Vert^2_{{\mathbb L}^2} +
   \Bigl(\frac{2\theta-1}{2}\Bigr)^2 \Vert \xi^n \Delta_n W\Vert^2_{{\mathbb L}^2}.
   \end{split}
   \end{equation}
   For the first term $I_a$ we have by (\ref{waa0})
   {
      \begin{equation*}
   \begin{split}
   \frac{\tau\theta}{2} \Im \int_{{\mathscr O}} \big[\vert \phi_1^{n+1} \vert^2 + \vert \phi_1^n\vert^2\big] \xi^{n+1/2}\bar{\xi}^{n+\theta}\,dx
   &= \frac{\tau}{2} \Im \int_{{\mathscr O}} \big[\vert \phi_1^{n+1} \vert^2 + \vert \phi_1^n\vert^2\big] \left(\frac{2\theta-1}{2\theta}\xi^n\right)   \bar{\xi}^{n+\theta}\,dx \\
   &= \frac{\tau(2\theta-1)}{4} \Im \int_{{\mathscr O}}  \big[\vert \phi^{n+1}_1\vert^2 + \vert \phi^n_1\vert^2\big] \xi^{n}(\bar\xi^{n+\theta} - \bar\xi^n)\,dx.
  \end{split}
   \end{equation*}
   }
   In order to estimate the term $I_b$ we use again (\ref{waa0}) to calculate for the relevant term
   {
   \begin{eqnarray*}
   &&\vert \phi_1^{n+1}\vert^2 - \vert \phi_2^{n+1}\vert^2 =  \xi^{n+1} \bar{\phi}_1^{n+1} -
 \phi^{n+1}_2\bar{\xi}^{n+1}  \\
   &&\quad = \frac{1}{\theta}(\xi^{n+\theta}-\xi^{n})\bar{\phi}_{1}^{n+1}+\xi^{n} \bar{\phi}_{1}^{n+1}
   -\frac{1}{\theta}(\bar\xi^{n+\theta}-\xi^{n})\phi_{2}^{n+1}-\bar\xi^{n}\phi_{2}^{n+1}.
   \end{eqnarray*}
   }
   We may then use ${\mathbb H}^1({\mathscr O}) \hookrightarrow {\mathbb L}^\infty({\mathscr O})$ to estimate
   \begin{align*} I \leq&  K\tau^2 \max_{1 \leq j\leq 2}\bigl(\Vert \phi^{n+1}_j\Vert^4_{{\mathbb H}^1} + \Vert \phi^{n}_j\Vert^4_{{\mathbb H}^1}\bigr)\Vert \xi^n\Vert^2_{{\mathbb L}^2} + \frac{1}{4} \Vert \xi^{n+\theta} - \xi^{n}\Vert^2_{{\mathbb L}^2} \\
   &+ K\tau \max_{1 \leq j\leq 2}\bigl(\Vert \phi^{n+1}_j\Vert^2_{{\mathbb H}^1} + \Vert \phi^{n}_j\Vert^2_{{\mathbb H}^1}\bigr)\Vert \xi^{n+\theta} -  \xi^n\Vert^2_{{\mathbb L}^2}.
   \end{align*}
   Now multiply (\ref{waa1}) with $\bm{1}_{\widehat{\Omega}_{\kappa,n+1}}$, where
   \[\widehat{\Omega}_{\kappa,n+1}=\Big\{\omega\in\Omega\Big| \max_{0\leq \ell \leq n+1}\|\phi^{l}\|_{\mathbb{H}^{1}}^{2}\leq \kappa\Big\}
  \supset \tilde{\Omega}_{\kappa,n+1}.\]
  Note that again $\widehat{\Omega}_{\kappa,n+1} \subset \widehat{\Omega}_{\kappa,n}$. We then obtain
  from the above considerations, for $\kappa \leq  \tau^{-\alpha}$ ($\alpha < 1$) and $\tau \leq \tau^*$ sufficiently small  the estimate
  \begin{eqnarray}
  \bm{1}_{\widehat{\Omega}_{\kappa,n+1}} \|\xi^{n+\theta}\|_{{\mathbb L}^2}^{2}
  &\leq& \bm{1}_{\widehat{\Omega}_{\kappa,n}}  \|\xi^{n}\|_{{\mathbb L}^2}^{2}  \Bigl(1 +
    \Bigl(\frac{2\theta-1}{2}\Bigr)^2 \Vert \Delta_n W\Vert^2_{{\mathbb H}^1}  + K\tau^{2(1-\alpha)} \Bigr).
  \end{eqnarray}
  We may now proceed by induction: for $n=0$ we have $\xi^0 = 0$ $P$-a.s.~on $\widehat{\Omega}_{\kappa,0}$, in particular. Therefore, we may deduce $\theta^2 \bm{1}_{\widehat{\Omega}_{\kappa,1}} \|\xi^{1}\|_{{\mathbb L}^2}^{2} = 0$, and hence $\xi^{1} = 0$ on $\widehat{\Omega}_{\kappa,1}$. Correspondingly, we find
  $$ \bm{1}_{\widehat{\Omega}_{\kappa,n+1}}\Vert \xi^{n+1}\Vert^2_{{\mathbb L}^2} = 0 \quad P\mbox{-a.s.}
  \qquad (0 \leq n \leq M).$$
  For $\Omega_{\kappa} := \Omega_{\kappa,M}$, (\ref{Y}) implies $\lim_{\tau\rightarrow 0}P(\widehat{\Omega}_{\kappa})=1$
  for $\kappa \propto \log(\tau^{-\varepsilon})$ for every $0 <\varepsilon <1$, and thus we retrieve uniqueness of solutions for the limiting problem \eqref{sdd1} with $\lambda=-1$. --- In the practical studies performed in Section~\ref{num_exp} we had that {\em all} simulations are included in $\widehat{\Omega}_{\kappa}$ for some
  moderate $\kappa = O(1)$; see Figure 2, (c).
  \end{remark}}

\section{Numerical Experiments}\label{num_exp}
In the previous sections, we showed stability and convergence (${\mathscr O} \subset {\mathbb R}^d$), and convergence with local
rates (${\mathscr O} \subset {\mathbb R}^1$)  for the $\theta$-scheme (\ref{algorithm 2}) and the defocusing nonlinearity
($\lambda = -1$) in (\ref{sdd1}) with spatially regular noise.
The following example is chosen to computationally study stability and rates of convergence for different values $\theta_i \in \{ \frac{1}{2}, \frac{1}{2} + \sqrt{\tau}, 1 \}$ in the $\theta$-scheme (\ref{algorithm 2})  to solve the stochastic cubic Schr\"odinger equation ($\lambda=-1$) with colored in space noise.
%
%
%
In order to better clarify the interplay of nonlinearity and noise, we scale the
noise in (\ref{sdd1}) and (\ref{algorithm 2}) by a parameter $\nu \in {\mathbb R}$.

\medskip

\begin{example}\label{exa_1}
Let ${\mathscr O} = (-1,1)$, $T = \frac{1}{4}$, and $\psi_0({x}) = \sin^2(\pi x)$. For $1 \leq L \leq 8$, and $\{ \beta_{\ell};\,  1 \leq \ell \leq L\}$ a family of independent ${\mathbb R}$-valued Wiener processes, consider the real-valued Wiener process $W \equiv \{ W_t;\, t \geq 0\}$, $W(t) = \sum_{\ell=1}^L \frac{1}{\ell} \sin (\pi \ell x) \beta_\ell(t)$, and $\nu = \sqrt{2}$ in (\ref{sdd1}).
We use the $\theta$-scheme (\ref{algorithm 2}) with values $\theta_1 = \frac{1}{2}$, $\theta_2 =  \frac{1}{2} + \sqrt{\tau}$, and
$\theta_3 = 1$ for the numerical approximation. Let $I_{\tau} = \{ t_n;\, 0 \leq n \leq M\}$ be the uniform discretization of $[0,T]$ of size $\tau > 0$, and
${\mathscr T}_h$ be the uniform triangulation of ${\mathscr O}$ of size $h = \frac{1}{256}$,
on which the lowest-order ${\mathbb H}^1$-conforming finite element discretization of (\ref{algorithm 2}) is realized.
The reference values (for Figure~\ref{figure1} a) and b)) are generated for the smallest
mesh size $\widetilde{\tau} = 2^{-14}$. Newton's method is used, and $500$ realizations are chosen to approximate the expectations.
\end{example}
%
\begin{figure}[htbp]
  \centering
  \begin{minipage}[b]{5.2 cm}
   \begin{center} a) $\nu = 0$  \end{center}
   \centering
    \includegraphics[width=1.078\textwidth]{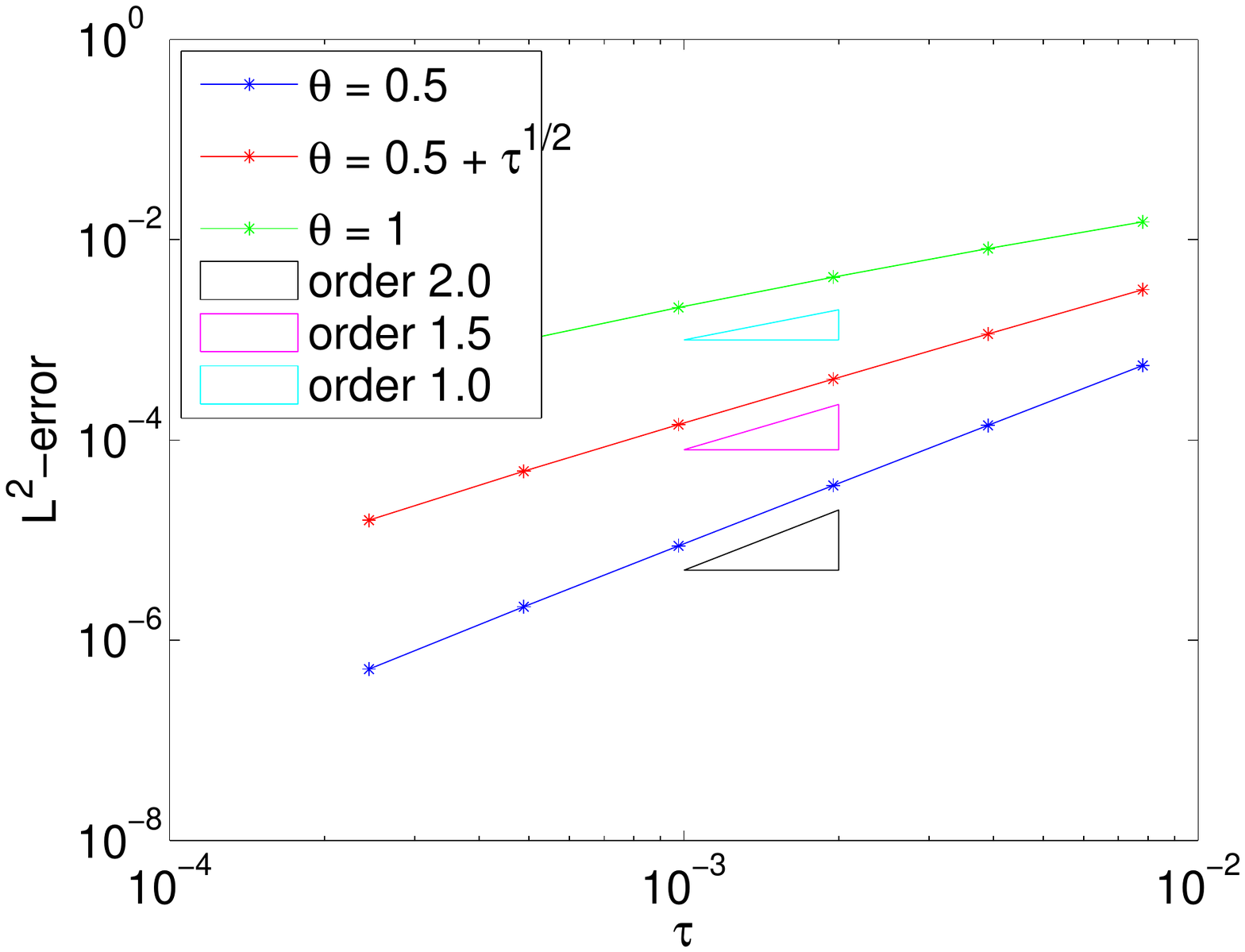}
    \end{minipage}
    \begin{minipage}[b]{5.2 cm}
     \begin{center} b) Trajectories at $x=0$   \end{center}
     \centering
    \includegraphics[width=1.078\textwidth]{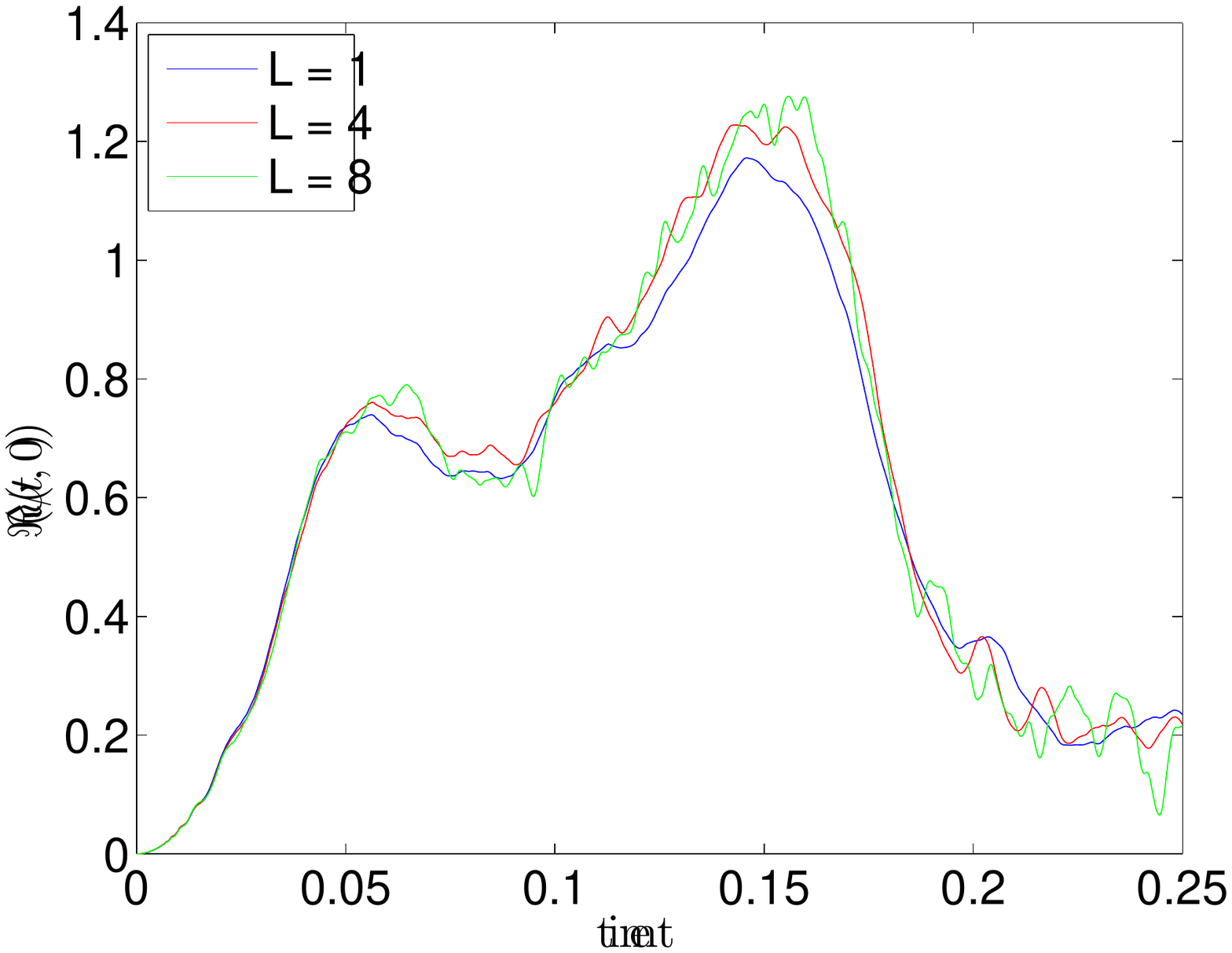}
  \end{minipage}
      \begin{minipage}[b]{5.2 cm}
     \begin{center} c) $\nu = \sqrt{2}$    \end{center}
     \centering
    \includegraphics[width=1.078\textwidth]{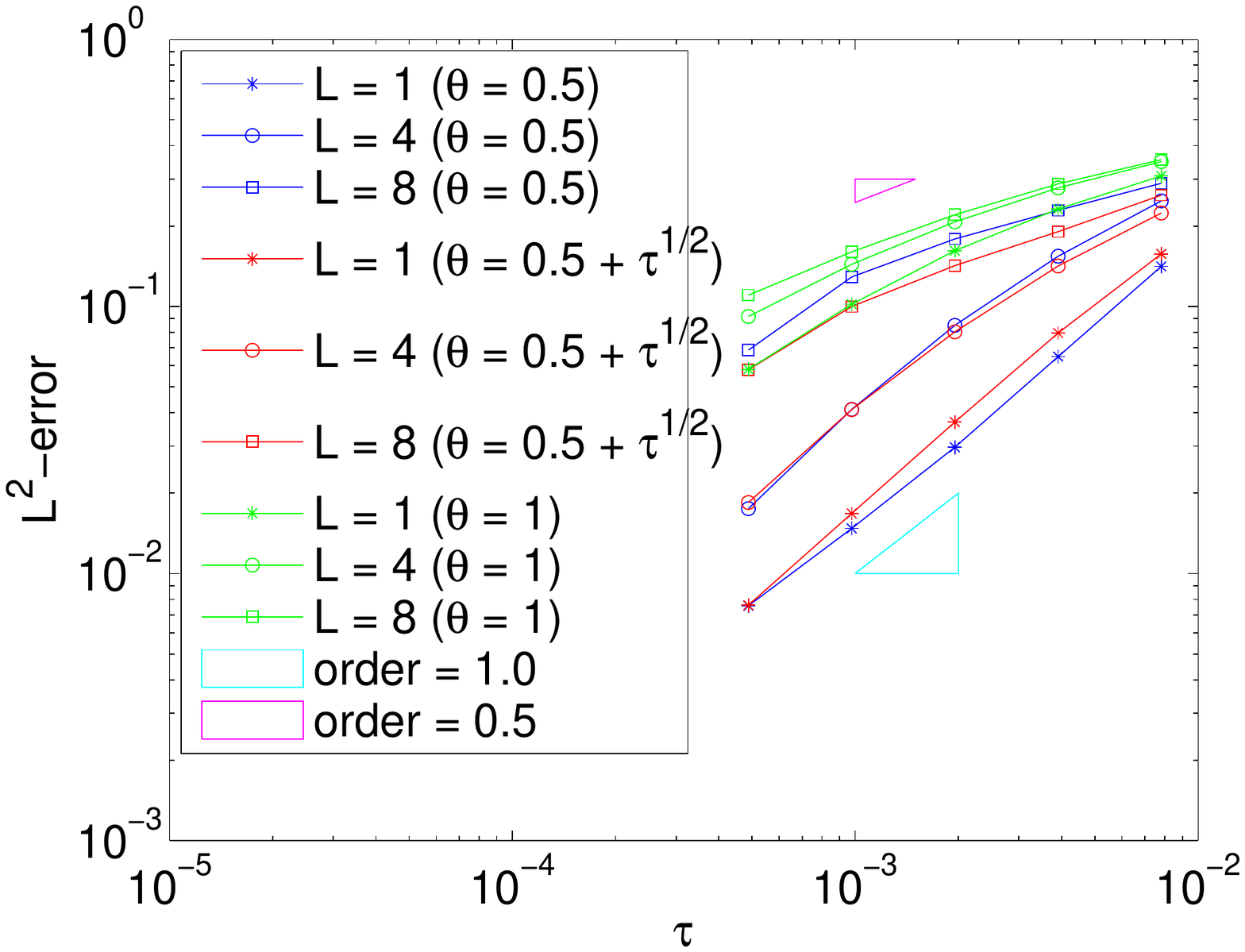}
  \end{minipage}
  \caption{ a) Rates of convergence for the deterministic case in the norm $\Vert \psi(T) - \phi^{{[\frac{T}{\tau}]}}\Vert_{{\mathbb L}^2}$ ($d=1$, $T = \frac{1}{4}$, $\nu = 0$, $h = \frac{1}{256}$, $\tau \in \{ 2^{-i};\, 7 \leq i \leq 11\}$).
  b) Trajectory at $x = 0$ for $L\in \{ 1,4,8\}$ ($\theta = \frac{1}{2} + \sqrt{2}$).
  c) Rates of convergence for the
  stochastic NLS driven by $W(t) = \sum_{\ell=1}^L \frac{1}{\ell} \sin (\pi \ell x) \beta_\ell(t)$ in the norm $\bigl(E[\Vert \psi(T) - \phi^{[\frac{T}{\tau}]}\Vert^2_{{\mathbb L}^2}]\bigr)^{1/2}$ ($d=1$, $T = \frac{1}{4}$, $\nu = \sqrt{2}$, $\tau \in \{ 2^{-i};\, 7 \leq i \leq 11\}$).}
   \label{figure1}
\end{figure}
We consider $\nu = 0$ first: Figure~\ref{figure1} a) shows order $2$ for the ${\mathbb L}^2$-error of the $\theta$-scheme for $\theta = \frac{1}{2}$; the order
drops to $1.5$ for $\theta = \frac{1}{2} + \sqrt{\tau}$, and to order $1$ for the implicit Euler scheme ($\theta = 1$). The observations are different in the stochastic case ($\nu = \sqrt{2}$) where different sorts of Wiener processes depending on $L$ are used: as is displayed in Figure~\ref{figure1} c),  the strong order of convergence for $\theta \in \{ \theta_1, \theta_2\}$ drops from approximately $1$ to $0.5$ for values $1$ to $8$ of $L$. The choice $\theta = \theta_3$ is exceptional since we obtain the approximate order $0.5$ for all values of $L$.  Figure~\ref{figure1} b) compares typical
trajectories for $L \in \{1,4,8\}$.
%
%
%
%
\begin{figure}[htbp]
  \centering
  \begin{minipage}[b]{5.2 cm}
   \begin{center} a) $t_n \mapsto E[\Vert \phi^n\Vert^2_{{\mathbb L}^2}]$  \end{center}
   \centering
    \includegraphics[width=1.078\textwidth]{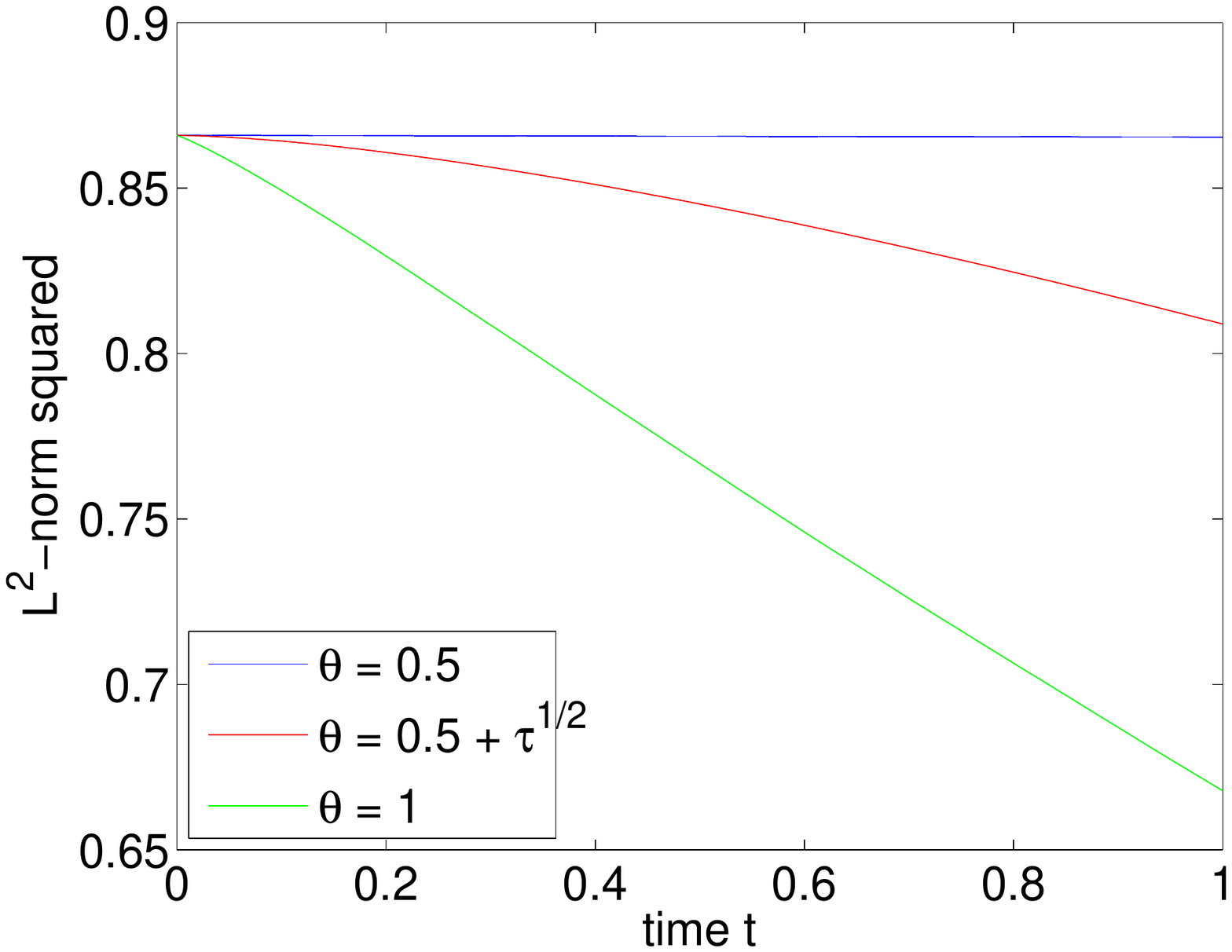}
    \end{minipage}
    \begin{minipage}[b]{5.2 cm}
     \begin{center} b) $t_n \mapsto E[\Vert \phi^n\Vert^2_{{\mathbb L}^2}]$   \end{center}
     \centering
    \includegraphics[width=1.078\textwidth]{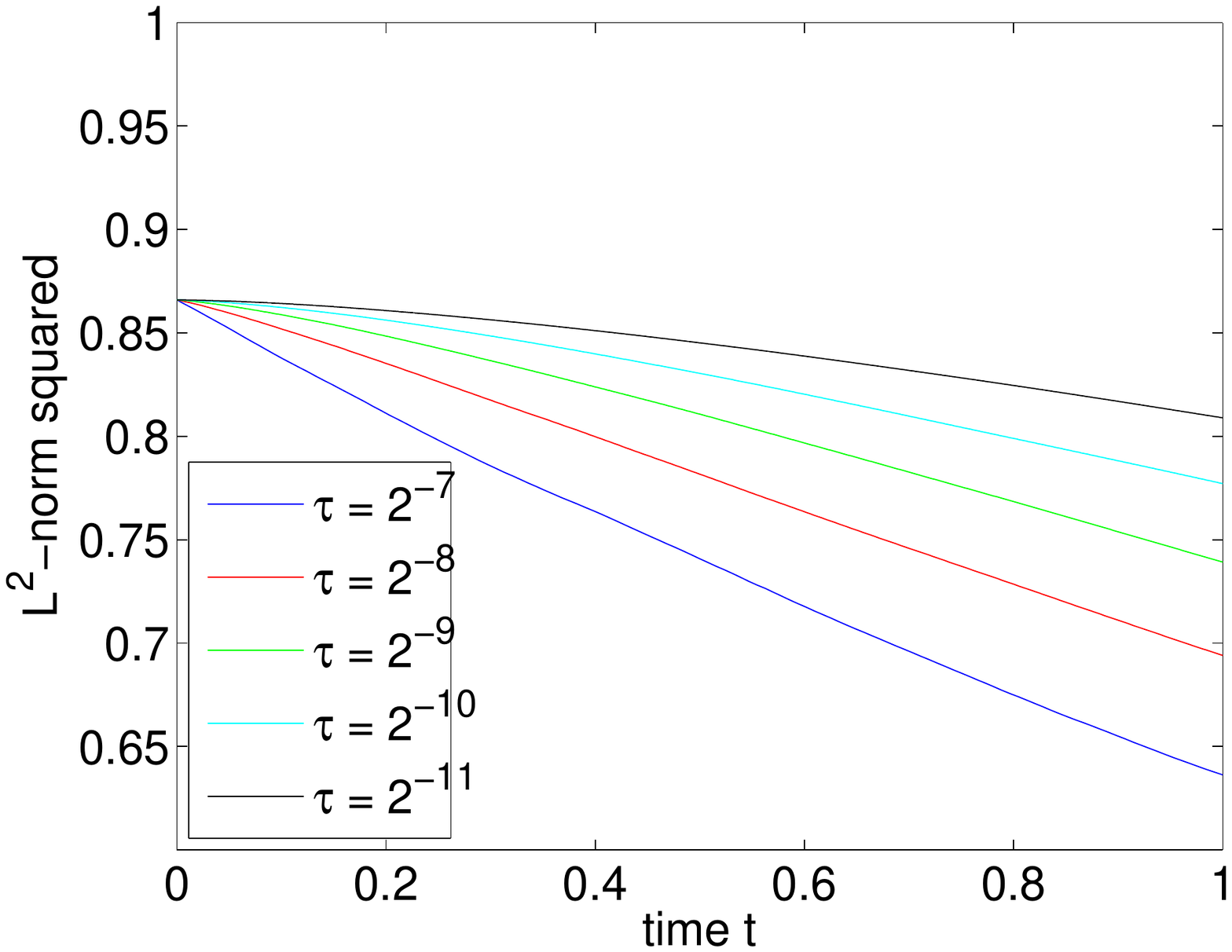}
  \end{minipage}
      \begin{minipage}[b]{5.2 cm}
     \begin{center} c) distribution    \end{center}
     \centering
    \includegraphics[width=1.078\textwidth]{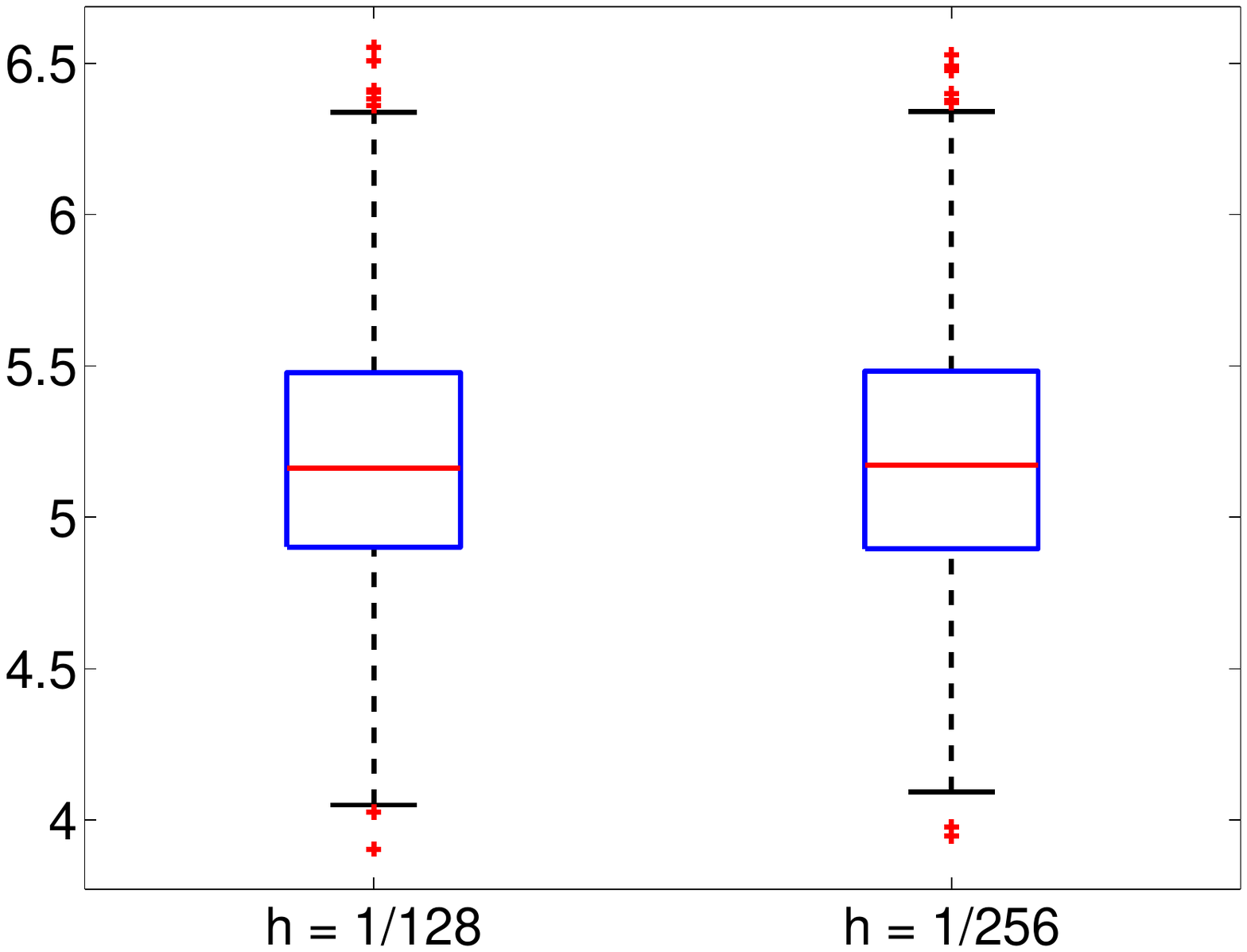}
  \end{minipage}
  \caption{   a) Squares of (averaged) ${\mathbb L}^2$-norm for $\theta \in \{\theta_1, \theta_2, \theta_3\}$
  ($d=1$, $T = \frac{1}{4}$, $\nu = \sqrt{2}$, $L=8$, $h = \frac{1}{256}$, $\tau = 2^{-8}$).
  b) Squares of the (averaged) ${\mathbb L}^2$-norm for $\theta = \frac{1}{2} + \sqrt{\tau}$ and different
  step sizes $\tau$ ($d=1$, $T = \frac{1}{4}$, $\nu = \sqrt{2}$,  $L=8$, $h = \frac{1}{256}$, $\tau \in \{ 2^{-i};\, 7 \leq i \leq 11\}$).
  c)
  Distribution of
 $\max_{0 \leq n \leq M} \Vert \nabla \phi^n\Vert_{{\mathbb L}^2}$ for $\theta = \frac{1}{2} + \sqrt{\tau}$, with median (5.1) and lower (4.8) and upper (5.5) quartile
  ($d=1$, $T = \frac{1}{4}$, $\nu = \sqrt{2}$, $L=8$, $\tau = 2^{-10}$).}
  \label{afigure2a}
\end{figure}

The box plot in Figure 2 c) complements this result: the set $\widehat{\Omega}_{\kappa}  :=
\bigl\{ \max_{0 \leq n \leq M} \Vert \nabla \phi^n \Vert_{{\mathbb L}^2} \leq \kappa\bigr\}$ is $\Omega$ for values of $\kappa$ exceeding approximately $6.5$. Figures 2 a), b) study the conservation of mass for the three schemes: we observe a mild decrease for
$\theta = \frac{1}{2} + \sqrt{\tau}$, which is far more pronounced for $\theta = 1$.
  \section*{Acknowledgment}

  This work was initiated when A.~P. visited the Academy of Mathematics and Systems Science, Chinese Academy of Sciences (Beijing) in March 2014. The authors are grateful to Christian Schellnegger (U T\"ubingen) for providing the computational results, and the careful reading of one referee to improve the quality of the paper. C.~C and J.~H are supported by National Natural Science Foundation of China (NO.~91130003, NO.~11021101 and NO.~11290142). The authors are grateful to the helpful suggestions of one anonymous referee.


\begin{thebibliography}{99}
\bibitem{BT} A.~Bensoussan and R.~Temam, {\em Equations stochastiques du type Navier-Stokes}, J.~Funct.~Anal.~{\bf 13}, pp.~195--222 (1973).
\bibitem{DBD0a} A.~De Bouard, A.~Debussche, {\em The stochastic nonlinear Schr\"odinger equation in $\mathbb{H}^1$},
Stochastic Anal.~Appl.~{\bf 21}, pp.~97--126 (2003).
\bibitem{DBD1} A.~De Bouard, A.~Debussche, {\em A semi-discrete scheme for the stochastic nonlinear Schr\"odinger equation}, Numer.~Math.~{\bf 96}, pp.~733--770 (2004).
\bibitem{DBD2} A.~De Bouard, A.~Debussche, {\em Weak and strong order of convergence of a semidiscrete scheme for the stochastic nonlinear Schr\"odinger equation}, Appl.~Math.~Optim.~{\bf 54}, pp.~369--399 (2006).
\bibitem{CP} E.~Carelli and A.~Prohl, {\em Rates of convergence for discretizations of the stochastic incompressible Navier-Stokes equations}, SIAM J.~Numer.~Anal.~{\bf 50}, pp.~2467--2496 (2012).
\bibitem{DDM1} A.~Debussche, L.~Di Menza, {\em Numerical simulation of focusing stochastic
nonlinear Schr\"odinger equations}, Physica D~{\bf 162}, pp.~131--154 (2002).
\bibitem{L1} J.~Liu, {\em Order of convergence of splitting schemes for both deterministic and stochastic nonlinear
Schr\"odinger equations}, SIAM J.~Numer.~Analysis~{\bf 51}, pp.~1911--1932 (2013).
\bibitem{PR} C.~Pr\'ev\^ot and M.~R\"{o}ckner.  A concise course on stochastic partial differential equations. Springer-Verlag Berlin Heidelberg 2007. MR2329435 (2009a:60069).
\bibitem{RGBC1} K.O.~Rasmussen, Y.B.~Gaididei, O.~Bang, P.L.~Christiansen, {\em The influence of noise
on critical collapse in the nonlinear Schr\"odinger equation}, Phys.~Letters~A~{\bf 204}, pp.~121--127 (1995).


\end{thebibliography}
\end{document}